\documentclass[11pt,reqno]{amsart}

\usepackage{booktabs}               
\usepackage[english]{babel}
\usepackage{latexsym}
\usepackage[utf8]{inputenc}
\usepackage{babel}
\usepackage{fancyhdr}
\usepackage{longtable}
\usepackage{calligra,mathrsfs}
\usepackage{geometry}
\usepackage{comment}
\usepackage{textcomp}
\usepackage{caption}
\usepackage{lmodern}
\usepackage{mathrsfs}
 \usepackage[T1]{fontenc}
\usepackage[all,cmtip]{xy}
\usepackage{epsfig}
\usepackage{amsmath,amsfonts,amssymb,amsthm}
\usepackage{thmtools}
\usepackage{graphics}
\usepackage{graphicx}
\usepackage{verbatim}
\usepackage{dsfont}
\usepackage{enumerate}  
\usepackage{pdflscape}
\usepackage{longtable}
\usepackage{booktabs}
\usepackage{colortbl}
\usepackage[shortlabels]{enumitem}
\usepackage[breaklinks]{hyperref} 
\hypersetup{
	colorlinks,
	linkcolor={red},
	citecolor={blue},
	urlcolor={red}
}
\usepackage{stmaryrd}
\usepackage{multirow}

\usepackage[disable]
{todonotes}

\newcommand{\evnrow}{\rowcolor[gray]{0.95}}

\newcommand{\C}{\mathbb{C}}

\newcommand{\QQ}{\mathbb{Q}}
\newcommand{\Q}{\mathbb{Q}}
\newcommand{\Z}{\mathbb{Z}}

\newcommand{\cM}{\mathcal{M}}

\newcommand{\mQ}{\mathcal{Q}}

\newcommand{\mR}{\mathcal{R}}
\newcommand{\mF}{\mathcal{F}}
\newcommand{\mG}{\mathcal{G}}
\newcommand{\sod}[1]{\langle #1 \rangle}
\newcommand{\PP}{\mathbb{P}}

\newcommand{\mE}{\mathcal{E}}
\newcommand{\mU}{\mathcal{U}}
\newcommand{\mV}{\mathcal{V}}
\newcommand{\mL}{\mathcal{L}}

\newcommand{\mA}{\mathcal{A}}
\newcommand{\mB}{\mathcal{B}}
\newcommand{\mC}{\mathcal{C}}
\newcommand{\mD}{\mathcal{D}}
\newcommand{\mM}{\mathcal{M}}
\newcommand{\mH}{\mathcal{H}}

\newcommand{\mP}{\mathcal{P}}

\newcommand{\mZ}{\mathscr{Z}}
\newcommand{\of}{\mathcal{O}}

\newcommand{\mK}{\mathcal{K}}

\newcommand{\W}{\bigwedge}

\newcommand{\bbGr}{\mathbb{G}\mathrm{r}}
\DeclareMathOperator{\Db}{{D^{\mathrm{b}}}}

\DeclareMathOperator{\Ker}{Ker}

\DeclareMathOperator{\Sym}{S}
\DeclareMathOperator{\ddim}{dim}
\DeclareMathOperator{\Gr}{Gr}
\DeclareMathOperator{\Fl}{Fl}

\DeclareMathOperator{\Hom}{Hom}
\DeclareMathOperator{\HOM}{\mathscr{H}\text{\kern -3pt {\calligra\large om}}\,}

\DeclareMathOperator{\Bl}{Bl}

\DeclareMathOperator{\dP}{dP}
\DeclareMathOperator{\Sing}{Sing}
\def\ra{\rightarrow}\def\lra{\longrightarrow}

\newcounter{fanos}
\newcommand{\fanoid}[1]{
[\#{#1}]
}

\newtheorem{thm}{Theorem}[section]

\newtheorem{question}[thm]{Question}

\newtheorem{lemma}[thm]{Lemma}
\newtheorem{proposition}[thm]{Proposition}

\newtheorem*{aim*}{Aim of this paper}

\theoremstyle{definition}
\newtheorem{workhyp}[thm]{Working Hypothesis}
\newtheorem{definition}[thm]{Definition}
\newtheorem{rmk}[thm]{Remark}
\newtheorem{example}[thm]{Example}

\usepackage{xcolor}
\colorlet{myred}{red!80!black}

\declaretheoremstyle[
spaceabove=1.5ex, spacebelow=1.5ex,
headfont=\bf,
notefont=\mdseries, notebraces={(}{)},
bodyfont=\normalfont,
headpunct=.,
numberwithin=,
postheadhook=\leavevmode%
qed={$\diamond$}
]{mystyle}
\declaretheorem[style=mystyle]{fano}

\renewcommand\thefano{C--\arabic{fano}}
\usepackage[capitalise]{cleveref}
\crefname{mystyle}{}{}
\crefname{claim}{Claim}{Claims}
\crefname{rmk}{Remark}{Remarks}
\crefname{workhyp}{WH}{WH}

\makeatletter    
\def\l@subsection{\@tocline{1}{0,2pt}{2pc}{8mm}{\ \ }} 
\makeatletter    
\def\l@section{\@tocline{1}{0,2pt}{2pc}{8mm}{\ \ }}

\usepackage{float}
\floatstyle{plaintop}
\restylefloat{table} 

\author{Marcello Bernardara}
\address{Institut de Math\'ematiques de Toulouse \\ UMR 5219 \\ Universit\'e 
de Toulouse \\ 
Universit\'e Paul Sabatier \\ %
118 route de Narbonne \\ %
31062 Toulouse Cedex 9\\ %
France}
\email{marcello.bernardara@math.univ-toulouse.fr}

\author{Enrico Fatighenti}
\address{Dipartimento di Matematica \\
Universit\`a di Bologna\\
Piazza di Porta San Donato 5\\
40127 Bologna, Italy}
\email[E.~Fatighenti]{enricofatighenti6@gmail.com}

\author{Laurent Manivel}
\address{Institut de Math\'ematiques de Toulouse \\ UMR 5219 \\ Universit\'e 
de Toulouse \\ CNRS, Universit\'e Paul Sabatier \\ %
118 route de Narbonne \\ %
31062 Toulouse Cedex 9\\ %
France}
\email{laurent.manivel@math.cnrs.fr}

\author{Fabio Tanturri}
\address{Dipartimento di Matematica \\
Universit\`a di Genova\\
Via Dodecaneso 35\\
16146 Genova, Italy}
\email{tanturri@dima.unige.it}

\@namedef{subjclassname@2020}{\textup{2020} Mathematics Subject Classification}
 
\title[Fano fourfolds of K3 type]{Fano fourfolds of K3 type}

\usepackage{libertine}
\usepackage{libertinust1math} 

\protected\def\ignorethis#1\endignorethis{}
\let\endignorethis\relax
\def\TOCstop{\addtocontents{toc}{\ignorethis}}
\def\TOCstart{\addtocontents{toc}{\endignorethis}}

\begin{document}

\begin{abstract}
We produce a list of $64$ families of Fano fourfolds of K3 type, extracted from our database of at least $634$ Fano fourfolds
constructed as zero loci of general global sections of completely reducible homogeneous vector bundles on products 
of flag manifolds. We study the geometry of these Fano fourfolds in some detail, and we find the origin of their K3 structure by relating most of them either to cubic fourfolds, Gushel--Mukai fourfolds, or actual K3 surfaces. Their main invariants and some information on their rationality and on possible semiorthogonal decompositions for their derived categories are provided.
\end{abstract}

\maketitle

{
\hypersetup{linkcolor=black}
\tableofcontents 
}

\section{Introduction}

In this paper we pursue the quest for Fano varieties of K3 type started in \cite{eg2}, as a  
chapter of the quest for Fano varieties of Calabi--Yau type initiated from \cite{ilievmanivel}.  
We refer to these papers for details on the motivations, both from the side of IHS manifolds 
and of Fano varieties. Our main objects of study here are Fano fourfolds that can be obtained 
as zero loci of general sections of fully decomposable equivariant vector bundles 
on products of flag
manifolds. 

With the help of computer calculations, we produced a huge database containing thousands of families of Fano fourfolds. We computed the main numerical invariants thereof, yielding at least $634$ distinct families. From this 
database we extracted those Fano fourfolds with Hodge number $h^{3,1}=1$, which is the main
numerical condition for being of K3 type. This leads to a list of $64$ families, and the main goal of this paper is to study these families 
in some detail. 

Being constructed inside products, these fourfolds admit natural projection maps which 
are usually either Mori fibrations (often conic or quadric bundles), or birational transformations. In the latter
case, they can roughly be described as blow ups of rational surfaces inside cubic fourfolds or 
Gushel--Mukai fourfolds, or as blow ups of (sometimes non-minimal, meaning blown-up at some points)
K3 surfaces inside Fano fourfolds with $h^{3,1}=0$ (often, but not always, rational varieties). In fact, there are only four families of  Fano fourfolds in our list which do not admit such a birational description. We thus get a rather colorful zoo of Fano fourfolds 
with interesting geometric properties. 

In particular, most of the K3 surfaces involved are nicely embedded into rational manifolds. We 
recover classical examples of degree seven in $\PP^4$ or degree five in $\PP^5$, studied in 
\cite{Oko84, Oko86}, and meet other ones of a similar flavor. Moreover, in many cases our fourfolds
admit several birational contractions, and we obtain without much effort very nice birational transformations of simple rational manifolds, sometimes {\it special} in the sense of \cite{crauder-katz, alzati-sierra}. For example, from case \cref{27-100-2-1-24-23-1-6-2-4},
we construct a birational involution of the four-dimensional quadric $\QQ^4$, whose exceptional locus is a K3 surface of genus seven blown up at two points. This is very reminiscent 
of the famous birational involution of $\PP^3$ defined by the cubic equations of a genus 3 curve 
of degree 6 \cite{katz-veryspecial}. Of course more examples with similar properties (but not 
involving K3 surfaces) are to be found in the part of the database that we temporarily left aside, 
and that we plan to exploit in the next future. More generally, we get many 
Sarkisov links in dimension four, and we tried to present our database not only as a mere list of items, but rather as an ecosystem where these geometric
objects interact through natural transformations. 

Another very interesting question that we shall touch upon in this paper
is the possibility to write down explicit semiorthogonal decompositions for the derived category of the Fano fourfolds presented here. Recall that such decompositions could be fundamental for example in studying
birational properties, or moduli spaces of bundles, or complexes on the varieties in question.

Beside the several aspects mentioned above, the interested reader might be eager to know what else our list of Fano varieties can be used for and, more specifically, what are the different angles they can look at this paper from. Depending on their taste, the reader may find of interest, for instance: the toolbox we developed to give a systematic interpretation of the geometry of zero loci of sections of homogeneous vector bundles over products of flag varieties; the web of birational correspondences between many of the fourfolds to be met in this paper; the rationality questions we address concerning our varieties and their link with derived categories; the 4 Fano varieties for which the K3 structure seem to be induced by particular singular fourfolds. We refer to \cref{whatCan} for more details on these aspects.

An important problem that we have to mention, but that will not be visible in the sequel, is that we
have been confronted to many kinds of Fano fourfolds in our database, with several different
descriptions as zero loci of vector bundles. We had to find identifications that in
some cases were not obvious at all (see \cref{exampleflavor} for a flavor). Even more seriously, we were not able 
to construct these identifications in a systematic and complete way. In particular we are not
able to bound a priori the Picard ranks and the dimensions of our ambient manifolds, and even less
to guarantee that our lists are complete. It would nevertheless be very interesting to be able 
to carry over a complete classification along the lines followed, in similar but simpler settings, 
in \cite{kuchle, benedetti, inoue-ito-miura}.

This is all the more tempting than we can suspect that the examples of smooth Fano fourfolds in the larger database could 
provide a relevant part of a eventual future full classification. For Fano threefolds this is the point made 
in \cite{DFT}. In dimension four there is an overlap with the list
of families of Fano subvarieties of quiver flag varieties found in \cite{kalashnikov}, 
since the latter can be constructed as towers of Grassmann bundles; a geometric analysis of most relevant cases of \cite{kalashnikov} has been performed in \cite{tufo}. This overlap could be further clarified by 
comparing the quantum periods, which for quiver flag zero loci were computed in the Appendix of \cite{kalashnikov}. In our setting this computation could also be carried out, using the equivariant 
quantum Lefschetz principle, and the results should be compared to those coming from the other
main source of examples of Fano manifolds, namely from the toric world \cite{coates-proc}. 
It would in particular be extremely interesting to understand whether the property of being Fano 
of K3 type can be read on the quantum periods. 
 
The paper is structured as follows. In \cref{sect:results} we report on the general context and summarize our main results, giving an insight on the methods and more details on the motivations behind this work. In \cref{sec:lemmas} we compile a few classical tools and results we take advantage of in the subsequent sections. In \cref{sect:CK3,,sect:GM,,sect:FK3,,sect:rogue} we pursue the investigation of our 64 families of FK3, following the convenient classification we presented above. \cref{sect:blowups} contains a few families of  Fano fourfolds which are not FK3 themselves, but whose Hodge structure in middle dimension contains a proper K3 substructure. In \cref{sec:SOD} we provide some background and references for the semiorthogonal decompositions we exhibit in our examples. Finally, in \cref{Tabless} we present tables summarizing the relevant data for the FK3 under investigation in this paper.

\subsubsection*{Acknowledgements}
The first and third authors acknowledge support from the ANR project FanoHK, grant
ANR-20-CE40-0023. The second author was partially supported by PRIN2017 ``2017YRA3LK'' and by the ANR LabEx CIMI (grant ANR-11-LABX-0040) within the French State Programme ``Investissements d’Avenir''. The second and the fourth authors are members of INDAM-GNSAGA.
We would like to thank Vladimiro Benedetti, Alexander Kuznetsov, Giovanni Mongardi, and Kieran O'Grady for insightful conversations, and the referee for many insighful observations.

\section{Context and results}\label{sect:results}

\TOCstop

\subsection{Fano varieties of K3 type}
We begin with the definition  of  \emph{Fano varieties of K3 type} (abbreviated FK3), which is purely Hodge-theoretical. Recall first the very classical notion of \emph{level} of a Hodge structure.
\begin{definition}Let $H =\bigoplus_{p+q=k} H^{p,q}$ be a non-zero Hodge structure of weight $k$. The \emph{level} $\lambda(H)$ of $H$ is \[ \lambda(H) := \textrm{max} \lbrace q-p \ | \ H^{p,q} \neq 0 \rbrace.\]
\end{definition}
The level of a Hodge structure is a measure of its complexity. It is conventionally set to $-\infty$ if $H=0$, and it is equal to zero if and only if $H$ is central (or of \emph{Hodge--Tate type}), i.e.\ $H^{p,q} \neq 0$ only for $p=q$. 

We now introduce the key notion of Hodge structure of K3 type (see \cite{ilievmanivel, eg2}). 
\begin{definition} Let $H$ be a Hodge structure of weight $k$. We define $H$ to be of K3 type if 
 $$\lambda(H)=2 \qquad \mathrm{and} \qquad h^{\frac{k+2}{2}, \frac{k-2}{2}}=1. $$
\end{definition}
Finally, we define a FK3 as follows.
\begin{definition}\label{definition:fk3} Let $X$ be a smooth Fano variety. We say that $X$ is a Fano variety of K3 type (abbreviated FK3) 
if $H^*(X, \C)$ contains at least one sub-Hodge structure of K3 type, and at most one for each given weight.
\end{definition}
For Fano fourfolds, this simplifies into the condition that $h^{3,1}=1$.

\subsection{Fano fourfolds of K3 type of index bigger than one}
 
 Fano fourfolds $X$ of index $\iota_X>1$ have been classified, see e.g.\ \cite[Chapter 5]{ag5}, or \cite[4-6]{coates} for an overview. In short, there are 35 such families: 1 of index 5, 1 of index 4, 6 of index 3 (out of which 5 with $\rho=1$) and 27 of index 2 (out of which 9 with $\rho=1$). Among these there are three well-known families of FK3 varieties: the cubic fourfolds, the Gushel--Mukai (GM for short) fourfolds, and the Verra fourfolds (double covers of $\PP^2 \times \PP^2$ ramified over a divisor of bidegree $(2,2)$). 
 Note that an important difference between the cubic or GM fourfolds and the Verra fourfolds is that for the general member of the former two families, the derived categories (or Hodge theory) cannot be equivalent to the one of a K3 surface \cite{kuznetcubicfourfold,Kuz-Perry-GM}. On the other hand the derived category of a Verra fourfold can be realized in two different ways as the (Brauer-twisted) derived category of a K3 surface naturally occurring in the construction \cite{Kap-Kap-Mosch}.
 
 We checked that among these 35 families there are no other FK3.

 \subsection{Fano fourfolds of K3 type of index 1}
 Fano fourfolds $X$ of index $\iota_X=1$ are of course not classified. An important source of examples is K\"uchle's list \cite{kuchle}; Only 3 FK3 can be found in there, namely:
 \begin{itemize}
 \item (c5): $\mZ(\Gr(3,7),\W^2 \mU^{\vee} \oplus \mQ^{\vee}(1) \oplus \of(1))$;
 \item (c7): $\mZ(\Gr(3,8),\W^3 \mQ \oplus \of(1))$;
 \item (d3): $\mZ(\Gr(5,10),(\W^2 \mU^{\vee})^{ \oplus 2 }  \oplus \of(1))$.
 \end{itemize}
 
 \medskip\noindent {\bf Convention}. 
 Here and in all the sequel the zero locus of a general section of a globally generated vector bundle $\mathcal{F}$ on a variety $Y$ will be denoted by $\mZ(Y,\mF)$, or simply by $\mZ(\mF)$ when no confusion can arise. In particular $\mZ(\mathcal{F})$ is smooth of codimension equal to the rank of 
 $\mathcal{F}$ (or empty). Moreover its canonical bundle is given by the adjunction formula. 
 It will always be implicit in all the examples below that we only consider general 
 sections, and we will never discuss this genericity condition in more detail. (Note that we will not discuss the question of the completeness of our families.)
 \medskip
 
 If (c5) is still mysterious, (c7) and (d3) have been studied by Kuznetsov in \cite{kuznetsovKuchlePicard}, where he showed how they are (respectively) the blow up of a cubic fourfold in a Veronese surface, and the blow up of $(\PP^1)^4$ in a K3 surface. We do not include (c5) and (c7) in our list; we only briefly mention (d3), which we denote by \cref{31-120-5-1-26-16-1-2-1-2-1-2-1-2-1-2}, since we add some extra information other than what is already contained in the two aforementioned works.
 
 Two other examples of FK3 fourfolds were included in \cite{eg2}, with labels (B1) and (B2) therein. Since they were only sketched in the cited work, we include them in our analysis. They are the Fano fourfolds \cref{39-160-2-1-22-21-1-2-2-4} and \cref{39-160-3-1-22-18-1-2-1-2-1-4}.

Plentiful of four-dimensional Fano toric complete intersections were found in \cite{coates-proc} (at least $738$ families), and a few dozens more (at least $141$) 
were constructed in \cite{kalashnikov} as quiver flag zero loci. As already mentioned 
it would be very interesting to compare these families with ours, and find how many 
families have been found altogether -- probably a very significant proportion of 
the complete, still elusive list of Fano fourfolds. Moreover, having two different descriptions for the same variety allows one to get more information, as shown e.g. in \cite{BelmansFT}.

\subsection{FK3 in this paper: a recap} 
Once again, the main goal of this paper is to analyze 64 families of Fano fourfolds of K3 type not previously considered - or only briefly mentioned - in the literature. Given their description from zero loci of bundles in products of flag manifolds, it was not a surprise to check that all of them have Picard rank greater than one. For each case, we will provide the main numerical invariants, and we will geometrically describe the projections to each factor and deduce semiorthogonal decompositions for their derived categories. This gives access to the 
geometric structure of the Fano fourfolds in each family.

We subdivided our examples into four disjoint groups, namely Fano varieties of \textbf{C, GM, K3} or \textbf{R} origin, inside which the examples are listed, roughly, by 
increasing complexity.
Fano varieties of \textbf{C} origin admit at least one birational map to a  cubic fourfold, either special or general. Similarly, Fano fourfolds of \textbf{GM} origin  admit at least one birational map to a Gushel--Mukai fourfold,  either special or general. We remark that the two sets are disjoint; in other words, no Fano fourfold 
in our list has a birational morphism onto both a cubic and a GM fourfold.

Fano varieties of \textbf{K3} origin are varieties which do not belong to any of the two previous groups, but  can be realized as blow ups of other Fano varieties with central Hodge structure along a (sometimes non-minimal) K3 surface.
Varieties in each of these first three groups share many properties among each other, and we can often understand their K3 structure from the birational maps in question.

Finally, there is a small bunch of Fano varieties (of Picard rank two) of \textbf{R} (or \emph{rogue}) origin that do not fall in any of the previous categories. They are in fact birational to singular fourfolds, and we can understand their K3 structures from one of their projections (either a $\dP_5$-fibration or a conic bundle). Summing up, we have this following convenient classification:
\begin{itemize}
    \item 17 Fano fourfolds of \textbf{C} origin;
    \item 6 Fano fourfolds of \textbf{GM} origin;
    \item 37 Fano fourfolds of \textbf{K3} origin;
    \item 4 Fano fourfolds of \textbf{R} origin.
\end{itemize}

However, thinking of the more ambitious task of classifying Fano fourfolds without restricting to those of K3 type, such a notation is certainly not a good idea. We introduced an identification number for each Fano fourfold that we hope can carry over to any larger database: we denoted by $[\sharp\rho-s-v-n-?]$ a 
family of Fano fourfolds $X$ of Picard rank $\rho$, with $h^0(-K_X)=s$, $(-K_X)^4=v$ and Hodge level $n$. The last symbol $?$ is needed only in case the  first four numbers do not identify a unique deformation class of fourfolds in our list, and it is a capital latin letter.

\begin{rmk}\label{rmk2.4}
The study of the natural projections on the factors of the ambient variety of the Fano varieties presented here is particularly relevant for the study of extremal contractions of Fano fourfolds (cfr.~\cite[Chapter 8]{matsuki}). In particular, when \( h^{1,1}(X) = 2 \) for such an $X$, the restrictions of the two projections give the two extremal contractions: indeed, if the two projections contract at least a curve, say $C$ and $D$ respectively, the pullbacks $L$ and $M$ of ample bundles verify 
    $L.C=M.D=0$, from which one can deduce that $\mathrm{Eff}=\langle C,D\rangle $ and 
    $\mathrm{Nef} = \langle L,M\rangle$. The other conditions (normality of the target and connectedness of the fibers) can be checked case-by-case on all relevant examples. When \( h^{1,1} > 2 \), \( X \) may have other extremal contractions not induced by the projections. 
    
    Along the study of our examples, there are a few instances of birational maps which are quite special, as the contracted divisors have a curve or a point as image.  
\end{rmk}

\begin{rmk} 
\label{noKuchleNoVerra}
No Fano fourfold in our list has a birational morphism to either K\"uchle's (\emph{c5}) fourfolds or to the \emph{Verra} fourfolds. In the (\emph{c5}) case, it is likely that the low index prevents any blow up from having ample anticanonical class. 
For \emph{Verra} fourfolds the situation is  different. In fact, the index is two, and it seems conceivable that some blow up  could be Fano. However, the Verra fourfolds themselves can only be described, to the best of our knowledge, as zero loci of sections of a non-completely reducible vector bundle. To see this, consider the zero locus of a general section of $\mQ_{\PP^n}(0,1)$ on $\PP^n\times\QQ^n$; the second projection is an isomorphism onto $\QQ^n$, while the first projection to $\PP^n$ is a double cover branched over a quadric. Letting $n=8$ and restricting to $\PP^2\times \PP^2
\subset\PP^8$, we thus get the Verra fourfolds as zero loci of general sections of  $\Lambda(0,0,1) \oplus \of(0,0,2)$ on $\PP^2 \times \PP^2 \times \PP^9$, where the rank eight vector bundle $\Lambda$ is the restriction of $\mQ_{\PP^8}$  
to $\PP^2\times \PP^2$: a homogeneous but not completely reducible bundle. 
\end{rmk}

A particularly intriguing question about FK3 varieties is about the generalization of Kuznetsov's conjecture on the cubic fourfold, which is expected to be rational whenever its Kuznetsov component is  equivalent to the derived category of a K3 surface. One could ask a similar question for each of the 64 families in our list.

For Fano fourfolds of \textbf{C} and \textbf{GM} origin, there is no surprise: we either recover the rationality of known special examples, or we are unable to conclude.
For almost all Fano fourfolds of \textbf{K3} origin, we can easily decide about their rationality. There are two families however, for which we cannot. In particular, Fano \cref{23-80-2-1-24-23-1-2-3-6} is rational if and only if the prime index $2$ Fano fourfold $X_{16}$ of degree 16 is (which is currently unknown). For Fano \cref{19-60-3-5-1-22-23-1-2-1-5},  rationality would imply the stable rationality of a smooth cubic threefold, which is another famously open problem.

For two out of the four families of  \textbf{R} origin, we are not able either to determine their rationality. The  exceptions are the case \cref{27-99-2-1-25-25-1-3-2-6} and \cref{31-120-3-1-2-22-20-1-6-1-8}, where we expect the Fano fourfold to contain in its derived category a copy of the derived category of a K3 surface blown up in three points.
In the case of \cref{31-120-3-1-2-22-20-1-6-1-8} and \cref{25-90-3-5-1-22-23-1-5-1-7}, we can identify a copy of the category of a smooth K3 surface in the derived category.  We have only a partial understanding of the K3 structure of \cref{33-129-2-1-23-23-1-4-2-7}. On the other hand all these fourfolds are birational to some singular fourfolds defined explicitly. The question of their rationality could therefore turn out to be quite interesting, and a good testing ground for a generalization of Kuznetsov's conjecture.

Finally, we thought it would be useful for the reader to have a visualization of the various birational links we found among all our Fano fourfolds. The tables at the end of this paper give a recap of this and many other pieces of information on the Fano varieties analyzed in this paper.

\subsection{Construction of the database}
We insist that the leading role in this manuscript is played by Fano fourfolds $X$ 
of K3 type satisfying the following
\begin{workhyp}
\label{descriptionAsZero}
$X$ admits a description as the zero locus of a general global section of a globally generated completely reducible homogeneous bundle over a product of flag varieties, together with a strong amplitude condition of the anticanonical bundle, see assumption (\ref{stronglyFanity}) below.

\end{workhyp}
These bundles can be conveniently built by taking direct sums of irreducible bundles, for which a combinatorial description is easy to handle.

More in general, for a pair $(Y,\mathcal{F})$ with $Y$ a product of flag varieties and $\mathcal{F}$ a vector bundle on $Y$ as above, we can consider the $n$-dimensional zero locus $X$ of a general global section of $\mathcal{F}$. We can make use of classical tools such as the Hirzebruch--Riemann--Roch Theorem and the Borel--Weil--Bott Theorem, and Koszul complexes, to (at least partially) determine the following data:
\[
h^0(-K_X), \quad (-K_X)^n, \quad h^{i,j}=h^i(X,\Omega^j_X),  \quad \chi(T_X).
\]
For a more detailed explanation of the steps to be performed for this sake, we refer to \cite[\textsection 3]{DFT}. We remark that the determination of the Hodge numbers $h^{i,j}$ is based on the computation of Euler characteristics from several cohomology sequences. If one is unfortunate, this can give only partial results, especially when $\mathcal{F}$ has several summands. We are nonetheless able to determine all the Hodge numbers for the Fano varieties we deal with in this paper.

The search for Fano fourfolds satisfying \cref{descriptionAsZero} and the computation of their invariants can be efficiently implemented on computer algebra softwares. In order to construct our database we took advantage of the computer algebra software \cite{M2} and its packages \cite{M2:Schur, M2:Schubert2}. 
\medskip

The reader may wonder whether a classification of all Fano fourfolds satisfying \cref{descriptionAsZero} is possible. This was our initial, a posteriori very optimistic aim. However, a first serious obstacle occurs when one realizes that there exist  varieties  $\mZ(Y,\mathcal{F})$ whose anticanonical bundle is ample, even if $\omega_Y\otimes \det(\mathcal{F})$ is not, as the following examples show.

\begin{example}
Consider $Y=\PP^2 \times \PP^1 \times \PP^1 \times \PP^1 \times \PP^1$ and $\mathcal{F}=\of(1,1,0,0,0) \oplus \of(1,0,1,0,0) \oplus \of(1,0,0,1,0) \oplus \of(1,0,0,0,1)$. Then
\[
\omega_Y \otimes \det(\mathcal{F}) = \of(1,-1,-1,-1,-1),
\]
which is clearly not anti-ample on $Y$. However, the zero locus of a general section of $\mathcal{F}$ is, by 
\cite[Lemma 2.8]{DFT}, a del Pezzo surface of degree 5.
\end{example}

\begin{example}
We have seen in \cref{noKuchleNoVerra} that the zero locus of a general section of $\mQ_{\PP^n}(0,1)$ on $\PP^n\times\QQ^n$ is 
isomorphic with $\QQ^n$, which we can thus also obtain as the  zero locus of a general section of $\mF=\mQ_{\PP^n}(0,1)\oplus \of(0,2)$ on 
$Y=\PP^n
\times\PP^{n+1}$. But the bundle $\omega_Y \otimes \det(\mathcal{F}) = \of(n,0)$ is not ample, although it will of course become ample 
when restricted to $\QQ^n$. 
\end{example}

For this reason, in our search and in this paper we only deal with Fano fourfolds $X=\mZ(Y,\mathcal{F})$ such that
\begin{equation}
\label{stronglyFanity}
\omega_Y\otimes \det(\mathcal{F}) \mbox{ is ample.}
\end{equation}

Another serious problem is that a given fourfold $X$ can have several incarnations
as such zero loci, and even infinitely many as the next example shows.

\begin{example}
\label{exampleflavor}
Consider $X=\PP^1\times \PP^{d-1}$. This is the zero locus of a general section of
$\mathcal{F}=(\mQ \boxtimes \mU^\vee)^{\oplus 2}$ on $Y=\PP^{m-1} \times \PP^{m+d-1}$ for any $m\geq d$. 
Indeed, these sections are defined by two general morphisms $s,t: V_{m+d}\rightarrow 
V_m$, and with a little abuse of notation they vanish at $(x,y)$ when $s(y),t(y)\subset x$. In particular there exist
scalars $\sigma,\tau$, not both zero, such that $\sigma s(y)+\tau t(y)=0$. Moreover the
point $[\sigma,\tau]$ in $\PP^1$ is unique since when $m\ge d$, the kernels of 
$s$ and $t$ intersect trivially. So $\mZ(Y,\mathcal{F})$ is the projectivization 
of a kernel bundle of rank $d$ on $\PP^1$, and this kernel bundle is trivial since 
$\sigma s+\tau t$ is everywhere surjective. Hence the claim. 
\end{example}

We found many instances of such phenomena, sometimes organizing into infinite series 
like in the previous example, sometimes more sporadic (see \cref{26-94-2-1-28-28-1-2-5} for an example). A systematic understanding 
of these redundancies would be necessary for a classification, at least when \eqref{stronglyFanity} holds, but remains out of reach for the time being and falls outside the scope of this paper.

\subsection{What can the reader do with this list?}
\label{whatCan}
There are several angles from which a reader could look at this paper. For example, people interested in the study of the (at the moment untamed) geography of Fano fourfolds could think of this paper as a \emph{toolbox}. There are several attempts to classify, or at least gain a partial understanding of these varieties, including our own \emph{vector bundle search} method as explained in the introduction, or for example via mirror symmetry. The problem often is not only to produce lists of examples, but to understand them, and in a systematic way. The tools that the second and fourth author started developing in \cite{DFT}
and that we broaden here in \cref{sec:lemmas} and throughout the whole paper will be useful to describe and characterize whichever Fano (in dimension 4 and above) we will encounter from now on.

\smallskip
Moreover, we believe that one of the merits of the present paper is to highlight the web of birational correspondences between many of the fourfolds we present. This matches, at least in spirit, the style of Mori--Mukai original classification of Fano threefolds with Picard rank $\rho>1$. We tried to make these links as visually clear as possible by summarizing them in the tables at the end of the paper. More information can also be extracted from the actual
description of the examples. 

\smallskip
The reader keen in understanding \emph{rationality questions} (especially in connection with \emph{derived categories}) could feel intrigued by some special examples in our list and their connection with a generalized version of Kuznetsov's conjecture on the rationality of the cubic fourfold. In fact, many of our examples are both (sometimes evidently) rational, and with their derived category containing (sometimes less evidently) a copy of the derived category of a K3 surface.  There are however some exceptions, in which our understanding is less complete and they give rise to many interesting questions.

\smallskip
For example, in the cases \cref{31-120-3-1-2-22-20-1-6-1-8} 
we understand the K3 category as a subcategory coming from an actual K3 surface, but we cannot conclude anything (yet) about its rationality.
There are as well examples where again the K3 category is geometrical (or \emph{commutative}), but the rationality of the examples would be equivalent to solving known open problems. For instance, determining the rationality of \cref{23-80-2-1-24-23-1-2-3-6} would imply the rationality of the Fano fourfold $X_{16}$, while the rationality of \cref{19-60-3-5-1-22-23-1-2-1-5,25-90-3-5-1-22-23-1-5-1-7} would yield the stable rationality of the cubic threefold.
Finally, if we look at \cref{33-129-2-1-23-23-1-4-2-7}, there is an interesting link between our work and the recent work of Hassett--Pirutka--Tschinkel \cite{HPT-22}. In fact, they showed how there exists a family of smooth complete intersections of bidegree $(2,2)$ in $\PP^2 \times \PP^3$ whose very general fiber is not even stably rational, whereas some special fiber are rational. In our case Fano \cref{33-129-2-1-23-23-1-4-2-7} is in fact birational to a \emph{singular} complete intersection of this type. It would be interesting to push this analogy further and say something similar for our family, at the rationality level.
 
 \smallskip
Any Hodge theorist or categorist interested in the study of \emph{K3 Hodge structures} or \emph{K3 categories} may feel the urge to check the most unexpected cases, that is, those which we call \emph{rogue} listed in \cref{sect:rogue}. In particular, we invite the interested reader to think of how these K3 structures seem induced by singular fourfolds, and how the study of special singular fourfolds could lead to the discovery of new, unexpected, FK3 varieties in future.

\smallskip
Last but not least, the \emph{elephant in the room} is the absence in this paper of any projective families of \emph{irreducible holomorphic symplectic manifolds} (IHS, or hyperk\"ahler) constructed from these Fano varieties of K3 type. In fact, we could speculate that to many of these Fano varieties one could associate (possibly non-locally complete) families of IHS, very likely of K3$^{[n]}$ type, for example considering moduli spaces of appropriate objects contained in these FK3. This is definitely a direction worth pursuing, and which will hopefully lead to many other interesting geometric constructions in the future.

\subsection{Notation} 
We denote by $\PP^n$ the projective space, and by $\QQ^n$ the smooth quadric
of dimension $n$. (When there are several, we distinguish them by an
index.) The Grassmannian of $k$-dimensional subspaces of $\mathbb{C}^n$ is denoted $\Gr(k,n)$. More generally, $\Fl(k_1,k_2,n)$ 
denotes the flag manifold parametrizing flags of subspaces of dimensions $k_1<k_2$
in $\mathbb{C}^n$ (and a similar notation is used for more general flags). As in the case of projective case, the projectivization of a vector bundle $\PP(\mE)$ will denote the space of 1-dimensional subspaces in the fiber.

A fixed $i$-dimensional vector space is denoted $V_i$, and a moving subspace of 
dimension $j$ is denoted $U_j$, except for lines which we usually denote by
$l,m$, etc. 

General vector bundles are denoted by $\mE$, $\mF$, $\mG$, and line bundles by 
$\mL$, $\mM$, etc. 
The tautological and quotient vector bundles on $\Gr(k,n)$, whose ranks are 
$k$ and $n-k$, are denoted by $\mU$ and $\mQ$ respectively. Over $\Fl(k_1,k_2,n)$, we denote by $\mQ_2$ the pull-back of the quotient vector bundle from the projection to 
$\Gr(k_2,n)$. We will occasionally use $\mR_i$ to denote the tautological bundle $\mU_{i}/\mU_{i-1}$. 
As already discussed, the zero locus of a general section of a globally generated vector bundle $\mathcal{F}$ on a variety $Y$ will be denoted by $\mZ(Y,\mF)$, or simply by $\mZ(\mF)$ when no confusion can arise

\medskip
The Fano fourfolds from our database are denoted $X$ in each example, while $Y,Z$ are
usually some auxiliary fourfolds. When they are singular we add a $\circ$, typically 
$Y^\circ, Z^\circ$, except for a projective cone that we rather denote $C(Y)$ or $C(Z)$,
with vertex $v$. We add a $\sim$ when we blow up a variety already encountered 
in a previous example. In many cases our Fano fourfold $X$ admits a quadric bundle 
structure, and we denote the discriminant by $\Delta$.

\medskip
We also use the notation $X_d$ for a Fano fourfold of degree $d$: typically,  
$X_3$ for a cubic fourfold and $X_{10}$ for a Gushel--Mukai fourfold, that we freely
abbreviate into GM fourfold. Notice that in the case of a prime Fano, the \emph{classical} degree coincides with the anticanonical degree, up to $\iota^{\dim X}.$

Fano threefolds of degree $d$ are denoted by $W_d$. Similarly,  K3 surfaces
of degree $d$ are denoted by $S_d$, and a blow up in $k$ points of a K3 surface of degree $d$  is denoted by $S_d^{(k)}$. Also, a del Pezzo surface of degree $d$ is denoted by $\dP_d$, and a  scroll or a Hirzebruch surface  of degree $d$ by $\Sigma_d$.
Finally, 
a curve of degree $d$ is denoted by $C_d$.

\TOCstart

\section{A few preliminaries}
\label{sec:lemmas}
In this section we compile a few classical tools that will be used multiple times in the study of our examples. Unless otherwise specified, with \emph{blow up} we mean the blow up of a smooth irreducible variety with center a smooth subvariety.

\TOCstop

\subsection{Cayley tricks} 
We start with some natural birational transformations naturally showing up in our settings. 

\begin{lemma}\label{lem:blowhighercod} \label{lem:higherpeco} \label{lem:blowupcodim2}
Let $X$ be a smooth projective variety with a globally generated line bundle
$\mL$, and consider $Y:=\mZ(\PP^m \times X,\mQ_{\PP^m}\boxtimes \mL)$.
Then $Y \cong \Bl_Z X$, where $Z=\mZ(X,\mL^{\oplus m+1})$ and $Y$ and $Z$ are smooth for a general choice of a section. In particular 
$Y \cong X$ if and only if $m \geq \ddim X$.
\end{lemma}

\begin{proof}
 Pick a basis of $v_0, \ldots, v_m$ of $\mathbb{C}^{m+1}$. A section $\sigma \in H^0 ( \PP^m \times X, \mQ_{\PP^m}\boxtimes \mL)$ is given by $\sigma=\sum v_i \otimes s_i$ for some sections $s_0,\ldots ,s_m$ of $\mL$, and $\sigma$ vanishes at $(l,x) \in \PP^m \times X$ if and only if $\sum s_i(x) v_i$ belongs to $l \otimes \mL_x$. This implies that the projection to $X$ is an isomorphism outside the base locus 
 $Z=\mZ(s_0, \ldots, s_m)$, which is smooth in general when $\mL$ is globally generated. The special fibers over $Z$ are isomorphic to $\PP^m$. So we end up with a birational morphism $f:Y \to X$ whose irreducible exceptional divisor $f^{-1}(Z)$ is a 
 $\PP^m$-bundle over a smooth locus $Z$, and we can conclude by \cite[Thm. 1.1]{Ein-ShepB}.
 \renewcommand\qedsymbol{$\square$}
\end{proof}
In the above lemma, picking $m=1$ results in a certain (codimension 2) blow up being described as the zero locus of a line bundle in a product space. There is a well-known generalization to this, under the name of
\emph{Cayley trick} (see \cite[Thm 2.4]{kimkim}, or \cite[3.7]{ilievmanivel}), which admits both a Hodge-theoretical and a derived categorical statement, and can be considered as a generalization of the formula for the Hodge structure of the blow up and Orlov's formula for the derived category of blow ups.

To be explicit, assume that we have $Y= \mZ(G,\mF)$, where $\mF$ is ample of rank $r\geq 2$. We have a natural isomorphism $H^0(G,\mF) \cong H^0(\PP(\mF^{\vee}), \of_{\PP(\mF^{\vee})}(1))$. Hence, the same section defining $Y$ also 
defines a hypersurface $X$ in $\PP(\mF^{\vee})$. It follows that $X$ admits a stratified projective bundle structure over $G$, with generic fiber $\PP^{r-2}$, and special fibers $\PP^{r-1}$ over $Y$, both locally trivial over their strata. This implies that the Hodge structure of $X$ can be described in terms of those of $Y$ and $G$, see \cite[Prop.\ 48]{BFM}, and as well that 
the derived category of $X$ admits a semiorthogonal decomposition containing $r-1$ copies of $D^b(Y)$. When the rank of $\mF$ is exactly 2, this produces a generalization of the above lemma. In our paper, this phenomenon happens with $r=2$, as in \cref{49-211-2-1-21-19-1-5-1-6}. In particular, we have the following lemma (where we remark that $Y$ - and therefore $X$ - are smooth for a general choice of a section).
\begin{lemma}{\cite[3.7]{ilievmanivel},\cite[3.2]{kimkim}} Let $X$ be $\mZ(\PP_G(\mF^{\vee}),\of_{\PP}(1))$, with $\of_{\PP}(1)$ being the relative ample line bundle, $G$ smooth and $\mF^{\vee}$ of rank 2. Then $X \cong \Bl_Y G$, with $Y= \mZ(G,\mF)$.
\end{lemma}

\subsection{Ubiquity}
An important number of ambiguities in our models of Fano fourfolds come from the bundle 
 $\mQ \boxtimes \mU^{\vee}$ on a product of Grassmannians. The next statement gives a few
identifications that will be useful in the sequel. 

\begin{proposition}\label{lem:flag}\label{lem:qinflag} \label{lem:qboxu}\label{lem:secondproj} \label{lem:blowgrass}
Consider $\mZ=\mZ(\Gr(a,m) \times \Gr(b,n),\mQ \boxtimes \mU^\vee)$.
\begin{enumerate}
    \item If $m\leq n$, then $ \mZ$ is isomorphic with the Grassmann bundle 
    $\Gr(b,\mU \oplus \of^{\oplus n-m})$ over $\Gr(a,m)$. It can also be seen as 
    a general zero locus  $\mZ(\Fl(b,a+n-m,n), \mQ_2^{\oplus n-m})$.
In particular $\mZ$ is empty if $a+n<b+m$.
 \item If moreover $a=b$, then the second projection $\mZ\rightarrow \Gr(b,n)$ is birational, with fiber isomorphic
 to $\Gr(d,n-a+d)$ over $B\in  \Gr(b,n) $ meeting some fixed $K_{n-m}$ in dimension $d$. 
 \item If $a=b$ and $m=n-1$, then $\mZ\cong \Bl_{\Gr(a-1, n-1)}\Gr(a,n)$.
\end{enumerate}
\end{proposition}

\begin{proof}
We start with (1). 
A section of $\mQ \boxtimes \mU^{\vee}$ is defined by a morphism $\alpha\in \Hom(V_{n}, V_m)$. 
If $n\ge m$ this morphism is in general surjective and we can choose an
isomorphism  $V_{n}\simeq V_m\oplus K_{n-m}$ such that $\alpha$ is just the 
projection to the first factor. Then $\mZ$ is  
variety of pairs $(A,B)$ such that $\alpha(B)\subset A$, or equivalently, 
$B\subset \alpha^{-1}(A)=A\oplus K_{n-m}$. Hence the first claim. This is also
the variety of flags $B\subset C\subset V_n$ 
such that $C\supset K_{n-m}$, which is exactly the zero locus of a general section of 
$\HOM(K_{n-m},\mQ_2)$.  Hence the second claim.

When $a=b$, the condition $\alpha(B)\subset A$ implies that $\alpha(B)=A$ when $B$ is transverse to 
$K_{n-m}$, so the second projection is birational. Then (2) readily follows. 

Finally (3) is  \cite[Lemma 2.2]{DFT}.
\renewcommand\qedsymbol{$\square$}
\end{proof}

\subsection{Torelli} 
In several cases we will get a fourfold $X$ of K3 type which is a blow up in several different 
ways, and we will use the following statement. 

\begin{lemma}\label{lem:dereq} Let $X$ be isomorphic both to the blow up of a smooth fourfold $Y$ along a (possibly non-minimal) K3 surface $S$ and to the blow up of a fourfold $Z$
along a (possibly non-minimal) K3 surface $T$. 
 If $H^{2,2}(Y)$ and $H^{2,2}(Z)$ have no transcendental part, then the two K3 surfaces $S_{min}$ and $T_{min}$ are derived equivalent.
\end{lemma}

\begin{proof}
Recall that two minimal K3 surfaces are derived equivalent if and only if there is a Hodge isometry between their transcendental lattices (see, e.g., \cite[16, Cor. 3.7]{HuyK3book}). Under our assumptions, the blow up formulae identify the transcendental lattices of $S_{min}$ and $T_{min}$ with the transcendental lattice of $H^{2,2}(X)$.
\renewcommand\qedsymbol{$\square$}
\end{proof}

\subsection{Conic bundles}
Assume that $f: X \to B$  is a conic bundle, given by a line bundle $\mK$, a rank three bundle $\mE$ on $B$ and a bilinear map $\mK^{\vee} \to \Sym^2\mE^{\vee}$. This is a standard way of defining a conic bundle, see for example \cite{Sar83}. In this case the degeneracy loci of the conic bundle are given by the degeneracy loci of the induced morphism $\varphi: \mE\to \mE^{\vee} \otimes \mK$. Since $\mE$ has rank three 
there are (under the usual genericity conditions) only two degeneracy loci in $B$, 
namely the discriminant divisor $\Delta$, and its  codimension two singular locus $\Delta_{sing}$.
We can use the formula in \cite[Theorem 1, Theorem 10]{Harris-Tu} to compute the cohomology classes of $\Delta$ and $\Delta_{sing}$. 

\begin{lemma} \label{lem:conicbundle} In the above situation, let $k:=c_1(\mK)$ and  $c_i:= c_i(\mE)$. 
Then
\[
[\Delta]= 2 c_1 +3k \qquad \mbox{and} \qquad 
[\Delta_{sing}]=4(k^3+2k^2c_1+kc_1^2+kc_2+c_1c_2-c_3). \]
\end{lemma}

These formulae will apply to many of our examples.

\subsection{Eagon--Northcott complexes}
We will use several times the following classical result (see, e.g., \cite[(6.1.6)]{weyman}, \cite[Example 2.8]{BFMT}).

\begin{proposition}\label{lem:EN}
Let $\varphi: \mE \to \mF$ be a morphism of vector bundles on $X$, of respective ranks $e>f$. Suppose that its degeneracy locus $Z$ 
(defined as a scheme by the $f$-th minors of $\varphi$) has the expected codimension $e-f+1$.
Then the $\of_X$-module $\of_Z$ admits a locally free resolution given by the Eagon--Northcott complex 
\[
0 \to 
\W^{e}\mE\otimes S^{e-f}\mF^\vee \otimes \det\mF^\vee\to
\cdots \to 
\W^{f+1}\mE\otimes\mF^\vee\otimes \det\mF^\vee\to
\W^f\mE \otimes \det\mF^\vee \to  \of_X\to \of_Z \to 0.
\]
\end{proposition}

\medskip
In our situations $Z$ will be smooth (typically 
because the rank will not be allowed to drop more for dimensional reasons), and we will 
resolve its canonical bundle by dualizing the Eagon--Northcott complex. It turns out that
in many cases the dual complex will have a special shape that will allow us to apply the next
simple, but very useful statement. 

\begin{lemma}\label{lem:section}
Let $Z\subset X$ be a smooth subvariety, with a line bundle $\mL_Z$ that can be described,
as an $\of_X$-module, as the cokernel of a morphism of vector bundles on $X$, say 
$$\mA  \to \mB\oplus \mM\to \mL_Z\to 0,$$
where $\cM$ is a line bundle. 
Then the morphism  $\cM\to \mL_Z$ defines a canonical section $\sigma$ of  $\cM^\vee_{|Z}\otimes\mL_Z$, 
whose zero locus is exactly the locus where the morphism $\mA  \to \mB$ is not 
surjective. 
\end{lemma} 

\begin{proof}
Since $\mL_Z$ is the cokernel of the morphism  $\theta: \mA  \to \mB\oplus \mM$, the section $\sigma$
vanishes exactly when the fiber of $0_\mB\oplus\mM$ is contained in its image. If it does, the component $\theta_1: \mA  \to \mB$ cannot be surjective, since otherwise $\theta$ would itself be surjective. Conversely, if $\theta_1$ is not surjective at some point, the induced morphism $\Ker\theta_1\to\mM$ must be surjective, since the cokernel of $\theta$ would have rank  bigger than one at any point where this fails. 
\renewcommand\qedsymbol{$\square$}
\end{proof}

\subsection{Derived categories}

A natural question to ask about a Fano variety $X$ is to describe semiorthogonal decompositions of $\Db(X)$ and eventually compare them. Due to the fact that $H^i(X,\of_X)=0$ for $i \neq 0$, every line bundle is exceptional on a Fano variety. Moreover one can expect many exceptional vector bundles and objects in the derived category. Building exceptional sequences from such objects gives interesting semiorthogonal decompositions. In the particular case of Fano fourfolds of K3 type, there should exist at least one exceptional collection whose semiorthogonal complement is a K3 category, sometimes actually described by a K3 surface.
In all the examples presented in this paper, we can use explicit geometric descriptions to provide semiorthogonal decompositions, and even several of them; we will avoid the  interesting but a priori deeply technical question of comparing them.

\smallskip
We refrain from recalling here definitions and properties of semiorthogonal decompositions, for which the interested reader can consult \cite{Kuz:ICM2014}. We focus on the description of semiorthogonal decompositions of $\Db(X)$ for $X$ a smooth projective variety.
More generally, if $X$ is singular, one can replace $\Db(X)$ by a so-called \emph{(crepant) categorical resolution}, that is, a smooth and proper dg category $\mD$ with a functor $\mD \to \Db(X)$ whose homological properties are the ones of a functor induced by a geometric (crepant) resolution of singularities.
We also refrain from introducing such topic, for which we refer again to \cite{Kuz:ICM2014}.

For the sake of readability, we treat explicit examples and give tables with references in \cref{sec:SOD}.
The main idea is to consider in each example surjective morphisms $\pi: X \to M$ from a Fano fourfold $X$ to another variety $M$.
We will face two cases: $M$ is also a smooth fourfold, and $\pi$ is a blow up; or
the dimension of $M$ is smaller, and $\pi$ is a Mori fiber space, i.e.\ a  map of relative Picard number 1 whose general fiber is a smooth Fano variety. In both cases, semiorthogonal decompositions can be constructed from the map $\pi$. We list all of them in \cref{tab:semiorth1} (Fano and other rational varieties of dimension up to 4) and \cref{tab:semiorth2} (Mori fibrations).
In any case, given a Fano fourfold $X$, we describe natural exceptional collections whose complement, called \emph{Kuznetsov components} (with respect to of $\Db(X)$), contain the most interesting pieces of information on the geometry of $X$.
Such components depend on the exceptional collection and are only conjecturally equivalent; moreover,
in some cases (in particular Mori fiber spaces with high-dimensional fibers), the Kuznetsov component for the given exceptional collection is not known in general.

\TOCstart

\section{FK3 from cubic fourfolds}\label{sect:CK3}
This section contains the  list of the Fano fourfolds of K3 type that we obtained inside products of flag manifolds,  
where at least one projection is the blow up of a cubic fourfold. In this section as well as in the next sections, for each Fano fourfold we list its invariants, we study its geometric properties via the projections on the factors of the ambient variety, and we provide information on its rationality. Possible orthogonal decompositions are also exhibited, and we refer to \cref{sec:SOD} for the corresponding notation. We only list the Hodge numbers which are different from zero, of course avoiding to mention the Hodge symmetries and $h^{0,0}$.

\medskip\medskip\begin{fano}\fanoid{2-36-144-2} 
$\mZ(\PP^1 \times \PP^5,\of(0,3) \oplus \of(1,1))$.
\label[mystyle]{36-144-2-1-28-28-1-2-1-6}
\subsubsection*{Invariants}  $h^0(-K)=36, \ (-K)^4=144, \  h^{1,1}=2, \ h^{3,1}=1, \ h^{2,2}=28$, $-\chi(T)=28$.

\subsubsection*{Description} 
\begin{itemize}
    \item $\pi_1 \colon X \to \PP^1$ is a fibration in cubic threefolds.
    \item $\pi_2 \colon X \to \PP^5$ is the blow up $\Bl_{\dP_3}X_3$ of a general cubic fourfold along a cubic surface.
    \item Rationality: not known.
\end{itemize}

\subsubsection*{Semiorthogonal decompositions}
\begin{itemize}
\item 4 exceptional objects and an unknown category from $\pi_1$.
\item 12 exceptional objects and a K3 category from $\pi_2$.
\end{itemize}

\subsubsection*{Explanation} The description of $\pi_1$ is obvious. For $\pi_2$, simply apply \cref{lem:blowupcodim2}. 
\end{fano}

\medskip\medskip\begin{fano}\fanoid{2-28-99-2} 
\label[mystyle]{28-99-2-1-1-23-29-1-3-1-6} $\mZ(\PP^2 \times \PP^5,\mQ_{\PP^2}(0,1) \oplus \of(0,3)).$

\subsubsection*{Invariants}  $h^0(-K)=28, \ (-K)^4=99, \  h^{1,1}=2, \  h^{1,2}=1, \ h^{3,1}=1, \ h^{2,2}=23$, $-\chi(T)=29$.
\subsubsection*{Description} 
\begin{itemize}
    \item $\pi_1 \colon X \to \PP^2$ is a $\dP_3$ fibration.
    \item $\pi_2 \colon X \to \PP^5$ is the blow up $\Bl_E X_3$ of a cubic fourfold along  a plane elliptic curve $E$.
    \item Rationality: unknown, since the cubic fourfold is general.
\end{itemize}

\subsubsection*{Semiorthogonal decompositions}
\begin{itemize}
\item 3 exceptional objects and an unknown category from $\pi_1$.
\item 3 exceptional objects, the K3 category of $X_3$, and two copies of $\Db(E)$ from $\pi_2$.
\end{itemize}

\subsubsection*{Explanation} We simply use  \cref{lem:blowupcodim2} to identify the zero locus of the first bundle with $\Bl_{\PP^2}\PP^5$. Intersecting with the zero locus of a section of $\of(0,3)$ completes the identification, since it cuts the blown up plane in a cubic, hence elliptic curve.
\end{fano}

\medskip\medskip\begin{fano}\fanoid{2-45-192-2} 
$\mZ(\PP^2 \times \PP^5,\of(1,2) \oplus \mQ_{\PP^2}(0,1))$.
\label[mystyle]{45-192-2-1-22-19-1-3-1-6} 

\subsubsection*{Invariants}  $h^0(-K)=45, \ (-K)^4=192, \  h^{1,1}=2, \ h^{3,1}=1, \ h^{2,2}=22$, $-\chi(T)=19$.

\subsubsection*{Description}

\begin{itemize}
    \item $\pi_1 \colon X \to \PP^2$ is a quadric surface fibration degenerating along a smooth sextic.
    \item $\pi_2 \colon X \to \PP^5$ is $\Bl_{\PP^2}X_3$, the blow up of a cubic fourfold along a plane.
    \item Rationality: unknown.
\end{itemize}

\subsubsection*{Semiorthogonal decompositions}
\begin{itemize}
\item 6 exceptional objects and $\Db(S_2,\alpha)$, where $S_2 \to \PP^2$ is the discriminant double cover associated to $\pi_1$ and $\alpha$ the Clifford invariant of $\pi_1$, a Brauer class of order $2$ on $S_2$ (see \cref{subsub:quad-surf-fib}).
\end{itemize}
\subsubsection*{Explanation}
The sections are defined by $\alpha\in V_3\otimes V_6^\vee=\Hom(V_6,V_3)$ and $\beta\in 
V_3^\vee\otimes S^2V_6^\vee$. They vanish at $(l,m)$ if $\alpha(m)\subset l$ and 
$\beta(l,m,m)=0$. This implies $\alpha(\beta(m),m,m)=0$, which defines a cubic fourfold
containing the plane $\PP(\Ker\alpha)$. The fiber of the projection to $\PP^5$
is a projective line over this plane, and a single point over the other points of 
the cubic fourfold. 
\end{fano}

\medskip\medskip\begin{fano}\fanoid{2-40-163-2} 
$\mZ(\PP^3 \times \PP^5,\of(1,2) \oplus \mQ_{\PP^3}(0,1))$
\label[mystyle]{40-163-2-1-23-24-1-4-1-6}

\subsubsection*{Invariants}  $h^0(-K)=40, \ (-K)^4=163, \  h^{1,1}=2, \ h^{3,1}=1, \ h^{2,2}=23$, $-\chi(T)=24$.
\subsubsection*{Description} 
\begin{itemize}
    \item $\pi_1 \colon X \to \PP^3$ is a conic bundle, ramified over a singular quintic surface with 16 ordinary double points.
    \item $\pi_2 \colon X \to \PP^5$ is the blow up $\Bl_{\PP^1} X_3$ of a general cubic fourfold along a line. 
    \item Rationality: unknown. 
\end{itemize}

\subsubsection*{Semiorthogonal decompositions}
\begin{itemize}
\item 4 exceptional objects and $\Db(\PP^3, \mC_0)$ from $\pi_1$.
\item 7 exceptional objects and the K3 category from $\pi_2$.
\item We know that the K3 category sits inside $\Db(\PP^3,\mC_0)$ as a complement to 3 exceptional objects \cite{BLMS}.
\end{itemize}

\subsubsection*{Explanation}
By \cite[Corollary 2.7]{DFT}, the rank 3 bundle $\mQ_{\PP^3}(0,1)$ cuts in $\PP^3 \times \PP^5$ the blow up of $\PP^5$ in a line. By construction a section of $\of(1,2)$ in this locus cuts a cubic in $\PP^5$ containing this $\PP^1$. 

To understand the other projection, notice that $\Bl_{\PP^1}\PP^5 \cong \PP_{\PP^3}(\of^2 \oplus \of(-1)).$ Therefore $X$ can be interpreted as  $\mZ(\PP_{\PP^3}(\of^{\oplus 2} \oplus \of(-1)),\mL^2 \otimes \of(1))$, which is a conic bundle over $\PP^3$. Applying \cref{lem:conicbundle} with $\mE= \of^{\oplus 2} \oplus \of(1)$ and $\mK=\of_{\PP^3}(1)$, we get that the discriminant is a quintic surface with sixteen ordinary double points. Notice that to a quintic surface with 16 nodes one can associate a rank 23 Hodge structure of K3 type, as explained by Huybrechts in \cite{huybrechts2021nodal}.  We highlight the (classical) fact, which is remarked in \emph{op.\ cit.}, that if one takes a line $L$ in a cubic fourfold $X$ and its associated hyperk\"ahler $F(X)$ constructed as the variety of lines in $X$, then the variety $F_L(X)$ of all lines intersecting the given one is a smooth surface of general type $D_L$ with a natural involution. The quotient of $D_L$ by this involution is exactly our 16-nodal quintic surface.
\end{fano}

\begin{rmk}
\label{rmkKuz}
The above nodal quintic will appear once again in our paper, in \cref{33-129-2-1-23-23-1-4-2-7}.  Kuznetsov suggested that it would be interesting to understand whether this hints at some birational equivalence between the corresponding fourfolds. This is investigated in the recent \cite{nodalK3}.
\end{rmk}

\medskip\medskip\begin{fano}\fanoid{2-39-161-2} 
$\mZ(\PP^5 \times \Gr(2,8),\mU_{\Gr(2,8)}^{\vee}(1,0) \oplus \of(1,1) \oplus \mQ_{\PP^5} \boxtimes \mU^{\vee}_{\Gr(2,8)})$.
\label[mystyle]{39-161-2-1-23-21-1-6-2-8}

\subsubsection*{Invariants}  $h^0(-K)=39, \ (-K)^4=161, \  h^{1,1}=2, \ h^{3,1}=1, \ h^{2,2}=23$, $-\chi(T)=21$.
\subsubsection*{Description} 
\begin{itemize}
    \item $\pi_1 \colon X \to \PP^5$ is a blow up $\Bl_{\PP^1 \times \PP^1} X_3$ of a (special) cubic fourfold along a quadric surface.
    \item $\pi_2 \colon X \to \Gr(2,8)$ is the small contraction of the plane $\Pi$ residual to the quadric surface in $X_3$ onto a singular fourfold $X^\circ$ (with one singular point).
    \item Rationality: depends on the cubic fourfold. A  cubic fourfold contains a quadric if and only if it  contains a plane, see \cref{subsubC8}.
\end{itemize}

\subsubsection*{Semiorthogonal decompositions}
\begin{itemize}
\item 7 exceptional objects and $\Db(S_2,\alpha)$ from $\pi_1$.
\end{itemize}

\subsubsection*{Explanation} The section of the rank ten bundle is given by $\alpha\in \Hom(V_8,V_6)$, and vanishes on pairs $(l, P)$ such that 
$P\subset\alpha^{-1}(l)$. The section of the rank two bundle is given by $\beta\in \Hom(V_6,V_8^\vee)$, and vanishes when the linear form $\beta(l)$ is zero on $P$. 
For $l$ general, $P$ is thus uniquely defined as the intersection of $\alpha^{-1}(l)$ with the kernel of $\beta(l)$, and we get a fivefold $Y$ birational to $\PP^5$. More precisely, $Y$ is the blow up of $\PP^5$ along a quadric surface $\QQ^2\simeq\PP^1\times\PP^1$. 

Now we consider $X$ defined in $Y$ by the additional line bundle section, which is defined by a tensor $\gamma\in \Hom(V_6,\wedge^2V_8^\vee)$. The projection 
to $\PP^5$ has a non-trivial fiber over $l$ when  $\beta(l)\in \alpha^{-1}(l)^\vee$ and $\gamma(l)\in \wedge^2\alpha^{-1}(l)^\vee$ are compatible, which can be expressed by the vanishing of $\beta(l)\wedge\gamma(l)\in  \wedge^3\alpha^{-1}(l)^\vee\simeq l^\vee$. This defines a cubic fourfold $X_3$ in $\PP^5$, containing $\QQ^2$, 
and $X$ must be the blow up of this surface in $X_3$. In particular it contains the residual plane $\Pi$ to $\QQ^2$. 

The projection to $\Gr(2,8)$ contracts the plane $\Pi$ to the point defined by $p_0=\Ker\alpha$,
and outside this plane it is an isomorphism to its image $X^\circ$, which is therefore singular 
at $p_0$. In order to understand how singular it is, we parametrize $\Gr(2,8)$ locally around $p_0=\langle e,f\rangle$ as usual, by choosing a codimension two subspace $q_0$ in $V_8$ orthogonal to this plane: any plane $p$ transverse to $q_0$ admits a unique basis of the form $e+x,f+y$ with $x,y$ in $q_0$. The map $\alpha$ sends $q_0$ isomorphically to $V_6$, and we use it to identify the two spaces. A point $\langle e+x,f+y \rangle$ is in $X^\circ$ if and only if $x,y$ are collinear and
\[
\left({ }^t\beta(e+x)\right) (x) = \left({ }^t\beta(e+x)\right) (y) = \left({ }^t\beta(f+y)\right) (x) = \left({ }^t\beta(f+y)\right) (y) = 0,
\]
\[
\left({ }^t\gamma ((e+x)\wedge (f+y))\right) (x) = \left({ }^t\gamma ((e+x)\wedge (f+y))\right) (y) = 0.
\]
At first order, the latter  conditions reduce to three linear conditions applied to both $x$ and $y$, which are in general independent and amount to restricting them to a codimension three subspace of $V_8$. Then the collinearity condition defines a codimension two Schubert variety in $\Gr(2,5)$, which is well-known to be normal, and we can conclude that $X^\circ$ is normal as well. Blowing up $\Pi$, we get a fourfold $\widetilde{X}\subset \PP^1\times\PP^5\times \Gr(2,8)$, 
and the fibers of its projection to $\PP^1$ are intersections of two quadrics in $\PP^5$. 
\end{fano}

\medskip\medskip\begin{fano}\fanoid{2-49-211-2} 
$\mZ(\PP^4 \times \PP^5,\mQ_{\PP^4}(0,1) \oplus \of(2,1))$.
\label[mystyle]{49-211-2-1-21-19-1-5-1-6}

\subsubsection*{Invariants}  $h^0(-K)=49, \ (-K)^4=211, \  h^{1,1}=2, \ h^{3,1}=1, \ h^{2,2}=21$, $-\chi(T)=19$.
\subsubsection*{Description} 
\begin{itemize}
    \item $\pi_1 \colon X \to \PP^4$ is a blow up $\Bl_{S_6}\PP^4$.
    \item $\pi_2 \colon X \to \PP^5$ is a blow up $\Bl_p X^\circ_3$, where $X^\circ_3$ is a nodal cubic fourfold with a single node at $p$.
    \item Rationality: yes.
\end{itemize}
\subsubsection*{Semiorthogonal decompositions}
\begin{itemize}
\item 5 exceptional objects and $\Db(S_6)$.
\end{itemize}
\subsubsection*{Explanation} 
Let us consider the fivefold $\mZ(\PP^4 \times \PP^5,\mQ_{\PP^4}(0,1))$. By 
\cite[Corollary 2.7]{DFT}, this can be identified with $\Bl_p \PP^5 \cong \PP_{\PP^4}(\of \oplus \of(-1))$. 
Considering $\pi_1$, we can therefore describe $X$ as $\mZ(\PP_{\PP^4}(\of \oplus \of(-1)),\mL \otimes \of_{\PP^4}(2))$. We conclude by the Cayley trick, since 
$\mZ(\PP^4,\of(2) \oplus \of(3))$ is a K3 surface of degree 6.
Considering $\pi_2$, a section of $\of(2,1)$ restricted to the fivefold $\Bl_p \PP^5$  cuts a 
cubic fourfold which is nodal at $p$, and the claim follows.  
\end{fano}

\begin{rmk} In \cite[1.3]{debarreperiods} it is shown that $F_1(X_3^\circ)$ and 
$(S_6)^{[2]}$ are birationally equivalent, and a birational equivalence is described explicitly.\end{rmk}

\medskip\medskip\begin{fano}\fanoid{2-24-86-2} 
$\mZ(\PP^4 \times \PP^5,\of(1,1) \oplus \W^3 \mQ_{\PP^4}(0,1))$.
\label[mystyle]{24-86-2-1-26-24-1-5-1-6}

\subsubsection*{Invariants} $h^0(-K)=24, \ (-K)^4=86, \  h^{1,1}=2, \ h^{3,1}=1, \ h^{2,2}=26$, $-\chi(T)=24$.
\subsubsection*{Description} 
\begin{itemize}
    \item $\pi_1 \colon X \to \PP^4$ is $\Bl_{S_{14}^{(5)}} \PP^4$.
    \item $\pi_2 \colon X \to \PP^5$ is $\Bl_{\dP_5} X_3$, where $X_3$ is a Pfaffian cubic fourfold.
    \item Rationality: yes.
\end{itemize}
\subsubsection*{Semiorthogonal decompositions}
\begin{itemize}
\item 10 exceptional objects and  $\Db(S_{14})$ from both maps. 
\end{itemize}
\subsubsection*{Explanation}

Let us first consider the fivefold $Y$ defined by the general zero locus of $\W^3 \mQ(0,1)=\mQ^\vee(1,1)$ inside $\PP^4 \times \PP^5$. We denote by $V_n$ the vector space such that $\PP^{n-1}=\PP V_n$.
The space of sections $H^0(\PP^4 \times \PP^5,\mQ^\vee(1,1))$ is then
$\W^2 V_5^\vee \otimes V_6^\vee = \Hom(V_6,\W^2 V_5^\vee)$.

We want to understand the projection  $\pi:Y \to \PP^5$.
First note that since the morphism $\of_{\PP^5}(-1) \to \W^2 V_5^\vee \otimes \of_{\PP^5}$ is generically injective,  $\pi$ is birational.
The first degeneracy locus of $\W^2 V_5$ is a Pfaffian locus of codimension 3, and  a general Pfaffian surface in $\PP^5$ is $\dP_5$,  a smooth del Pezzo surface of degree five.
The fiber of $\pi$ over a point $p$ of $\dP_5$ is the projectivization of the 3-dimensional kernel of the form corresponding to $p$. Hence the fibers of $\pi$ over $\dP_5$ are projective planes.
This implies that $Y \to \PP^5$ is the blow up of the above $\dP_5$.

Now notice that $c_1(K_{\PP^4 \times \PP^5} \otimes \det \mQ^\vee(1,1)) = -2H_1 -2H_2$. On the other hand, since $Y$ is a blow up, we have $-K_Y=-K_{\PP^5} + 2E$ for $E$ the exceptional divisor.
Therefore $\of_Y (1,1)$ is the line bundle associated to half-anticanonical sections of $Y$, hence it corresponds to a cubic fourfold (a half anticanonical divisor in $\PP^5$) plus the exceptional divisor $E$.
It follows that $X$, obtained by cutting $Y$ with a section of $\of (1,1)$, is the blow up of a special cubic fourfold $X_3$ along the $\dP_5$. Indeed, a cubic fourfold containing such a surface is Pfaffian, see \cite{Beau-Dona}.
In order to see this cubic fourfold concretely, note that there is a natural map 
$$S^2(\W^2 V_5^\vee \otimes V_6^\vee)\longrightarrow \W^4 V_5^\vee \otimes S^2V_6^\vee
\simeq V_5\otimes S^2V_6^\vee.$$
The linear section that defines $X$ is a tensor in $V_5^\vee \otimes V_6^\vee$.
Contracting with the previous one we get a tensor in $S^3V_6^\vee$, hence a cubic fourfold. 

\smallskip Now we consider the projection of $Y$ to $\PP^4$. The morphism $V_6 \to \W^3 V_5$ 
defines over $\PP^4$ a map of vector bundles $V_6 \otimes \of \to \W^3 \mQ=\mQ^\vee(1)$, and the generic 
fiber is the projective line defined by the kernel of this morphism. The exceptional fibers
are projective planes, and occur over its degeneracy locus $C$, a smooth curve whose twisted
structure sheaf is resolved by the Eagon--Northcott complex
$$0\longrightarrow S^2\mQ(-2)
\longrightarrow V_6^\vee\otimes Q(-1)\longrightarrow \W^2 V_6^\vee\otimes \of_{\PP^4} \longrightarrow \of_{\PP^4}(3)\longrightarrow \of_C(3)\longrightarrow 0.$$
The only contributions in cohomology comes from the $h^0$ of the third and fourth terms, so $\chi(\of_C(3))=35-15=20$. 
If we twist by $\of(-1)$, the only contribution comes from the fourth term and we deduce that 
$\chi(\of_C(2))=15$. After one extra twist, we get the bundle $S^2\mQ(-4)$ 
which by Borel--Weil--Bott Theorem has $h^3=5$, and we deduce that $\chi(\of_C(1))=5+5=10$. 
Finally, after one more twist, we get $\mQ(-4)$ with $h^3=1$, and $S^2\mQ(-5)$ 
which $h^3=10$, therefore $\chi(\of_C)=1-6+10=5$. These computations confirm that 
the Hilbert polynomial $\chi(\of_C(k))=5k+5$, hence that $C$ is the disjoint 
union of five projective lines. 

When we pass to the hyperplane section $X$ of $Y$, the general fiber becomes a single 
point, so the projection is birational. The exceptional fibers
are projective lines, and occur over the smooth surface $S$ defined as the degeneracy locus 
of the morphism $V_6 \otimes \of \to \mQ^\vee(1)\oplus \of(1)$. The twisted structure sheaf of $S$ 
is resolved by the Eagon--Northcott complex
$$0 \longrightarrow \mQ(-1)\oplus  \of_{\PP^4}(-1) \longrightarrow V_6^\vee\otimes \of_{\PP^4}\longrightarrow \of_{\PP^4}(4)\longrightarrow \of_S(4)\longrightarrow 0.$$
We deduce that the Hilbert polynomial of $S$ is 
$$\chi(\of_S(k))=\frac{9k^2-5k}{2}+2.$$
This means that $\of_{S}(1)^2=9$ and $\of_{S}(1).\omega_S=5$. 
Moreover, dualizing the previous complex we get 
$$0 \longrightarrow \of_{\PP^4}(-5)\longrightarrow V_6\otimes \of_{\PP^4}(-1)\longrightarrow 
\mQ^\vee\oplus  
\of_{\PP^4} \longrightarrow \omega_S\longrightarrow 0.$$
From \cref{lem:section} we deduce that $\omega_S$ admits a canonical section, vanishing exactly 
on the locus where the morphism $V_6\otimes \of_{\PP^4}(-1)\longrightarrow 
\mQ^\vee$ is not surjective. This is the same smooth curve $C$ 
that we proved to be the disjoint union of five lines $l_1,\ldots , l_5$. Moreover $\omega_S=\of(l_1+\cdots +l_5)$. 
By the genus formula, each of these lines must be a $(-1)$-curve on $S$. Finally, after contracting 
them we get a genus K3 surface $S_{min}$. We have proved:

\begin{proposition}
$S={S}_{14}^{(5)}$ is a K3 surface of genus eight blown up at five points.
\end{proposition}

Note that a K3 surface of genus eight is a section of the Grassmannian $\Gr(2,6)$ by 
a $\PP^8$, so a naive projection from five points would send $S$ to a $\PP^3$. 
That K3 surfaces of genus eight blown up at five points can be embedded in  $\PP^4$
is mentioned in \cite{Aure-Rane}; it parametrizes a 
family of lines in $\PP^5$, whose union is a ruled threefold $W\subset \PP^5$; then the 
smooth sections of $W$ by a hyperplane $\PP(H)$ are blown up of the K3 at its five points 
contained in $\Gr(2,H)$.

Note also that, if we denote by $H_1, H_2$ the hyperplane bundles pulled-back via the two
projections $\pi_1, \pi_2$, and by $E_1, E_2$ the two exceptional divisors, we can compute
the canonical divisor of $X$ as $\omega_X=-H_1-H_2=-5H_1+E_1=-3H_2+E_2$. In particular
$H_2=4H_1-E_1$, which means that the birational map $\pi_2\circ\pi_1^{-1}: \PP^4\dashrightarrow 
X_3$ is defined by the linear system of quartics containing the non-minimal K3 surface 
$S_{14}^{(5)}$. This special birational transformation already appears in the work of
Fano \cite[Section 6]{stagliano20}. In fact, we have just re-discovered from
this example the classical proof that Pfaffian cubic fourfolds are rational!
\end{fano}

\medskip\medskip\begin{fano}\fanoid{2-30-115-2-A} 
$\mZ(\Fl(1,3,8),\mQ_2^{\oplus 2} \oplus \of(1,1) \oplus \mR_2^{\vee}(1,1) )$.
\label[mystyle]{30-115-2-1-27-26-1-3-8}

\subsubsection*{Invariants}  $h^0(-K)=30, \ (-K)^4=115, \  h^{1,1}=2, \ h^{3,1}=1, \ h^{2,2}=27$, $-\chi(T)=26$.
\subsubsection*{Description} 
\begin{itemize}
    \item $\pi_1 \colon X \to \PP^7$ is birational to its image, the intersection $Y$ of three quadrics, singular along a line. 
    \item $\pi_2 \colon X \to \Gr(3,8)$ is $\Bl_{\dP_4}X_3$, the blow up of a cubic fourfold in a $\dP_4$.
    \item Rationality: depends on the cubic fourfold. A cubic fourfold contains a $\dP_4$ if and only if it  contains a plane, see \cref{subsubC8}.
\end{itemize}

\subsubsection*{Semiorthogonal decompositions}
\begin{itemize}
\item 11 exceptional objects and $\Db(S_2,\alpha)$ from $\pi_2$, see \cref{subsubC8}.
\end{itemize}

\subsubsection*{Explanation}

Over the flag manifold parametrizing flags $U_1\subset U_3\subset V_8$, of dimension $17$, 
the bundles are two copies of the rank five bundle $V_8/U_3$, the line bundle $U_1^\vee\otimes\det(U_3)^\vee$, and the rank two bundle $(U_3/U_1)^\vee\otimes U_1^\vee\otimes\det(U_3)^\vee$. Let us start with the first two: their global sections 
are two vectors in $V_8$ and the locus where such sections vanish is where $U_3$ contains them. 
So we need that $U_3$ contains a fixed two-dimensional space $V_2$: it is then parametrized 
by $\PP=\PP(V_8/V_2)\simeq\PP^5$. Taking $U_1$ into account we are reduced to the 
seven-dimensional projective bundle
$$Z=\PP(V_2\oplus \of_{\PP}(-1))\stackrel{\pi_2}{\longrightarrow} \PP.$$
Denote by $\mL$ the relative tautological line bundle. We have an exact sequence
$$0\longrightarrow \mL\longrightarrow \pi_2^*(U_2\oplus \of_{\PP}(-1)) \longrightarrow 
\mK \longrightarrow 0$$
for $\mK$ the rank two quotient bundle. Note that the relative tangent bundle 
is $\HOM(\mL,\mK)$, whose determinant is $\det(\mK)\otimes \mL^{-2}=
\mL^{-3}(-1)$, hence $K_Z=\mL^{3}(-5)$. 

Then we take into account the rank two vector bundle. Note that restricted to $Z$, $U_1$ is just 
$\mL$ and $U_3$ is $\pi_2^*(V_2\oplus \of_{\PP}(-1))$, so $U_3/U_1\simeq 
\mK$ and $(U_3/U_1)^\vee\otimes U_1^\vee\otimes\det(U_3)^\vee\simeq \mK^\vee\otimes\mL^{-1}(1)=\mK(2)$. In particular a general section 
of this bundle defines a smooth fivefold $Y\subset Z$ such that $K_Y=\mL_Y^{2}(-2)$.
Restricted to the fibers of $\pi_2$ this bundle is a copy of the quotient bundle on the projective
plane, whose global sections vanish at a single point, or everywhere. So $\pi_2$ restricted to 
$Y$ is birational. Moreover, 
the exceptional fibers are projective planes, that occur over the locus where the induced 
section of $\pi_{2,*}(\mK(2))=V_2\otimes \of_{\PP}(2)\oplus\of_{\PP}(1)$
vanishes, hence over a $\dP_4$. We conclude that $\pi_2$ is the blow up of $\PP$ in a $\dP_4$. 

Finally we take into account the line bundle $U_1^\vee\otimes\det(U_3)^\vee\simeq \mL^{-1}(1)$. A general section vanishes along a fourfold $X\subset Y$ such that 
$K_X=\mL_X(-1)$. In the fibers of $Z$ over $\PP$, note that $X$ is defined by a 
section of $\mL^{-1}(1)$ and a section of $\mK(2)$. By the tautological
exact sequence the latter vanishes on a section of the line bundle $\mL(2)$, on which the section of 
$\mL^{-1}(1)$ also vanishes if and only if the induced section of $\mL(2)
\otimes \mL^{-1}(1)= \of_{\PP}(3)$ vanishes. This means that the fiber of $\pi_2$ 
is non-trivial over a cubic fourfold, and that $X$ is the blow up of this cubic fourfold
in a $\dP_4$. 

\smallskip Now we  turn to $\pi_1$. Note that over $Z$ we have a morphism of vector bundles
from $U_1\oplus V_2$ to $U_3$, hence by duality from $U_3^\vee$ to the trivial bundle $V_2^\vee$,
and from $\det(U_3)^\vee$ to $U_1^\vee$. Therefore the sections of $U_1^\vee\otimes\det(U_3)^\vee$ and 
 $(U_3/U_1)^\vee\otimes U_1^\vee\otimes\det(U_3)^\vee$ that define $X$ induce sections of 
 $(U_1^\vee)^2$ and  $V_2^\vee\otimes (U_1^\vee)^2$ respectively. We thus get three quadrics
 in $\PP^7$ whose intersection $Y$ is the image of $\pi_1$. Obviously $Y$ is isomorphic to $X$
 outside the projective line $\ell$ where $U_1\subset V_2$. Over this line, $Z$ restricts to 
 $\PP^1\times\PP^5$, and $U_3/U_1$ restricts to $\of(1,0)\oplus\of(0,-1)$. We deduce that 
 the preimage of $\ell$ in $X$ is the intersection in $\PP^1\times\PP^5$ of three divisors
 of bidegrees $(0,1), (1,1), (1,2)$; hence the intersection in $\PP^1\times\PP^4$ of two divisors
 of bidegrees $(1,1), (1,2)$. The image in $\PP^4$ of such a threefold is a cubic threefold with
 four singular points. 
\end{fano}

\medskip\medskip\begin{fano}\fanoid{2-27-101-2} 
$\mZ(\PP^5 \times \Gr(2,4), (\mU_{\Gr(2,4)}^{\vee}(1,0))^
{\oplus 2} \oplus \of(1,1)).$
\label[mystyle]{27-101-2-1-23-21-1-6-2-4}

\subsubsection*{Invariants}  $h^0(-K)=27, \ (-K)^4=101, \  h^{1,1}=2, \ h^{3,1}=1, \ h^{2,2}=23$, $-\chi(T)=21$.
\subsubsection*{Description} 
\begin{itemize}
    \item $\pi_1 \colon X \to \PP^5$ is $\Bl_{\Sigma_4}X_3$, the blow up of a cubic fourfold along a rational quartic scroll.
    \item $\pi_2 \colon X \to \Gr(2,4)$ is $\Bl_{{S}_{14}^{(1)}}\Gr(2,4)$, the blow up of a quadric along a K3 surface of genus 8 blown up in one point.
    \item Rationality: yes, as a cubic containing a rational quartic scroll is rational.
\end{itemize}

\subsubsection*{Semiorthogonal decompositions}
\begin{itemize}
\item 7 exceptional objects and $\Db(S_{14})$ from both maps. Indeed the cubic $X_3$ lies in $\mC_{14}$ since it contains a rational quartic scroll, see \cref{subsubC14}.
\end{itemize}

\subsubsection*{Explanation}
Let us start with the fivefold $Y$ defined by two global sections of the same two rank two 
bundles, hence a tensor $\alpha\in V_2\otimes V_6^\vee\otimes V_4= \Hom(V_2\otimes V_6,V_4)$. 
So $Y\subset\PP^5\times \Gr(2,4)$
parametrizes the pairs $(l \subset V_6, P\subset V_4)$ such that $\alpha(V_2\otimes l)\subset P$. 
Generically, there must be equality, so the projection to $\PP^5$ is birational. The exceptional fibers are projective planes, and they appear above the locus where the morphism
$V_2\otimes \of(-1)\longrightarrow V_4 \otimes \of$ defined by $\alpha$ drops rank. This locus is a 
smooth surface. More precisely,
a rational quartic scroll $\Sigma_4$, and $Y$ is the blow up of $\PP^5$ along this surface.

Now let us consider the projection to $\Gr(2,4)$. Considering $\alpha$ as a morphism from 
$V_2\otimes V_4^\vee$ to $V_6^\vee$, we can see that the preimage of a plane $P\in V_4$
is the (projectivized) linear space defined by the image of $V_2\otimes P^\perp$. 
The generic fiber is thus a projective line. Moreover the exceptional fibers are projective
planes, which occur above the locus where the morphism of vector bundles $$V_2\otimes \mQ^\vee\rightarrow V_6^\vee \otimes \of$$
drops rank, which is a smooth curve $C$. The twisted structure sheaf of this curve is resolved by the Eagon--Northcott complex on $\Gr(2,4)$
$$0\longrightarrow S^2(V_2\otimes \mQ^\vee)
\longrightarrow V_{12}\otimes \mQ^\vee\longrightarrow V_{15} \otimes \of_{\Gr(2,4)} \longrightarrow \of_{\Gr(2,4)}(2)\longrightarrow \of_C(2)\longrightarrow 0,$$
where as usual $V_k$ is a $k$-dimensional vector space. Here $S^2(V_2\otimes \mQ^\vee)=\of_{\Gr(2,4)}(-1)\oplus V_3\otimes S^2\mQ^\vee$ is acyclic by Borel--Weil--Bott Theorem, as well as $\mQ^\vee$, so 
$\chi(\of_C(2))=20-15=5.$ If we twist by $\of_{\Gr(2,4)}(-1)$, 
the bundle $S^2\mQ^\vee(-1)$ has a one-dimensional second cohomology group, yielding 
$\chi(\of_C(1))=6-3=3.$ After another twist, all the bundles in the complex
become acyclic and therefore $\chi(\of_C)=1$. This is more than enough to imply
that the Hilbert polynomial of $C$ is $\chi(\of_C(k))=2k+1$, which implies that
$C$ must be a smooth conic $C_2$. We end up with the following diagram
\begin{equation*}
\xymatrix{
& E \ar[dl]_{\PP^2} \ar@{^{(}->}[r] & Y \ar[dl]^{\PP^0} \ar[dr]_{\PP^1} & F \ar[dr]^{\PP^2} \ar@{_{(}->}[l]\\
\Sigma_4 \ar@{^{(}->}[r] &  \PP^5 & & \Gr(2,4) &  C_2\ar@{_{(}->}[l],
}
\end{equation*}
\begin{rmk} How to obtain a smooth conic from the tensor $\alpha$? First observe that 
its $6\times 6$ minors defined a tensor $\alpha^{(6)}\in \Hom(\wedge^6(V_2\otimes V_4),\wedge^6 V_6)$.
Once volume forms have been fixed, this is just $\wedge^2(V_2\otimes V_4)$, which has a natural
map $p$ to $S^2V_2\otimes \wedge^2 V_4\simeq \Hom(S^2V_2^\vee,\wedge^2 V_4)$. So the 
image of $p(\alpha^{(6)})$ will in general define a projective plane in $\PP(\wedge^2 V_4)$,
whose intersection with $\Gr(2,4)$ is a conic canonically defined by $\alpha$. 
\end{rmk}

Now we turn to the fourfold $X$ obtained as a hyperplane section of $Y$. This hyperplane 
is defined by a tensor $\beta\in V_6^\vee\otimes \wedge^2 V_4^\vee$. If we fix a basis 
of $V_2$ and denote by $\alpha_1, \alpha_2$ the components of $\alpha$ in this basis, $X$ parametrizes 
the pairs $(l, P)$ such that $P$ contains $\alpha_1(l)$ and $\alpha_2(l)$, 
and $\beta (l, p_1, p_2)=0$ for any basis $(p_1,p_2)$ on $P$. Outside $\Sigma_4$, 
$\alpha_1(l)$ and $\alpha_2(l)$ are such a basis of $P$, hence the condition 
$\beta (l, \alpha_1(l), \beta_2(l))=0$ defines a cubic in $\PP^5$. This implies that
$X$ is the blow up of a cubic fourfold $X_3$ along the rational quartic scroll $\Sigma_4$. In particular, $X_3$ is Pfaffian \cite{Beau-Dona} and hence rational.

The projection of $X$ to $\Gr(2,4)$ is also birational. Its exceptional locus is now
the degeneracy locus of the morphism 
$$V_2\otimes \mQ^\vee\oplus\of_{\Gr(2,4)}(-1)\rightarrow V_6^\vee \otimes \of_{\Gr(2,4)}$$
defined by $\alpha$ and $\beta$. This locus is a smooth surface $S$ 
whose structure sheaf is resolved by the Eagon--Northcott complex
$$0\longrightarrow 
V_2\otimes \mQ^\vee\oplus\of_{\Gr(2,4)}(-1)\rightarrow V_6^\vee \otimes \of_{\Gr(2,4)}
\longrightarrow \of_{\Gr(2,4)}(3)\longrightarrow \of_S(3)\longrightarrow 0.$$
Cohomologically there is no difference with the previous case, so we can conclude in
the same way that the Hilbert polynomial of $S$ is again  
$\chi(\of_S(k))=5k^2-k+2.$

\begin{proposition}\label{prop:C2}
$S=S_{14}^{(1)}$ is a K3 surface of genus eight blown up at a single point.
\end{proposition}

\begin{rmk}
Of course the cubic $X$ should be Pfaffian and the K3 surface $S_{14}$ should be its HP-dual.\end{rmk}

\begin{proof}
By dualizing the Eagon--Northcott complex, we get 
$$0\longrightarrow 
\of_{\Gr(2,4)}(-4)\longrightarrow V_6(-1)\longrightarrow 
V_2\otimes \mQ(-1)\oplus\of_{\Gr(2,4)}\longrightarrow \omega_S\longrightarrow 0.$$
We deduce that $\omega_S$ admits a canonical section, vanishing exactly on the curve $D$ where the 
morphism $V_6(-1)\rightarrow V_2\otimes \mQ(-1)$ is not surjective. By the Thom--Porteous formula,
the class of $D$ in the Chow ring of $\Gr(2,4)$ is 
$$[D]=s_3(V_2\otimes \mQ - V_6\otimes \of_{\Gr(2,4)})=s_3(V_2\otimes \mQ)=2\sigma_{21},$$
since the Segre class $s(U_2\otimes \mQ)=s(\mQ)^2$ and $s(\mQ)=1+\sigma_1+\sigma_2$
is the sum of the special
Schubert classes. In words, $[D]$ is twice the class of a line. Moreover, 
the  structure sheaf of $D$ is resolved by another Eagon--Northcott complex, 
$$0\longrightarrow S^2(V_2\otimes \mQ)^\vee(-2) \longrightarrow V_{12}\otimes \mQ^\vee(-2)
\longrightarrow V_{15}\otimes \of_{\Gr(2,4)}(-2)\longrightarrow 
\of_{\Gr(2,4)}\longrightarrow \of_D\longrightarrow 0.$$
Note that $S^2(V_2\otimes \mQ)^\vee(-2)=S^2U_2\otimes S^2\mQ^\vee(-2)\oplus \of_{\Gr(2,4)}(-3)$ 
is acyclic by Borel--Weil--Bott, as well as  $\mQ^\vee(-2)$ and of course  $\of_{\Gr(2,4)}(-2)$.
We deduce that $h^0(\of_D)=1$, hence that $D$ must be connected, and finally that this must 
be a smooth conic. 
By the genus formula, since $\omega_S=\of_S(D)$ we conclude that $D$ is a $(-1)$-curve on $S$, 
after contracting which we get a genuine K3 surface $S_{min}$ of degree $14$.
\let\oldqedbox\qedsymbol
\newcommand{\twoqedbox}{$\square$ \, \oldqedbox}
\renewcommand{\qedsymbol}{\twoqedbox}
\end{proof}
\let\qed\relax
\end{fano}

\medskip\medskip\begin{fano}\fanoid{3-20-63-2} 
$\mZ(\PP^1 \times \PP^1 \times \PP^5,\of(0,0,3) \oplus \of(0,1,1) \oplus \of(1,0,1))$.
\label[mystyle]{20-63-3-1-38-36-1-2-1-2-1-6}

\subsubsection*{Invariants}  $h^0(-K)=20, \ (-K)^4=63, \  h^{1,1}=3, \ h^{3,1}=1, \ h^{2,2}=38$, $-\chi(T)=36$.
\subsubsection*{Description} 
\begin{itemize}
    \item $\pi_{12} \colon X \to \PP^1 \times \PP^1$ is a $\dP_3$-fibration.
    \item $\pi_{13},\pi_{23} \colon X \to \PP^1 \times \PP^5$ are  blow ups $\Bl_{\dP_0}Y$, where $\dP_0$ is a blow up of a cubic surface at three points and $Y$ is \cref{36-144-2-1-28-28-1-2-1-6}.
    \item Rationality: unknown, since the cubic fourfold is general. 
\end{itemize}
\subsubsection*{Semiorthogonal decompositions}
\begin{itemize}
\item 4 exceptional objects and an unknown category from $\pi_{12}$.
\item 24 exceptional objects and the K3 category of $X_3$ from $\pi_{13}$ or $\pi_{23}$.
\end{itemize}

\subsubsection*{Explanation}
One can apply \cref{lem:blowupcodim2} twice to understand $\pi_{13}$ and $\pi_{23}$ as the composition of two blow ups of cubic surfaces. One checks that the strict transform of the second cubic surface is its blow up at three extra points (which are its intersection points with the line $l$ at the intersection of the two $\PP^3$'s).

In order to understand $\pi_{12}$, we reinterpret $X$ as $\mZ(\Fl(1,4,6),\mQ_2^{\oplus 2} \oplus \of(0,1)^{\oplus 2} \oplus \of(3,0))$. The latter is nothing but $\mZ(\PP_{\PP^1 \times \PP^1}(\of^{\oplus 2} \oplus \of(-1,0) \oplus \of(0,-1)),\mL^{\otimes 3})$, which is obviously a $\dP_3$-bundle over $\PP^1 \times \PP^1$.
\end{fano}

\medskip\medskip\begin{fano}\fanoid{3-24-81-2} 
$\mZ(\PP^1 \times \PP^2 \times \PP^5,\of(0,0,3) \oplus \mQ_{\PP^2}(0,0,1) \oplus \of(1,1,0))$.
\label[mystyle]{24-81-3-1-1-30-31-1-2-1-3-1-6}

\subsubsection*{Invariants}  $h^0(-K)=24, \ (-K)^4=81, \  h^{1,1}=3, \  h^{2,1}=1, \ h^{3,1}=1, \ h^{2,2}=30$, $-\chi(T)=31$.
\subsubsection*{Description} 
\begin{itemize}
    \item $\pi_{12} \colon X \to \PP^1 \times \PP^2$ factors as the blow up of a $\dP_3$-fibration $T \to \PP^2$ in a general fiber which is a cubic surface.
    \item $\pi_{13} \colon X \to \PP^1 \times \PP^5 $ is the blow up $\Bl_{\Sigma} Y$, where $\Sigma$ is a ruled surface and $Y$ is as in \cref{36-144-2-1-28-28-1-2-1-6}.  
    \item $\pi_{23} \colon X \to \PP^2 \times \PP^5$ is a blow up of \cref{28-99-2-1-1-23-29-1-3-1-6} along a cubic surface. 
    \item Rationality: unknown, since the cubic fourfold is general.
\end{itemize}

\subsubsection*{Semiorthogonal decompositions}
\begin{itemize}
\item 12 exceptional objects and an unknown category from $\pi_{12}$ or $\pi_{23}$.
\item 16 exceptional objects and the K3 category of $X_3$ from $\pi_{13}$.
\end{itemize}

\subsubsection*{Explanation} The projection to $\PP^2\times \PP^5$ is  birational outside a point of $\PP^2$, and blows up its image $Y$ along a cubic surface. The projection of $Y$ to $\PP^5$ is the blow up of
a cubic fourfold along a plane section, hence an elliptic curve 
contained in the cubic surface. 
(Notice that this Fano fourfold can be described as well by applying \cref{lem:blowgrass} to \cref{28-99-2-1-1-23-29-1-3-1-6}.)

The projection to $\PP^1 \times \PP^2$ can be understood by rewriting $\Bl_{\PP^2} \PP^5$ as $Z:=\PP_{\PP^2}(\of^{\oplus 3} \oplus \of(-1))$. A section of $\of_Z(3)$ cuts the latter in a fibration in $\dP_3$, which we call $Z'$. We thus obtain a zero locus  $\mZ(\PP^1 \times Z',\of_{\PP^1}(1) \otimes \pi^*\of_{\PP^2}(1))$. By \cref{lem:blowgrass} this is the blow up of $Z'$ in a general fiber, i.e.\ a cubic surface.

Finally, the fibers of the projection to $\PP^1\times X_3$ are defined by morphisms $\of(0,-1)\ra V_3$
and $\of(-1,0)\ra V_3^\vee$, whose compatibility is verified over a section of $\of(1,1)$. So the 
image of $\pi_{13}$ is isomorphic to the blow up of $X_3$ along a cubic surface $\Sigma$. Moreover, the 
fibers of $\pi_{13}$ do not reduce to a single point when the first morphism $\of(0,-1)\ra V_3$ vanishes, that is over a copy of $\PP^1\times E$, where $E$ is an elliptic curve in $\Sigma$. 
\end{fano}

\medskip\begin{fano}\fanoid{3-27-99-2} 
$\mZ(\PP^1 \times \PP^2 \times \PP^5,\of(0,1,2) \oplus \mQ_{\PP^2}(0,0,1) \oplus \of(1,0,1))$.
\label[mystyle]{27-99-3-1-30-27-1-2-1-3-1-6}

\subsubsection*{Invariants}  $h^0(-K)=27, \ (-K)^4=99, \  h^{1,1}=3, \ h^{3,1}=1, \ h^{2,2}=30$, $-\chi(T)=27$.
\subsubsection*{Description} 
\begin{itemize}
    \item $\pi_{12} \colon X \to \PP^1 \times \PP^2$ factors as the blow up $\Bl_{\dP_2}Z$, where $Z \to \PP^2$ is a quadric surface bundle over $\PP^2$ ramified along a sextic from \cref{45-192-2-1-22-19-1-3-1-6}.
    \item $\pi_{13} \colon X \to \PP^1 \times \PP^5$ is a blow up $\Bl_{\dP_8}Y$, where $Y$ is a blow up of a cubic surface in a cubic fourfold containing a plane, and the $\dP_8$ is the strict transform of this plane.
     \item $\pi_{23} \colon X \to \PP^2 \times \PP^5$ is a blow up $\Bl_{\dP_2}W$, where $W$ is the Fano fourfold \cref{45-192-2-1-22-19-1-3-1-6}.
    \item Rationality: unknown
\end{itemize}
\subsubsection*{Semiorthogonal decompositions}
\begin{itemize}
\item 16 exceptional objects and $\Db(S_2,\alpha)$ from each map, where $S_2$ and $\alpha$ are the K3 surface and the Brauer class either from \cref{45-192-2-1-22-19-1-3-1-6} or from the cubic containing a plane, see \cref{subsubC8}.
\end{itemize}

\subsubsection*{Explanation}
Let us start by describing $\pi_{23}$. By \cref{lem:blowupcodim2}, $X$ is the blow up of a fourfold,  from \cref{45-192-2-1-22-19-1-3-1-6}, say $Y$, in a surface which is a complete intersection of $Y$ with two extra $(0,1)$ divisors from $\PP^2 \times \PP^5$. In particular, this surface is described in $\PP^2 \times \PP^3$ as $\mZ(\mQ_{\PP^2}(0,1) \oplus \of(1,2))$. The latter is identified with the blow up of a 
$\dP_3$ in a point. Note that such a cubic surface is a double hyperplane section of the cubic fourfold in $\PP^4$, and hits the blown up plane in exactly one point. This gives the description of $\pi_{13}$ as well.

In order to understand $\pi_{12}$, let us reinterpret this variety as $\mZ(\PP^1 \times \Fl(1,4,6),\mQ_2^{\oplus 3} \oplus \of(0;2,1) \oplus \of(1;1,0))$. The first four bundles cut $\mZ(\PP^1 \times \PP_{\PP^2}(\of^{\oplus 3} \oplus \of(-1)),\of_{\PP^1}(1) \boxtimes \mL \oplus \of_{\PP^2}(1)\boxtimes \mL^{\otimes 2} )$. If we just consider $Z$ as $\mZ(\PP_{\PP^2}(\of^{\oplus 3} \oplus \of(-1)),\mL^{\otimes 2} \boxtimes \of_{\PP^2}(1))$ this can be identified as in 
\cref{45-192-2-1-22-19-1-3-1-6} with a quadric surface bundle over $\PP^2$ with ramification divisor a smooth sextic. By \cref{lem:blowupcodim2} the latter is then the blow up in the complete intersection of two copies of $\mL$, that is the double cover of $\PP^2$ ramified over a sextic. The result follows.
\end{fano}

\medskip\begin{fano}\fanoid{3-34-134-2} 
\label[mystyle]{34-134-3-1-28-25-1-2-1-2-6} $\mZ(\PP^1 \times \Fl(1,2,6),\mQ_2 \oplus \of(0;1,2) \oplus \of(1;0,1)).$

\subsubsection*{Invariants}  $h^0(-K)=34, \ (-K)^4=134, \  h^{1,1}=3, \ h^{3,1}=1, \ h^{2,2}=28$, $-\chi(T)=25$.
\subsubsection*{Description} 
\begin{itemize}
    \item $\pi_{12} \colon X \to \PP^1 \times \PP^5$ is a blow up $\Bl_{\dP_3}\Bl_p X^\circ_3$, where $X^\circ_3$ is a nodal cubic fourfold and $\Bl_pX_3^\circ$ is \cref{49-211-2-1-21-19-1-5-1-6}.
    \item $\pi_{13} \colon X \to \PP^1 \times \Gr(2,6)$ is a  blow up $\Bl_{S_6^{(6)}}\Bl_{\PP^2}\PP^4$, where $S_6^{(6)}$ is a sextic K3 surface blown up in 6 points.
     \item $\pi_{23} \colon X \to \Fl(1,2,6)$ is a blow up $\Bl_{\dP_3}\Bl_{S_6}\PP^4$, with $S_6$ a sextic K3 from \cref{49-211-2-1-21-19-1-5-1-6}. 
    \item Rationality: yes.
\end{itemize}

    \subsubsection*{Semiorthogonal decompositions}
\begin{itemize}
\item 14 exceptional objects and $\Db(S_6)$ from each map.
\end{itemize}

\subsubsection*{Explanation}
In order to understand the first two projections, we write down $\mZ(\Fl(1,2,6),\mQ_2)$ as $\PP_{\PP^4}(\of \oplus \of(-1)) \cong \Bl_p \PP^5.$ We deduce that the projection to $\PP^1 \times \PP^5 $ maps $X$ to a $(1,1)$ divisor in $\PP^1 \times \Bl_p\Tilde{X}_3$, which in turn blows up a cubic surface.

Starting from the same description we can also understand the second projection: the image is a $(1,1)$-divisor in $\PP^1 \times \PP^4$, i.e.\ the blow up of $\PP^4$ in a plane. Using the Cayley trick 
we can see that we further blow up the strict transform of a sextic K3 surface in $\PP^4$, which is
the blow up of this sextic in six extra points.

 Finally, without the first factor and the last bundle, we know that we get the blow up of $\PP^4$ along a K3 surface of degree 6 by the Cayley trick. The extra
factor yields another blow up, along the strict transform of a plane, which is a $\dP_3$. 
\end{fano}

\medskip\begin{fano}\fanoid{3-30-113-2} 
\label[mystyle]{30-113-3-1-30-28-1-2-1-4-1-6} $\mZ(\PP^1 \times \PP^3 \times \PP^5,\of(0,1,2) \oplus \mQ_{\PP^3}(0,0,1) \oplus \of(1,1,0))$.
\subsubsection*{Invariants}  $h^0(-K)=30, \ (-K)^4=113, \  h^{1,1}=3, \ h^{3,1}=1, \ h^{2,2}=30$, $-\chi(T)=28$.
\subsubsection*{Description} 
\begin{itemize}

    \item $\pi_{12} \colon X \to \PP^1 \times \PP^3$ has for image the blow up of $\PP^3$ along a line and is a blow up  of the conic bundle described in \cref{40-163-2-1-23-24-1-4-1-6}.
\item $\pi_{13} \colon X \to \PP^1 \times \PP^5$ is a blow up of a general cubic fourfold along 
a cubic surface, i.e. \cref{36-144-2-1-28-28-1-2-1-6}, and then a quadratic surface. 
    \item $\pi_{23} \colon X \to \PP^3 \times \PP^5$ is a blow up $\Bl_{\dP_3}Y$, where $Y$ is \cref{40-163-2-1-23-24-1-4-1-6}. 
     
    \item Rationality: unknown, since the cubic fourfold is general.
\end{itemize}
\subsubsection*{Semiorthogonal decompositions}
\begin{itemize}
\item 16 exceptional objects and the K3 category of $X_3$ from the three maps.
\end{itemize}

\subsubsection*{Explanation}
Let $Y \subset \PP^3 \times \PP^5$ be a fourfold of type \cref{40-163-2-1-23-24-1-4-1-6}.
To understand $\pi_{23}$ we apply  \cref{lem:blowupcodim2} to $\PP^1 \times Y$. The intersection of two divisors of degree $(1,0)$ on $Y \subset \PP^3 \times \PP^5$ cuts a cubic surface that contains the line already blown up.

Alternatively, the other birational description for $Y$ was a conic bundle over $\PP^3$ ramified over a singular quintic. This suggests that $X$ can be alternatively described as $\mZ(\PP^1 \times \Fl(1,3,6),\mQ_2^{\oplus 2} \oplus \of(0;2,1) \oplus \of(1;1,1))$. In particular thanks to \cref{lem:blowupcodim2} applied in this context we can see that this fourfold is the blow up of the said conic bundle in a cubic surface. This yields  the description of $\pi_{12}$.

Finally, let $Z$ denote the image of $X$ in $\PP^1\times \PP^5$. The  projection 
of $Z$ to $\PP^5$ is the 
blow up of a general cubic fourfold along a surface section $\Sigma$. The projection of $X$ to $Z$ is 
the blow up of $\PP^1\times l$, where $l$ is a line in $\Sigma$. 
\end{fano}

\medskip\begin{fano}\fanoid{3-35-141-2} 
$\mZ(\PP^2_1 \times \PP^2_2 \times \PP^5,\mQ_{\PP^2_1}(0,0,1) \oplus \mQ_{\PP^2_2}(0,0,1) \oplus \of(1,1,1)).$
\label[mystyle]{35-141-3-1-23-18-1-3-1-3-1-6}

\subsubsection*{Invariants}  $h^0(-K)=35, \ (-K)^4=141, \  h^{1,1}=3, \ h^{3,1}=1, \ h^{2,2}=23$, $-\chi(T)=18$.
\subsubsection*{Description} 
\begin{itemize}
    \item $\pi_{12} \colon X \to \PP^2 \times \PP^2$ is a blow up $\Bl_{S_{14}} (\PP^2 \times \PP^2)$, where the K3 surface $S_{14}$ has Picard rank 2.
    \item $\pi_{13},\pi_{23} \colon X \to \PP^2 \times  \PP^5$ is a blow up along a plane $\Pi$ of a special \cref{45-192-2-1-22-19-1-3-1-6}, which is in turn a blow up $\Bl_{\Pi'}X_3$ of a cubic fourfold containing a second plane $\Pi'$, skew to $\Pi$.
    \item Rationality: yes.
\end{itemize}

\subsubsection*{Semiorthogonal decompositions}
\begin{itemize}
\item 9 exceptional objects and $\Db(S_{14})$ from $\pi_{12}$.
\item 9 exceptional objects and $\Db(S_2)$ from $\pi_{13}$ or  $\pi_{23}$ (the Brauer class vanishes when 
$X$ contains two skew planes).
\end{itemize}

\subsubsection*{Explanation} Let $A_3$ and $B_3$ be three-dimensional vector spaces such that $\PP^2_1=\PP (A_3)$ and $\PP^2_2=\PP (B_3)$. The first two sections are defined by morphisms $\alpha\in \Hom(V_6,A_3)$ and 
$\beta\in \Hom(V_6,B_3)$, and their zero locus is the 
set of triples $(x,y,z)$ of lines such that $\alpha(z)\subset x$ and $\beta(z)\subset y$. Generically 
$x$ and $y$ are determined by $z$ so the projection to $\PP^5$ is birational: it is the blow up of the two
skew planes $\PP(\Ker\alpha)$ and $\PP(\Ker\beta)$. 

The extra section of the line bundle $\of(1,1,1)$ is defined by a tensor 
$\gamma\in A_3^\vee\otimes B_3^\vee\otimes V_6^\vee$.
It induces a cubic $\gamma(\alpha(z),\beta(z),z)$ which gives an equation of 
the image of $X$ in $\PP^5$: this image is a generic cubic fourfold containing the projective planes 
$\PP(\Ker\alpha)$ and $\PP(\Ker\beta)$, and the claim follows. 

Note that we can identify $A_3$ and $B_3$ with subspaces of $V_6$ in transverse position, in such a way that $\alpha$ and $\beta$ are the associated projections. 
Given $x$ and $y$, $z$ must belong to the pencil they 
generate, on which $\gamma$ cuts a unique point in 
general. It cuts out a whole projective line when 
the induced morphism $\of(-1,-1)\rightarrow 
(\of(-1,0)\oplus \of(0,-1))^\vee$ vanishes, which happens along a K3 surface of Picard number two, intersection of two divisors of bidegrees 
$(1,2)$ and $(2,1)$ (degree $14$ in the Segre embedding, and double covers of the two $\PP^2$'s). 
We conclude that the projection to $\PP^2\times \PP^2$ is the blow up of this K3 surface. 

The rationality of this example follows from the well-known result about the rationality of a smooth cubic fourfold containing two skew planes, see e.g.\ \cite[1.2]{hassett2016cubic}.
\end{fano}

\medskip\begin{fano}\fanoid{3-42-176-2} 
\label[mystyle]{42-176-3-1-23-18-1-3-1-5-1-6} $\mZ(\PP^2 \times \PP^4 \times \PP^5,\mQ_{\PP^4}(0,0,1) \oplus \of(1,1,1) \oplus \mQ_{\PP^2}(0,1,0))$.
\subsubsection*{Invariants}  $h^0(-K)=42, \ (-K)^4=176, \  h^{1,1}=3, \ h^{3,1}=1, \ h^{2,2}=23$, $-\chi(T)=18$.
\subsubsection*{Description} 
\begin{itemize}
    \item $\pi_{12} \colon X \to \PP^2 \times \PP^4$ is a blow up $\Bl_{S_6} Y$, where $Y$ is a blow up $\Bl_{\PP^1} \PP^4$, and $S_6$ is a sextic K3 surface containing this $\PP^1$.
    \item $\pi_{13} \colon X \to \PP^2 \times \PP^5$ is a blow up of $\Bl_\Pi X_3^0$, $X_3^\circ \subset \PP^5$ is a cubic fourfold with one node, $\Pi$ is a plane containing the node, and $\pi_{13}$ is the blow  up of the fiber over the node, which is another plane.
     \item $\pi_{23} \colon X \to \PP^4 \times \PP^5$ is a blow up \cref{49-211-2-1-21-19-1-5-1-6} of the same nodal cubic $X^\circ_3$ at its node, then blown up along the preimage of the plane, which is a Hirzebruch surface $\Sigma_1$. 
    \item Rationality: yes.
\end{itemize}

    \subsubsection*{Semiorthogonal decompositions}
\begin{itemize}
\item 9 exceptional objects and $\Db(S_6)$ from each map.
\end{itemize}

\subsubsection*{Explanation} The sections that cut out $X$ are given by tensors $\alpha\in \Hom(V_6,V_5)$, $\beta\in \Hom(V_5,V_3)$, and $\gamma\in V_3^\vee\otimes V_5^\vee\otimes V_6^\vee$. 
The fourfold $X$ parametrizes the triples of lines $(x,y,z)$ in $\PP^2\times\PP^4\times \PP^5$
such that $$ \alpha(z)\subset y, \qquad \beta(y)\subset x, \qquad \gamma(x,y,z)=0. $$

We start with the description of $\pi_{12}$. 
The set of pairs $(x,y)$ such that $\beta(y)\subset x$ is the blow up of $\PP^4$ along the line $l=\PP(\Ker\beta)$. The extra point $z$ must then belong to $\alpha^{-1}(y)\simeq z_0\oplus y$,
where $z_0$ denotes the kernel of $\alpha$; this defines a projective line on which $\gamma$ cuts a 
unique point in general. So the projection of $X$ to $X'=\Bl_l\PP^4$ is birational, and 
its exceptional locus $S$ is given by the vanishing of the induced morphism $\of_{X'}(-1,-1)\rightarrow (z_0\oplus \of_{X'}(0,-1))^\vee$. This surface $S$
is isomorphic to its projection in $\PP^4$, which is a general K3 surface of degree six 
containing  the line $l$.  

\smallskip 
Now we turn to $\pi_{13}$.
Let $\Pi=\PP(\Ker(\beta\circ\alpha))$, a projective plane containing $z_0$. 
The set of pairs $(x,z)$ such that $\beta\circ\alpha(z)\subset x$ and $\gamma(x,\alpha(z),z)=0$
is a blow up $Y$ along $\Pi$ of the cubic $X_3^\circ$ of equation $\gamma(\beta\circ\alpha(z),\alpha(z),z)=0$,
which is a general cubic fourfold containing $\Pi$ and nodal at the point $z_0\in\Pi$.
(Note that the tangent cone to the cubic at $z_0$ is a quadric, in which the degree three 
part of the equation cuts the surface $S$, with its line $l$ coming from  $\Pi$.) 
The fiber $\Theta$  of $z_0$ in $Y$ is $\PP^2$, while the other non-trivial fibers are projective lines. 
The extra
point $y$ is uniquely determined as $ \alpha(z)$, except if the latter is zero, that is, 
over the projective plane $\Theta$ that is contracted to $z_0$. The fiber over $(x,z_0)$ is defined by the
conditions that $y$ belongs to the projective plane $\PP(\beta^{-1}(x))$ and $\gamma(x,y,z_0)=0$; 
this defines a projective line and $X$ must be the blow up of $\Theta$.

\smallskip
Finally we describe $\pi_{23}$.
The set of pairs $(y,z)$ such that $ \alpha(z)\subset y$ is the blow up of $\PP^5$ at $z_0$. With the additional condition that $\gamma (\beta(y),y,z)=0$ we get the blow up $Z$ of $X_3^\circ$ at its node $z_0$. 
The exceptional divisor is the quadric $\QQ^3 \subset\PP^4$ with equation $\gamma(\beta(y),y,z_0)=0$.
The extra point $x$ is uniquely determined as $\beta(y)$, except if the latter is zero, which
happens if $y$ is on the line $l$, and $z$ belongs to the projective line $\alpha^{-1}(y)$;
in other words, if $(y,z)$ belongs to the blow up $\Sigma_1$ of $\Pi$ at $z_0$, then $x$
 belongs to the projective line cut out by $\gamma$. So $X$ must be the blow up of $Z$ along 
 $\Sigma_1$. 
\end{fano}

\begin{rmk} A local computation allows us to check that the fourfold denoted by $Y$ in the previous proof is smooth. 
We thus get another Fano fourfold of K3 type, which is not the type considered in this paper: it is not in our list
of zero loci of sections of vector bundles. \end{rmk}

\medskip\begin{fano}\fanoid{3-37-149-2} 
\label[mystyle]{37-149-3-1-24-21-1-3-1-4-1-6} $\mZ(\PP^2 \times \PP^3 \times \PP^5,\mQ_{\PP^3}(0,0,1) \oplus \of(1,0,2) \oplus \mQ_{\PP^2}(0,1,0))$.
\subsubsection*{Invariants}  $h^0(-K)=37, \ (-K)^4=149, \  h^{1,1}=3, \ h^{3,1}=1, \ h^{2,2}=24$, $-\chi(T)=21$.
\subsubsection*{Description} 
\begin{itemize}
    \item $\pi_{12} \colon X \to \PP^2 \times \PP^3$ is a conic bundle over $\Bl_p \PP^3$.
    \item $\pi_{13} \colon X \to \PP^2 \times \PP^5$ is a blow up $\Bl_{\QQ^2} Y$, where $Y$ is the Fano fourfold \cref{45-192-2-1-22-19-1-3-1-6}, 
    that is $\Bl_{\PP^2} X_3$. 
     \item $\pi_{23} \colon X \to \PP^3 \times \PP^5$ is a blow up $\Bl_{\PP^2} Z$, where $Z=\Bl_{\PP^1} X_3$ is the Fano fourfold \cref{40-163-2-1-23-24-1-4-1-6}.
    \item Rationality: not known. Same as the cubic fourfold containing a plane.
\end{itemize}  
    \subsubsection*{Semiorthogonal decompositions}
\begin{itemize}
\item 6 exceptional objects and $\Db(\Bl_p\PP^3,\mC_0)$ from $\pi_{12}$.
\item 10 exceptional objects and $\Db(S_2,\alpha)$ from $\pi_{13}$ and $\pi_{23}$.
\end{itemize}

\subsubsection*{Explanation}

The sections of the three bundles are given by tensors
$\alpha\in \Hom(V_6,V_4)$, $\beta\in \Hom(V_4,V_3)$ and  $\gamma\in V_3^\vee\otimes S^2V_6^\vee$. The fourfold $X$ parametrizes the triples of lines $(x,y,z)$ 
such that $$\alpha(z)\subset y, \qquad \beta(y)\subset x, \qquad \gamma(x,z,z)=0.$$

We start with the description of $\pi_{12}$. 
The set  of pairs $(x,y)$ such that $\beta(y)\subset x$ is the blow up of $\PP^3$ at the point $y_0=\Ker(\beta)$. The extra point $z$ must be contained in $\alpha^{-1}(z)$, which defines a 
projective plane on which $\gamma$ cuts out a conic. So $X$ is a conic bundle over 
$X'=\Bl_{y_0}\PP^3$. In other words, $\alpha$ can be seen as a morphism of vector bundles over $\PP^3$, $\alpha: V_6 \otimes \of_{\PP^3} \to \mQ_{\PP^3}$, whose rank do not drop for dimension reasons, and whose kernel is exactly $\of(-1) \oplus \of^{\oplus 2}$. This is indeed a conic bundle defined by a morphism $\of_{X'}(-1,0)\rightarrow S^2(\of_{X'}^{\oplus 2}\oplus \of_{X'}(0,-1))^\vee$, 
its discriminant locus $\Delta$ is given by a morphism 
$$\of_{X'}(-3,0)\rightarrow S_{222}(\of_{X'}^{\oplus 2}\oplus \of_{X'}(0,-1))^\vee=
\of_{X'}(0,2),$$
hence a section of $\of_{X'}(3,2)$: so $\Delta$ is a surface with nef canonical divisor $\of_\Delta(1,0)$. We notice that there are no higher dimensional fibers. 

\smallskip 
Now we turn to $\pi_{13}$.
The set of pairs $(x,z)$ such that $\beta\circ \alpha(z)\subset x$ is the blow up of $\PP^5$ along the plane $\Pi=\PP(\Ker(\beta\circ \alpha))$. With the additional condition that $\gamma(\beta\circ \alpha(z), z,z)=0$ we get the blow up $Y$ of a  general cubic fourfold $X_3$ containing $\Pi$. The extra point $y$
is then defined as $\alpha(z)$, except if the latter is zero, which means that $z$ belongs to 
the line $l=\PP(\Ker\alpha)\subset \Pi$. The preimage of this line in $Y$ is a surface isomorphic to 
a $(1,2)$-divisor $F$ in $\PP^2\times\PP^1$, hence is a quadratic surface $\QQ^2=\PP^1\times \PP^1$ embedded by $\of(1,2)$. Finally $X$ is the blow up of $Y$ along $\QQ^2$.

\smallskip
Finally we describe $\pi_{23}$.
The set of pairs $(y,z)$ such that $\alpha(z)\subset y$ is the blow up of $\PP^5$ along the line $l$. With the additional condition that $\gamma (\beta(y),z,z)=0$ we get a divisor $Z$ 
of bidegree $(1,2)$ in this blow up, containing the projective plane $\Theta=y_0\times \alpha^{-1}(y_0)=y_0\times\Pi$. Thus $Z$ is the blow up of $X_3$ along the line $l\subset\Pi$; 
in particular it has to contain a copy of $\Bl_l\Pi\simeq\Pi$, namely $\Theta$. 
The extra point $x$ is then uniquely defined as $\beta(y)$, except if the latter is zero, that is on the plane $\Theta$, which is blown up in $X$. Note that the intersection of $\Theta$ with 
the exceptional divisor in $Y$ is just $l$. 
\end{fano}

\section{FK3 from GM fourfolds}\label{sect:GM}
\renewcommand\thefano{GM--\arabic{fano}}
This section contains the  list of the Fano fourfolds of K3 type that we obtained inside products of flag manifolds,  
where at least one projection is a blow up of a Gushel--Mukai fourfold.

\medskip\begin{fano}\fanoid{2-22-74-2} 
\label[mystyle]{22-74-2-1-30-30-1-3-2-5} $\mZ(\PP^2 \times \Gr(2,5),\mU^{\vee}_{\Gr(2,5)}(1,0) \oplus \of(0,1) \oplus \of(0,2)).$

\subsubsection*{Invariants}  $h^0(-K)=22, \ (-K)^4=74, \  h^{1,1}=2, \ h^{3,1}=1, \ h^{2,2}=30$, $-\chi(T)=30$.
\subsubsection*{Description}
\begin{itemize}
    \item $\pi_1 \colon X \to \PP^2$ is a $\dP_4$ fibration over $\PP^2$.
    \item $\pi_2 \colon X \to \Gr(2,5)$ is a blow up $\Bl_{\dP_2}X_{10}$, where the Gushel--Mukai 
    fourfold $X_{10}$ is general.
    \item Rationality: Not known. Birational to a general Gushel--Mukai fourfold.
\end{itemize}

    \subsubsection*{Semiorthogonal decompositions}
\begin{itemize}
\item 7 exceptional objects and $\Db(Z,\mC_0)$, for $Z \to \PP^2$ a $\PP^1$-bundle constructed from $\pi_1$ (see \cref{subsub-intersections}).
\item 14 exceptional objects and the K3 category of $X_{10}$ from $\pi_2$.
\end{itemize}

\subsubsection*{Explanation}

In general if $X$ is $\mZ(\Gr(2, n) \times \Gr(2, n+2),\mQ \boxtimes \mU^{\vee})$, then outside a point $p$, $X$ is isomorphic to $\Bl_{Y}\Gr(2, n+2)$, with $Y \cong \mZ(\PP^{n-1} \times \Gr(2,n+2),\mQ \boxtimes \mU^{\vee}).$

 We can see this by considering a section of $\mQ \boxtimes \mU^{\vee}$ as a morphism from $U_2 \rightarrow V_{n}$, where $U_2$ is the tautological bundle on  $\Gr(2, n+2)$. This morphism is generically injective, and can drop rank twice. If we consider the projection $X \to \Gr(2, n+2)$, this is therefore birational, with fiber $\PP^{n-2}$ over a certain $Y'$ (which corresponds to the locus where the morphism has rank one), and with fiber $\Gr(2, n)$ over a point $p$. This $Y'$
is singular, being resolved by $$Y=\lbrace U_2 \in \Gr(2,n+2), \ P_1 \in V_n, \quad  U_2 \to V_n \to V_n/P_1 \to 0 \rbrace.$$
Therefore $Y\cong \mZ(\PP^{n-1} \times \Gr(2,n+2),\mQ \boxtimes \mU^{\vee}).$

 From the discussion above, it follows that the sixfold $\mZ(\PP^2 \times \Gr(2,5),\mU^{\vee}_{\Gr(2,5)}(1,0))$ is birational to the blow up of $\Gr(2,5)$ in a fourfold $Y^\circ$,
singular in one point, which is resolved by $Y= \mZ(\PP^2 \times \Gr(2,5),\mQ \boxtimes \mU^{\vee})$.
Cutting with two extra $(0,1), (0,2)$ sections we avoid the singular point of $Y^\circ$. Moreover $\mZ(Y,\of(0,1) \oplus \of(0,2))$ is a degree 2 del Pezzo surface, and 
the result follows. 

The other projection, by \cref{lem:secondproj}, realizes the sixfold as $\bbGr_{\Gr(2,3)}(2, \mU \oplus \of^{\oplus 2})$. The other two relative sections cut this sixfold along a $\dP_4$ fibration over $\Gr(2,3)=\PP^2$.
\end{fano}

\medskip\begin{fano}\fanoid{2-23-80-2-A} 
\label[mystyle]{23-80-2-1-27-26-1-3-2-5} $\mZ(\PP^2 \times \Gr(2,5),\of(0,1) \oplus \of(1,1) \oplus \mQ_{\PP^2}(0,1))$. 
\subsubsection*{Invariants}  $h^0(-K)=23, \ (-K)^4=80, \  h^{1,1}=2, \ h^{3,1}=1, \ h^{2,2}=27$, $-\chi(T)=26$.
\subsubsection*{Description}
\begin{itemize}
\item $\pi_1 \colon X \to \PP^2 $ is a fibration in $\dP_5$.
 \item $\pi_2\colon X \to \Gr(2,5)$ is a blow up $\Bl_{\dP_{5}}X_{10}$.
     \item Rationality: yes. Birational to a special Gushel--Mukai containing a $\dP_5$.
\end{itemize}

\subsubsection*{Semiorthogonal decompositions}
\begin{itemize}
    \item 6 exceptional objects and $\Db(Z)$, with $Z \to \PP^2$ finite flat of degree 5, from $\pi_1$, see \cref{subsub-dp5}.
    \item 11 exceptional objects and $\Db(S_{10})$ from $\pi_2$.
\end{itemize}

\subsubsection*{Explanation}
A section of $\mQ(0,1)$ is given by a morphism $\alpha\in \Hom(\wedge^2V_5,V_3)$, a section of $\of(1,1)$ by
 a morphism $\beta\in \Hom(\wedge^2V_5,V_3^\vee)$. A section of  $\of(0,1)$
is defined by a skew-symmetric form $\Omega\in \wedge^2U_5^\vee$. The fourfold $X$ can then be described 
as the set of pairs $(l, P)\in\PP^2\times 2,5)$ such that 
$$\alpha(p_1\wedge p_2)\in l, \quad \beta(p_1\wedge p_2)(l)=0, \quad \Omega(p_1,p_2)=0$$
for any basis $(p_1,p_2)$ of $P$. The first two conditions imply that 
$\beta(p_1\wedge p_2)(\alpha(p_1\wedge p_2))=0$, which defines a quadric in $\Gr(2,5)$, 
while the last condition defines a hyperplane. Therefore the projection 
 to $\Gr(2,5)$ maps $X$ to the intersection of a quadric and a hyperplane, that is a Gushel--Mukai fourfold $X_{10}$. (Note that by construction 
 the quadric has rank six in general.) The exceptional fibers occur over the 
 locus where $\alpha(p_1\wedge p_2)=0$ and $\Omega(p_1,p_2)=0$, which is a $\dP_5$ 
 contained in $X_{10}$. 
 
 The fibers of the projection to $\PP^2$ are linear sections of $\Gr(2,5)$, of constant dimension. 
 So this projection is a fibration in $\dP_5$.
Since del Pezzo surfaces of degree 5 over any field are rational (see, e.g., \cite[page 624]{Isko-Fact}), $X$ is $k(\PP^2)$-rational and hence rational. 
\end{fano}

 \medskip\begin{fano}\fanoid{2-26-94-2} 
\label[mystyle]{26-94-2-1-28-28-1-2-5} $\mZ(\Fl(1,2,5),\of(1,0) \oplus \of(0, 1) \oplus \of(0,2))$.

\subsubsection*{Invariants}  $h^0(-K)=26, \ (-K)^4=94, \  h^{1,1}=2, \ h^{3,1}=1, \ h^{2,2}=28$, $-\chi(T)=28$.
\subsubsection*{Description} 
\begin{itemize}
    \item $\pi_1 \colon X \to \PP^4$ is a conic bundle over a $\PP^3 \subset \PP^4$, whose discriminant is a sextic surface with $40$ ODP.  
    \item $\pi_2 \colon X \to \Gr(2,5)$ is a blow up $\Bl_{\dP_4}X_{10}$, where the Gushel--Mukai 
    fourfold $X_{10}$ is general.
    \item Rationality: unknown. Birational to a general Gushel--Mukai fourfold.
    
\end{itemize}

\subsubsection*{Semiorthogonal decompositions}
\begin{itemize}
    \item 4 exceptional objects and $\Db(\PP^3,\mC_0)$ from $\pi_1$.
    \item 12 exceptional objects and the K3 category of $X_{10}$ from $\pi_2$.
\end{itemize}

\subsubsection*{Explanation} The flag manifold $\Fl(1,2,5)$ is isomorphic to $\PP_{\Gr(2,5)}(\mU)$. Under this identification $\of(1,0)$ becomes the relative $\of(1)$ of the projective bundle. We can therefore use the Cayley trick, to show that $X$ is the blow up of a general Gushel--Mukai fourfold in the zero locus $\mZ(\Gr(2,5),\mU^{\vee} \oplus \of(1) \oplus\of (2))$, the latter being a quartic del Pezzo surface. 

\begin{rmk}
The conic bundle structure given by $\pi_1$ has been described in \cite[Corollary 2.3]{Orna-Pert}.\end{rmk}

\subsubsection*{Another model}
It is interesting to note that there exists another model of a Gushel--Mukai fourfold blown up in a $\dP_4$. This model is embedded in $\PP^3 \times \PP^8$ as $\mZ(\of(0,2) \oplus \W^2 \mQ_{\PP^2}(0,1) \oplus \mQ_{\PP^2}(0,1))$. So  we
start with a quadratic form $q\in S^2V_9^\vee$, 
and two tensors $\alpha\in \Hom(V_9, V_4)$ and $\beta\in \Hom(V_9, \wedge^2V_4)$.
The fourfold $X$ parametrizes the pairs of lines $(x,y)$ in $\PP^3\times \PP^8$ 
such that $q(y)=0$, 
$\alpha(y)\subset x$ and $\beta(y)\subset x\wedge V_4$. 

If $x$ is given in $\PP^3$, $y$ must belong to $\alpha^{-1}(x)$ and $\beta^{-1}(x\wedge V_4)$, which in general defines a projective plane on which $q$ cuts out a conic. So the projection of $X$ to $\PP^3$ is a conic bundle. 

If $y$ is given in $\PP^8$, its fiber in $X$ is a unique point, or empty, 
except if $q(y)=0$, $\alpha(y)=0$ and $\beta(y)$ is degenerate, in which case we get a projective line. The degeneracy condition can be expressed as $\beta(y)\wedge\beta(y)=0$, which defines a quadric, in general smooth when 
restricted to $\PP(\Ker\alpha)$. So the projection of $X$ to $\PP^8$ should 
be the blow up of a smooth fourfold $Y$ along a $\dP_4$. This fourfold is defined 
by the condition that $q(y)=0$ and $\alpha(y)\wedge\beta(y)=0$. Observe that 
we may interpret  $\alpha$ and $\beta$ as defining a map $\gamma : V_9\rightarrow \wedge^2V_5$, where $V_5=V_4\oplus \mathbb{C} e_0$ so that $\wedge^2V_5\simeq \wedge^2V_4\oplus V_4$. Explicitly $\gamma(y)=\beta(y)+\alpha(y)\wedge e_0$.
The condition  $\alpha(y)\wedge\beta(y)=0$ implies that $\beta(y)=\alpha(y)\wedge v$ for some vector $v\in V_4$ (or $\alpha(y)=0$), in which case we deduce that 
$\gamma(y)=\alpha(y)\wedge (v+e_0)$ is decomposable. This means that $Y$ is contained 
in a hyperplane section of $\Gr(2,V_5)$, and the quadric condition implies that 
$Y$ is in fact a Gushel--Mukai variety. 
\end{fano}

We remark that there exists a link between this Fano variety and \cref{23-80-2-1-28-27-1-4-1-7}, investigated in the recent \cite{nodalK3}.

\medskip\begin{fano}\fanoid{2-30-114-2} 
\label[mystyle]{30-114-2-1-24-24-2-4-2-5} $\mZ(\Gr(2,4) \times \Gr(2,5),\of(0,2) \oplus \mQ_{\Gr(2,4)} \boxtimes \mU^\vee_{\Gr(2,5)} \oplus \of (1,0))$.

\subsubsection*{Invariants}  $h^0(-K)=30, \ (-K)^4=114, \  h^{1,1}=2, \ h^{3,1}=1, \ h^{2,2}=24$, $-\chi(T)=24$.
\subsubsection*{Description} 
\begin{itemize}
    \item $\pi_1 \colon X \to \Gr(2,4)$ is a conic bundle on $\QQ^3$ with discriminant divisor a singular quartic surface.
    \item $\pi_2 \colon X \to \Gr(2,5)$ is a blow up $\Bl_{\PP^1 \times \PP^1}X_{10}$  of a Gushel--Mukai fourfold along a $\sigma$-quadric surface.
    \item Rationality: unknown.
\end{itemize}

\subsubsection*{Semiorthogonal decompositions}
\begin{itemize}
    \item 6 exceptional objects and $\Db(\QQ^3,\mC_0)$ from $\pi_1$.
    \item 8 exceptional objects and the K3 category of $X_{10}$ from $\pi_2$.
\end{itemize}

\subsubsection*{Explanation}
By \cref{lem:blowgrass} the zero locus of  $\mZ(\Gr(2,4) \times \Gr(2,5), \mQ_{\Gr(2,4)} \boxtimes \mU^\vee_{\Gr(2,4)})$ 
is the blow up $\Bl_{\PP^3}\Gr(2,5)$. The two extra sections cut $\Gr(2,5)$ in 
a generic Gushel--Mukai fourfold $X_{10}$. The quadratic section cuts the $\PP^3$ in a quadric 
surface $\Sigma$.
In terms of the  morphism $\theta\in \Hom(V_5,V_4)$  defining the section 
of $\mQ_{\Gr(2,4)} \boxtimes \mU^\vee_{\Gr(2,5)}$, our fourfold $X$ parametrizes pairs of planes $(P,Q)$ such that $\theta(Q)\subset P$. So $P$ is uniquely determined by $Q$, unless the latter meets the 
kernel of  $\alpha$: this defines the $\PP^3$ blown up inside $\Gr(2,5)$, and we conclude that 
$\pi_2$ is just the blow up of $X_{10}$ along $\Sigma$.

The fiber of $\pi_1$ over a plane $P$ is the intersection of the Gushel--Mukai fourfold  with the variety
of planes in $\theta^{-1}(P)$.
This means we need to cut a projective plane 
with a quadric, yielding the conic bundle structure. In order to understand the discriminant 
locus, note that the latter projective plane is $\PP(\wedge^2\theta^{-1}(P))$, and that since 
we restrict a fixed quadric to this plane, the discriminant is a section of 
$$\det(\wedge^2\theta^{-1}(P)^\vee)^2= \det (\theta^{-1}(P)^\vee)^4=\det(P^\vee)^4.$$
So the discriminant locus is a quartic surface in $\QQ^3$. 
\end{fano}

We remark that there exists a link between this Fano variety and \cref{25-90-2-1-24-23-2-4-2-4}, investigated in the recent \cite{nodalK3}.

\medskip\begin{fano}\fanoid{2-33-130-2} 
\label[mystyle]{33-130-2-1-23-22-2-4-2-5} $\mZ(\Gr(2,4) \times \Gr(2,5),\mQ_{\Gr(2,4)} \boxtimes \mU^\vee_{\Gr(2,5)} \oplus \of(0,1) \oplus \of(1,1))$.
\subsubsection*{Invariants}  $h^0(-K)=33, \ (-K)^4=130, \  h^{1,1}=2, \ h^{3,1}=1, \ h^{2,2}=23$, $-\chi(T)=22$.
\subsubsection*{Description} 
\begin{itemize}
    \item $\pi_1 \colon X \to \Gr(2,4)$ is a blow up of $S_{10}^{(1)}$, a K3 surface of genus 6 blown up at one point.
    \item $\pi_2 \colon X \to \Gr(2,5)$ is a blow up $\Bl_{\PP^2}X_{10}$.
    \item Rationality: yes.
\end{itemize}

\subsubsection*{Semiorthogonal decompositions}
\begin{itemize}
    \item 7 exceptional objects and $\Db(S_{10})$ from both maps (a general $X_3$ containing a $\sigma$-plane also contains a $\dP_5$, see \cref{subsub:GMdp5}).
\end{itemize}

\subsubsection*{Explanation}
We start by applying \cref{lem:blowgrass} to the first bundle, whose section cuts $\Gr(2,4) \times \Gr(2,5)$ in $\Bl_{\PP^3}\Gr(2,5)$. The other two sections are then identified with a linear section of $\Gr(2,5)$ cutting the $\PP^3$ in a plane, and a quadratic section containing it. The special $\PP^3$ is of the form $\Gr(1, V_4)$: the resulting plane is therefore a $\sigma$-plane in the notation of \cite{Deb-Ili-Man}. In Proposition 7.1, \emph{loc.\ cit.\ }it is shown  that a  Gushel--Mukai fourfold containing such a special plane is rational.

Now consider the projection $\pi_1$ to $\Gr(2,4)$. The section of $\mQ_{\Gr(2,4)} \boxtimes \mU^\vee_{\Gr(2,5)}$ 
is defined by a morphism $\theta$ from $V_5$ to $V_4$, and it vanishes on the pairs $(A,B)$ such that $B\in \theta^{-1}(A)$. 
For a given $A$, this defines a projective plane, on which the two other line bundle sections define 
two linear conditions. So $\pi_1$ is birational and its exceptional locus is the degeneracy locus of the induced morphism 
$$\of_{\Gr(2,4)}\oplus \of_{\Gr(2,4)}(-1)\longrightarrow \wedge^2 \theta^{-1}(\mU)^\vee\simeq \mU^\vee\oplus  \of_{\Gr(2,4)}(-1),$$
where $\mU=\mU_{\Gr(2,4)}$ is the tautological bundle. 
This degeneracy locus is a smooth surface $S$ whose structure sheaf is resolved by the Eagon--Northcott complex $$0\longrightarrow \of_{\Gr(2,4)}(-2)\oplus \of_{\Gr(2,4)}(-3)\longrightarrow \of_{\Gr(2,4)}(-1)\oplus \mU(-1)\longrightarrow \of_{\Gr(2,4)}(1)\longrightarrow \of_S(1)\longrightarrow 0.$$
Since $\mU$ and $\mU(-1)$ are acyclic by Borel--Weil--Bott, we deduce that $\chi(\of_S)=2$, while $\chi(\of_S(1))=6$ and $\chi(\of_S(2))=20-1=19$, hence the Hilbert polynomial must be 
$$\chi(\of_S(k))=\frac{9k^2-k}{2}+2.$$
Dualizing the previous complex, we get 
 $$0\longrightarrow \of_{\Gr(2,4)}(-4)\longrightarrow 
  \of_{\Gr(2,4)}(-2)\oplus\mU^\vee(-2)
 \longrightarrow \of_{\Gr(2,4)}(-1)\oplus\of_{\Gr(2,4)}\longrightarrow \omega_S\longrightarrow 0.$$
 So $\omega_S$ has a canonical section, vanishing on the curve $D$ where the morphism 
 $ \of_{\Gr(2,4)}(-2)\oplus\mU^\vee(-2)
 \longrightarrow \of_{\Gr(2,4)}(-1)$ is not surjective, that is, the zero locus of a section 
 of $\of_{\Gr(2,4)}(1)\oplus\mU^\vee$. But the zero locus of a nonzero section of  $\mU^\vee$
 is a plane, so $D$ is just a projective line. By the usual arguments we deduce that 
$S$ is a K3 surface of genus $6$ blown up at one point. 
\end{fano}

\medskip\begin{fano}\fanoid{2-30-115-2-B} 
\label[mystyle]{30-115-2-1-24-23-1-3-5} $\mZ(\Fl(1,3,5),\of(0,1)\oplus \of(1,1) \oplus \mR_2(0,1))$.

\subsubsection*{Invariants}  $h^0(-K)=30, \ (-K)^4=115, \  h^{1,1}=2, \ h^{3,1}=1, \ h^{2,2}=24$, \ $-\chi(T)=23$.
\subsubsection*{Description} 
\begin{itemize}
    \item $\pi_1 \colon X \to \PP^4$ is birational, with base locus in $\PP^4$ a 
    singular surface $S^\circ$ obtained by identifying two points on a non-minimal K3.
    \item $\pi_2 \colon X \to \Gr(3,5)$ is a blow up $\Bl_{\QQ^2}X_{10}$ of a Gushel--Mukai fourfold along a quadratic surface. 
    \item Rationality: yes.
\end{itemize}

\subsubsection*{Semiorthogonal decompositions}
\begin{itemize}
    \item 8 exceptional objects and the K3 category of $X_{10}$ from $\pi_2$.
\end{itemize}

\subsubsection*{Explanation}
In the variety of flags $U_1\subset U_3\subset V_5$, the fourfold $X$ is defined by sections 
of $\det(\mU_3)^\vee$, $\mU_1^\vee\otimes\det(\mU_3)^\vee$ and $\mU_1^\vee\otimes (\mU_3/\mU_1)^\vee$. 

Consider the projection to $\PP^4$. The fibers are copies of $\Gr(2,4)$ and the 
fibers of the restriction to the fourfold $X$ are given by sections of $\mU^\vee\oplus \of(1)^{\oplus 2}$;
hence they are points, generically. So the map $\pi_1 : X\rightarrow \PP^4$ must be birational. 

Where are the non-trivial fibers? There are two possible reasons that could explain that the 
zero locus of a section of  $\mU^*\oplus \of(1)^{\oplus 2}$ on $\Gr(2,4)$ is positive dimensional. 
First the component on $\mU^\vee$ could vanish identically. Since this section is defined by a two-form $\omega$, this will happen over
the point $p$ of $\PP^4$ defined by the kernel of this two-form, whose fiber will be a 
quadratic surface. Outside this point, the section of $\mU^\vee$ defines a $\PP^2$, and we will have
non-trivial fibers when the two sections of $\of(1)$ are linearly dependent. 
This should 
happen over a surface $S^\circ\subset \PP^4$. What is this surface? 

The two-form $\omega$ defines a morphism $\of(-1)\rightarrow \mQ^\vee$ on $\PP^4$, 
which is not injective (because of the kernel of $\omega$). If it was, the quotient $\mK=\mQ^\vee(-1)/\of(-2)$
would be a rank three-bundle and we could describe $S^\circ$ as a degeneracy locus for a morphism
$$ \mK\oplus \of(-1)\oplus \of(-2)\longrightarrow \wedge^2\mQ^\vee.$$
In order to make this correct, we need to blow up $p$. Denote $Y=\Bl_p\PP^4$, 
$\pi$ the blow down to $\PP^4$,
$E$ the exceptional divisor, $H$ the pull-back of the hyperplane divisor.
On $Y$ there is an injective morphism $\of(-2H+E)\rightarrow \mQ^\vee(-H)$, 
whose quotient is a rank three vector bundle $\mF$. Then we can define a smooth 
surface $S$ in $Y$ as the degeneracy locus of the induced morphism 
  $$\mF(E)\oplus \of(-H)\oplus \of(-2H)\longrightarrow \wedge^2\mQ^\vee.$$
The structure sheaf of $S$ is resolved by the  Eagon--Northcott complex 
$$0\longrightarrow \mF(E)\oplus \of(-H)\oplus \of(-2H)\longrightarrow \wedge^2\mQ^\vee\longrightarrow \of(3H-2E)\longrightarrow \of_S(3H-2E)\longrightarrow 0.$$
We deduce that the class of $S$ in the Chow ring of $Y$ (where $HE=0$) is 
$[S]=8H^2+2E^2$. In particular the image $S^\circ$ of $S$ in $\PP^4$ has degree eight. Moreover the morphism $S\rightarrow S^\circ$ contracts the curve $C=S\cap E$, whose self intersection $C^2=[S]E^2=2E^4=-2$. This means that $C$ is either a $(-2)$-curve or the disjoint union of two $(-1)$-curves, and we need to decide 
whether it is connected or not.
For this we need to compute $h^1(\of_S(-E))$, which we can
do with a suitable twist of the Eagon--Northcott complex. 

\begin{lemma} 
The bundle $\mF(3E-3H)$ is acyclic, but $h^q(\mF(2E-3H))=\delta_{q,3}$. As a consequence, 
$$ h^1(\of_S)=0, \qquad h^2(\of_S)=2, \qquad h^1(\of_S(-E))=1.$$
\end{lemma}

\begin{proof} From the exact sequence that defines $\mF$, we get 
$$0 \longrightarrow\of((k+1)E-5H)\longrightarrow \mQ^\vee(kE-4H)
\longrightarrow \mF(kE-3H)\longrightarrow 0.$$
For $0\le k\le 3$,  $\of(kE)$ has no higher direct images by $\pi$, while $\pi_*\of(kE)=\of_{\PP^4}$. 
So $Q^\vee(kE-4H)$ has the same cohomology as  $Q^\vee(-4)$ on $\PP^4$, and by Borel--Weil--Bott this bundle is acyclic. So $$h^q(\mF(kE-3H))=h^{q+1}(\of((k+1)E-5H)=h^{3-q}(\of((2-k)E)$$ by Serre duality, since 
$\omega_Y=\of(3E-5H)$. This implies the first claim. 

Now consider the resolution of $\of_S$ given by the Eagon--Northcott complex. The bundle 
$ \wedge^2\mQ^\vee(2E-3H)$ has the same cohomology as $ \wedge^2\mQ^\vee(-3)$ on $\PP^4$, which 
is again acyclic by Borel--Weil--Bott. Moreover, since $\mF(3E-3H)$ is also acyclic, we deduce that 
$h^1(\of_S)=0$ while $h^2(\of_S)=h^4(\of(2E-5H))=1$. But when we twist the Eagon--Northcott complex
by $\of(-E)$, we get a non-trivial contribution to $h^1(\of_S(-E))=h^3(\mF(2E-3H))=1$.
\renewcommand\qedsymbol{$\square$}
\end{proof}

From $h^1(\of_S)=0$ and  $h^1(\of_S(-E))=1$ we can now conclude that $C$ has two connected components, 
hence $C=C_1+C_2$ is the disjoint union of two $(-1)$-curves on $S$. The projection to $S^\circ$ contracts 
these two curves to the same point $p$, at which $S^\circ$ is therefore 
not normal. Using Riemann--Roch we compute that $H\omega_S=2$
and $C\omega_S=-2$.  Now, dualizing the Eagon--Northcott complex we get 
$$0\longrightarrow \of(3E-5H)
\longrightarrow \wedge^2\mQ(E-2H)\longrightarrow 
\mF^\vee(-2H)\oplus \of(E-H)\oplus \of(E)\longrightarrow \omega_S\longrightarrow 0.$$
In particular $\omega_S(-C)=\of_S(D)$ is effective, and we get $DC=0$ and $DH=2$.
So $D$ is a smooth conic or the disjoint union of two lines, and in both cases we can 
conclude that $S$ is once again a non-minimal K3 surface.  

\begin{rmk} Another Eagon--Northcott complex should allow us to decide whether $D$ is connected or not.\end{rmk}

\smallskip 
Now consider the projection $\pi_2$ of $X$ to $\Gr(3,5)$. The section of $\det(\mU_3)^\vee$ defines a 
hyperplane section $H$ of $\Gr(3,5)$ that contains the image. The section of $\mU_1^\vee\otimes (\mU_3/\mU_1)^\vee$
is defined by a skew-symmetric two-form $\omega$ and $U_1$ must be contained in the kernel of this two-form
restricted to $U_3$. This restriction is an element of $\wedge^2U_3^\vee\simeq U_3\otimes \det(U_3)^\vee$, 
which, when non-zero, gives a line in $U_3$ that coincides with the kernel. Then the section of 
$\mU_1^\vee\otimes\det(\mU_3)^\vee$, which is an element of $U_3^\vee\otimes\det(U_3)^\vee$, gives a 
line of linear
forms that need to 
vanish on this kernel, yielding a section of $\det(\mU_3^\vee)^2$. As a consequence, the image
 of $X$ in $\Gr(3,5)$ is the intersection of a quadric and a hyperplane, hence a Gushel--Mukai fourfold. 
 
Finally, $\pi_2$ has non-trivial fibers when $\omega$ is zero on $U_3$. This condition defines a 
quadric of $3$-spaces containing $\Ker(\omega)$, that intersects $H$ along a quadric surface. 
\end{fano}

\section{FK3 from K3 surfaces}\label{sect:FK3}
\renewcommand\thefano{K3--\arabic{fano}}
This section contains the  list of the Fano fourfolds of K3 type that we obtained inside products of flag manifolds,  
where at least one projection  is a blow up with center birational to a K3 surface, 
and no cubic or Gushel--Mukai fourfold is involved. 

\medskip\begin{fano}\fanoid{2-19-60-2} 
\label[mystyle]{19-60-2-1-32-31-1-2-2-5} $\mZ(\PP^1 \times \Gr(2,5),\of(1,1) \oplus \mU_{\Gr(2,5)}^{\vee}(0,1)).$
\subsubsection*{Invariants}  $h^0(-K)= 19, \ (-K)^4=60, \  h^{1,1}=2, \ h^{3,1}=1, \ h^{2,2}=32$, $-\chi(T)=31$.
\subsubsection*{Description} 
\begin{itemize}
    \item $\pi_1 \colon X \to \PP^1 $ is a fibration in $W_{12}$, Fano threefolds of index 1 and degree 12.
 \item $\pi_2\colon X \to \Gr(2,5)$ is a blow up $\Bl_{S_{12}}X_{12}$.
     \item Rationality: yes.
\end{itemize}

\subsubsection*{Semiorthogonal decompositions}
\begin{itemize}
    \item 4 exceptional objects and an unknown category from $\pi_1$.
    \item 16 exceptional objects and $\Db(S_{12})$ from $\pi_2$.
\end{itemize}

\subsubsection*{Explanation} It suffices to apply \cref{lem:blowupcodim2} to describe $\pi_2$. This proves the rationality of this fourfold by \cite[Thm 3.3]{kuznetsov2020rationality}. On the other hand, $\PP^1$ parametrizes a family of hyperplane sections of $X_{12}$.
\end{fano}

\medskip\begin{fano}\fanoid{2-21-70-2} 
\label[mystyle]{21-70-2-1-28-27} $\mZ(\PP^1 \times \Gr(2,6),\of(1,1) \oplus \of(0,1)^{\oplus 4})$.
\subsubsection*{Invariants}  $h^0(-K)=21, \ (-K)^4=70, \  h^{1,1}=2, \ h^{3,1}=1, \ h^{2,2}=28$, $-\chi(T)=27$.
\subsubsection*{Description} 
\begin{itemize}
    \item $\pi_1 \colon X \to \PP^1 $ is a fibration in $W_{14}$, Fano threefolds of index 1 and degree 14.
    \item $\pi_2 \colon X \to\Gr(2,6)$ is a blow up $\Bl_{S_{14}}X_{14}$.
     \item Rationality: yes.
\end{itemize}

\subsubsection*{Semiorthogonal decompositions}
\begin{itemize}
    \item 4 exceptional objects and an unknown category from $\pi_1$.
    \item 12 exceptional objects and $\Db(S_{14})$ from $\pi_2$.
\end{itemize}

\subsubsection*{Explanation}
It suffices to apply \cref{lem:blowupcodim2} to describe $\pi_2$. On the other hand, $\pi_1$ 
 is clearly a fibration in $W_{14}$. By \cite[Thm 3.3]{kuznetsov2020rationality} $W_{14}$ is rational, and so is our fourfold $X$ as well.
\end{fano}

\medskip\begin{fano}\fanoid{2-23-80-2-B} 
\label[mystyle]{23-80-2-1-24-23-1-2-3-6} $\mZ(\PP^1 \times \Gr(3,6),\of(1,1) \oplus \of(0,1)^{\oplus 2} \oplus \W^2\mU_{\Gr(3,6)}^{\vee})$.
\subsubsection*{Invariants}  $h^0(-K)=23, \ (-K)^4=80, \  h^{1,1}=2, \ h^{3,1}=1, \ h^{2,2}=24$, $-\chi(T)=23$.

\subsubsection*{Description} 
\begin{itemize}
     \item $\pi_1 \colon X \to \PP^1 $ is a fibration in $W_{16}$, Fano threefolds of index 1 and degree 16.
\item $\pi_2\colon X \to \Gr(3,6)$ is a blow up $\Bl_{S_{16}}X_{16}$ (recall that the index two Fano fourfold $X_{16}$ is a codimension two linear section of the Lagrangian Grassmannian $\mathrm{LG}(3,6)$).
     \item Rationality: unknown. Birational to a general $X_{16}$.
\end{itemize}

\subsubsection*{Semiorthogonal decompositions}
\begin{itemize}
    \item 4 exceptional objects and an unknown category from $\pi_1$.
    \item 8 exceptional objects and $\Db(S_{16})$ from $\pi_2$.
\end{itemize}

\subsubsection*{Explanation}
It suffices to apply \cref{lem:blowupcodim2} to describe $\pi_2$. On the other hand, $\pi_1$  is 
clearly a fibration in $W_{16}$, hyperplane sections of $X_{16}$. The rationality of $X_{16}$ is unknown, see \cite[Section 4]{kuznetsov2020rationality}.
\end{fano}

\medskip\begin{fano}\fanoid{2-32-124-2} 
\label[mystyle]{32-124-2-1-28-28-1-2-8} $\mZ(\Fl(1,2,8),\mQ_2 \oplus \of(1,1) \oplus \of(0,2)^{\oplus 2}).$
\subsubsection*{Invariants}  $h^0(-K)=32, \ (-K)^4=124, \  h^{1,1}=2, \ h^{3,1}=1, \ h^{2,2}=28$, $-\chi(T)=28$.
\subsubsection*{Description} 
\begin{itemize}
    \item $\pi_1 \colon X \to \PP^7$ is 
a blow up $\Bl_{v_0} (Q^{\circ}_1 \cap Q^{\circ}_2 \cap \QQ^6)$, where $Q^{\circ}_i$ are quadratic cones with vertex $v_0$. 
    \item $\pi_2 \colon X \to \Gr(2,8)$ is a blow up $\Bl_{S_8} (\QQ_1^5 \cap \QQ_2^5)$ of 
    the intersection of two quadrics along a K3 surface of genus 5.
    \item Rationality: yes.
\end{itemize}

\subsubsection*{Semiorthogonal decompositions}
\begin{itemize}
    \item 12 exceptional objects and $\Db(S_8)$ from $\pi_2$. 
\end{itemize}

\subsubsection*{Explanation} By \cite[Corollary 2.7]{DFT}, $\mZ(\Fl(1,2,8),\mQ_2)$ can be identified with $\PP_{\PP^6}(\of \oplus \of(-1))$. The two $\of(0,2)$ sections further cut the base $\PP^6$ in the smooth complete intersection $Y=\QQ_1^5 \cap \QQ_2^5$ of two quadrics. By the Cayley trick, 
$\mZ(\PP_{Y}(\of \oplus \of(-1)),\mL \otimes \of_{Y}(1))$ is isomorphic to the blow up of $Y$ in $\mZ(\PP^6,\of(2)^3 \oplus \of(1))$, which is a general K3 surface of degree 8. Note that $Y$ is classically known to be rational, see e.g.\ \cite{reid1972complete}, so $X$ is rational as well.

Now consider the projection to $\PP^7$. The section of $\mQ_2$ defines a point $[v_0]\in \PP^7$ and 
outside this point, the fiber of $[v]$ is the plane in $V_8$ generated by $v$ and $v_0$. The two
sections of $\of(0,2)$ become quadrics $Q^{\circ}_1, Q^{\circ}_2$ in $\PP^7$ defined by equations of the form $Q_i(v\wedge v_0)=0$,
in particular these are quadratic cones singular at $[v_0]$. The section of $\of(1,1)$ becomes a non-singular quadratic 
form vanishing at $v_0$. We deduce that the projection of $X$ to $\PP^7$ is 
the blow up $\Bl_{v_0}(Q^{\circ}_1 \cap Q^{\circ}_2 \cap \QQ^6)$, with exceptional divisor the intersection of two quadrics in $\PP^5$.
\end{fano}

\medskip\begin{fano}\fanoid{2-25-90-2-A} 
\label[mystyle]{25-90-2-1-22-21-1-2-2-7} $\mZ(\PP^1 \times \Gr(2,7),\of(0,1) \oplus \of(1,1) \oplus \mQ_{\Gr(2,7)}^{\vee}(0,1)).$

\subsubsection*{Invariants}  $h^0(-K)=25, \ (-K)^4=90, \  h^{1,1}=2, \ h^{3,1}=1, \ h^{2,2}=22$, $-\chi(T)=21$.
\subsubsection*{Description} 
\begin{itemize}
    \item $\pi_1 \colon X \to \PP^1$ is a fibration in $W_{18}$, Fano threefolds of index 1 and degree 18.
    \item $\pi_2 \colon X \to \Gr(2,7)$ is a blow up $\Bl_{S_{18}}X_{18}$.
    \item Rationality: yes.
\end{itemize}

\subsubsection*{Semiorthogonal decompositions}
\begin{itemize}
    \item 4 exceptional objects and an unknown category from $\pi_1$.
    \item 6 exceptional objects and $\Db(S_{18})$ from $\pi_2$.
\end{itemize}

\subsubsection*{Explanation} This follows from  \cref{lem:blowupcodim2}. Rationality follows from \cite[Theorem 3.3]{kuznetsov2020rationality}.
\end{fano}

\medskip\begin{fano}\fanoid{2-39-160-2-A} 
\label[mystyle]{39-160-2-1-22-21-1-3-1-5}
$\mZ(\PP^2 \times \PP^4,\of(1,1)\oplus \of(1,2))$.

\subsubsection*{Invariants}  $h^0(-K)=39, \ (-K)^4=160, \  h^{1,1}=2, \ h^{3,1}=1, \ h^{2,2}=22$, $-\chi(T)=21$.

\subsubsection*{Description}
\begin{itemize}
    \item $\pi_1 \colon X \to \PP^2$ is a quadric bundle with discriminant a smooth sextic.
    \item $\pi_2 \colon X \to \PP^4$ is a blow up $\Bl_{S_8^{(1)}}\PP^4$ along a K3 surface of degree eight blown up at one point.
    \item Rationality: yes.
\end{itemize}

\subsubsection*{Semiorthogonal decompositions}
\begin{itemize}
    \item 6 exceptional objects and $\Db(S_2,\alpha)$ from $\pi_1$.
    \item 6 exceptional objects and $\Db(S_8)$ from $\pi_2$.
\end{itemize}

\subsubsection*{Explanation}
The fibers of the projection $\pi_1$ to $\PP^4$ are defined by the image of the 
induced morphism $$\of_{\PP^4}(-1)\oplus \of_{\PP^4}(-2)\longrightarrow V_3^\vee \otimes \of_{\PP^4}.$$
Generically this morphism has rank two and the fiber is one point, so $\pi_1$ is 
birational. The fibers are projective lines when the morphism drops rank, which 
happens over a smooth surface $S$ whose ideal is resolved by an Eagon--Northcott complex
$$0 \longrightarrow \of_{\PP^4}(-4)\oplus \of_{\PP^4}(-5)\longrightarrow 
V_3^\vee\otimes\of_{\PP^4}(-3)\longrightarrow \mathcal{I}_S\longrightarrow 0.$$
From this complex one deduces the Hilbert polynomial
$$\chi(\of_S(k))=\frac{7k^2-k}{2}+2.$$
In particular $S$ has degree seven, arithmetic genus $1$, sectional genus $5$.
According to Theorem 6 in \cite{Oko84}, the surface $S$ must be 
an inner projection of a K3 surface of degree eight in $\PP^5$. 

On the other hand, $\mZ(\PP^2 \times \PP^4,\of(1,1))$ is $\PP_{\PP^2}(\mE)$ with $\mE=\Omega^1_{\PP^2}(1)\oplus \of_{\PP^2}^{\oplus 2}$ by the Euler sequence. The quadric fibration is given by a map $\mE \to \mE^\vee(1)$, and the degeneracy locus has degree $c_1(\mE^\vee)+4=6$.
\end{fano}

\medskip\begin{fano}\fanoid{2-39-160-2-B} 
\label[mystyle]{39-160-2-1-22-21-1-2-2-4} $\mZ(\PP^1 \times \Gr(2,4),\of(1,2))$.

\subsubsection*{Invariants}  $h^0(-K)=39, \ (-K)^4=160, \  h^{1,1}=2, \ h^{3,1}=1, \ h^{2,2}=22$, $-\chi(T)=21$.
\subsubsection*{Description} 
\begin{itemize}
    \item $\pi_1 \colon X \to \PP^1$ is a fibration whose fibers are intersections of two quadrics in $\PP^5$.
    \item $\pi_2 \colon X \to \Gr(2,4)$ is a blow up $\Bl_{S_8}\QQ^4$ of a quadric along a K3 surface of genus 5.
    \item Rationality: yes.
\end{itemize}

\subsubsection*{Semiorthogonal decompositions}
\begin{itemize}
    \item 4 exceptional objects and $\Db(Z,\mC_0)$, for $Z \to \PP^1$ a Hirzebruch surface.
    \item 6 exceptional objects and $\Db(S_8)$ from $\pi_2$.
\end{itemize}

\subsubsection*{Explanation} We use \cref{lem:blowupcodim2} to determine $\pi_2$. In order to understand  $\pi_1$, just observe that each fiber is given by the intersection of $\Gr(2,4)$ with an extra quadratic section, hence the result.
\end{fano}

\begin{rmk}
This Fano fourfold already appeared in \cite{eg2}, under the name (B2).\end{rmk}

\medskip\begin{fano}\fanoid{2-23-80-2-C} 
\label[mystyle]{23-80-2-1-28-27-1-4-1-7} $\mZ(\PP^3 \times \PP^6,\of(0,2) \oplus \of(1,1) \oplus \W^2 \mQ_{\PP^3}(0,1))$.

\subsubsection*{Invariants}  $h^0(-K)=23, \ (-K)^4=80, \  h^{1,1}=2, \ h^{3,1}=1, \ h^{2,2}=28$, $-\chi(T)=27$.

\subsubsection*{Description}
\begin{itemize}
\item $\pi_1 \colon X \to \PP^3 $ is  a conic bundle whose discriminant is a sextic hypersurface.
 \item $\pi_2\colon X \to \PP^6$ is a blow up $\Bl_{S_{10}}(\QQ_1^5 \cap \QQ_2^5)$ of the intersection of two quadrics, for $S_{10}$ a K3 surface of genus 6.
 \item Rationality: yes.
\end{itemize}

\subsubsection*{Semiorthogonal decompositions}
\begin{itemize}
    \item 4 exceptional objects and $\Db(\PP^3,\mC_0)$ from $\pi_1$.
    \item 12 exceptional objects and $\Db(S_{10})$ from $\pi_2$.
\end{itemize}

\subsubsection*{Explanation}
Let us first consider the sixfold $Y$ defined as the zero locus of a general section of $\mQ^\vee(1,1)$ inside $\PP^3 \times \PP^6$.
The space of sections  is
$\W^2 V_4\vee \otimes V_7^\vee=\Hom(V_7,\W^2 V_4)$. If $\theta \in \Hom(V_7,\W^2 V_4)$ defines $Y$, 
the latter parametrizes the pairs $(x,y)\in \PP^3 \times \PP^6$ such that $\theta(y)\subset x\wedge V_4$. 
In particular $\theta(y)$ has rank two or zero, hence  $\theta(y)\wedge  \theta(y)=0$. 
This equation defines a rank six quadric $Q^\circ \subset \PP^6$, with vertex $v$ given by the 
kernel of $\theta$. Note that if $y$ is a smooth point of $Q^\circ$, $\theta(y)$ defines a plane 
in $V_4$ whose projectivization  is the  fiber over $y$ of the projection $Y\rightarrow Q^\circ$;
while over the vertex the fiber is the whole $\PP^3$.

Now when we cut $Y$ by a section of $\of(0,2)$ we are restricting the previous construction to the intersection of $Q^\circ$ with a general smooth quadric $Q'$, hence we avoid the singularity and obtain 
$Y' \to Q^\circ \cap Q'$ as a $\PP^1$-bundle $\PP (\mF)$ over the smooth intersection of two quadrics in $\PP^6$.

Finally, we cut $Y'$ by a section of  $\of(1,1)$ and we apply the Cayley trick to see that 
$\pi_2: X \to Q^\circ \cap Q'$ is the blow up along a surface $S$ cut by a general section of $\mF^\vee (1)$, which is a K3 surface. 
In fact $\mF$ is the pullback of the tautological bundle $\mU$ by the projection $p$ to $\Gr(2,4)$,
and the restriction of this projection to $S$ is a double cover over a $\dP_5$ (embedded in $\Gr(2,4)$ as the zero locus of a section of $\mU^\vee(1)$), branched over a quadratic section.
The embedding of $S$ into $Q^\circ\cap Q'$ gives a complex
$$0\longrightarrow  p^*\mU^\vee(-3)\longrightarrow p^*End(\mU)(-1)\longrightarrow p^*\mU^\vee
 \longrightarrow p^*\mU^\vee_S\longrightarrow 0.$$
 By Borel--Weil--Bott $h^0(\mU^\vee(-3))=1$. This implies that $h^0(\mU^\vee_S)=5$ and that $S$ is a K3 surface of genus $6$. 

\smallskip
To determine $\pi_1$, we notice that $\mZ(\PP^3 \times \PP^6,\of(1,1) \oplus \W^2 \mQ_{\PP^3}(0,1))$ is a $\PP^2$-bundle on $\PP^3$ and we cut it with a quadratic relative section. More precisely
the $\PP^2$-bundle is $\PP(\mH)$ where $\mH$ is the sub-bundle of the trivial bundle with fiber $V_7$
defined by the exact sequence 
$$0\ra \of_{\PP^3}(-1)\oplus \wedge^2 \mQ_{\PP^3}^\vee\ra \of_{\PP^3}\otimes V_7^\vee\ra \mH^\vee\ra 0.$$ So the discriminant is given by a section of $(\det \mH^\vee)^2=\of_{\PP^3}(6).$
\end{fano}

We remark that there exists a link between this Fano variety and \cref{26-94-2-1-28-28-1-2-5}, investigated in the recent \cite{nodalK3}.

\medskip\begin{fano}\fanoid{2-29-110-2} 
\label[mystyle]{29-110-2-1-22-21-1-2-5} $\mZ (\Fl(1,2,5),\of(1,1) \oplus \of(0,1)^{\oplus 2})$.

\subsubsection*{Invariants}  $h^0(-K)=29, \ (-K)^4=110, \  h^{1,1}=2, \ h^{3,1}=1, \ h^{2,2}=22$, $-\chi(T)=21$.
\subsubsection*{Description} 
\begin{itemize}
    \item $\pi_1 \colon X \to \PP^4$ is a blow up $\Bl_{S_{12}^{(1)}} \PP^4$, where $S_{12}^{(1)}$ is a K3 surface of degree 12 blown up in a point.
    \item $\pi_2 \colon X \to \Gr(2,5)$ is a blow up $\Bl_{S_{12}}X_5$ of 
    a K3 surface of genus 7
    in a Fano fourfold $X_5$ of  degree 5 and  index 3.
    \item Rationality: yes.
\end{itemize}

\subsubsection*{Semiorthogonal decompositions}
\begin{itemize}
    \item 6 exceptional objects and $\Db(S_{12})$ from both maps. Note that the surfaces of degree 12 are derived equivalent (see \cref{lem:dereq}) but we don't know if they are isomorphic.
\end{itemize}

\subsubsection*{Explanation} We can interpret the flag $\Fl(1,2,5)$ as $\PP_{\Gr(2,5)}(\mU)$. The two $\of(0,1)$ sections cut the base Grassmannian in the fourfold $X_5$. We are left with $\mZ( \PP_{X_5}(\mU|_{X_5}),\mL \otimes \of_{X_5}(1))$, where $\mL$ denotes the relative line bundle, or equivalently with $\mZ(\PP_{X_5}(\mU(-1)|_{X_5}),\mL )$. We conclude by the Cayley trick, since $\mZ(\Gr(2,5)),\of(1)^{\oplus 2} \oplus \mU^{\vee}(1))$ is a general K3 surface of degree 12.

When $l$ is a line in $V_5$, the planes that contain it generate a linear space 
$l\wedge V_5$
in $\wedge^2V_5$. This defines a vector bundle $\mK\simeq \mQ(-1)$, of rank four. The fibers of 
the projection of $X$ to $\PP=\PP^4$ are defined by the image of the induced morphism 
$$\of_\PP^{\oplus 2}\oplus \of_\PP(-1)\longrightarrow \mK^\vee,$$
which drops rank along a smooth surface $S$ whose structure sheaf is resolved by the Eagon--Northcott complex 
$$0\longrightarrow \of_\PP(-4)^{\oplus 2}\oplus \of_\PP(-1)\longrightarrow \mK^\vee(-4)\longrightarrow  \of_\PP\longrightarrow  \of_S\longrightarrow 0.$$
By Borel--Weil--Bott Theorem, $\mK^\vee(-4)$ and $\mK^\vee(-3)$ are acyclic, while $\mK^\vee(-2)$ has $h^1=1$. 
We deduce that $\chi( \of_S)=2$, while $\chi(\of_S(1))=5$ and $\chi(\of_S(2))=15+1=16$. So the Hilbert polynomial must be
$$\chi(\of_S(k))=4k^2-k+2.$$

\begin{proposition}
$S$ is a K3 surface of genus seven 
blown up at one point, and 
embedded in $\PP^4$ by a tangential projection from $\PP^7$.
\end{proposition}

This can be confirmed using Proposition 2.7 in \cite{Oko86}.
\end{fano}

\medskip\begin{fano}\fanoid{2-27-100-2} 
\label[mystyle]{27-100-2-1-24-23-1-6-2-4} $\mZ(\PP^5 \times \Gr(2,4),\mU_{\Gr(2,4)}^{\vee}(1,0) \oplus \of(1,1) \oplus \mQ_{\Gr(2,4)}(1,0)).$

\subsubsection*{Invariants}  $h^0(-K)=27, \ (-K)^4=100, \  h^{1,1}=2, \ h^{3,1}=1, \ h^{2,2}=24$, $-\chi(T)=23$.
\subsubsection*{Description} 
\begin{itemize}
    \item $\pi_1,\pi_2 \colon X \to \PP^5,\Gr(2,4)$ are both blow ups $\Bl_{S_{12}^{(2)}}\QQ^4$ 
    of a quadric along a K3 surface of genus seven blown up in two points.
    \item Rationality: yes.
\end{itemize}

\subsubsection*{Semiorthogonal decompositions}
\begin{itemize}
    \item 8 exceptional objects and $\Db(S_{12})$ from both maps.  Note that the surfaces of degree 12 are derived equivalent (see \cref{lem:dereq}) but we don't know if they are isomorphic.
\end{itemize}

\subsubsection*{Explanation}
Let us start with the fivefold $Y$ defined by the two rank two bundles and their global sections, which are tensors $\alpha\in V_6^\vee\otimes V_4=\Hom(V_6,V_4)$ and 
$\beta\in V_6^\vee\otimes V_4^\vee=\Hom(V_6,V_4^\vee)$. 
So $Y\subset\PP^5\times \Gr(2,4)$
parametrizes pairs $(l \subset V_6, P\subset V_4)$ such that 
$\alpha(l)\subset P\subset \beta(l)^\perp.$
In particular its image in $\PP^5$ is the 
quadric $Q$ defined by the equation $\langle \beta(l), \alpha(l)\rangle =0$, which is smooth in general. Moreover the generic fiber
is a projective line, and the exceptional 
fibers are 
projective planes over the two disjoint lines 
$\delta=\PP(\Ker(\alpha))$ and $\delta'=\PP(\Ker(\beta))$.

The fiber over $P$ of the projection to $\Gr(2,4)$ 
is the intersection of the  three-planes $\PP(\alpha^{-1}(P))$ and $\PP(\alpha^{-1}(P^\perp))$, hence a projective line in general. The exceptional fibers are linear spaces of 
bigger dimension, 
exactly over the degeneracy locus 
of the morphism 
$$\mU\oplus \mQ^\vee\longrightarrow V_6^\vee \otimes \of_{\Gr(2,4)}$$
induced by $\alpha$ and $\beta$. Since the bundle of those morphisms is generated by global sections, for $\alpha,\beta$ generic this degeneracy locus is a smooth curve $C$, defined by the associated Eagon--Northcott complex
$$0\longrightarrow S^2(\mU\oplus \mQ^\vee)
\longrightarrow L_6\otimes (\mU\oplus \mQ^\vee)\longrightarrow L_{15}\longrightarrow \of_{\Gr(2,4)}(2)\longrightarrow \of_C(2)\longrightarrow 0,$$
where $L_k$ is trivial of rank $k$. 
The first bundle of this complex contains a factor 
$\mU\otimes \mQ^\vee$ isomorphic to the cotangent of 
$\Gr(2,4)$, which has therefore a one-dimensional first cohomology group. This is the only non-trivial contribution to the cohomology of this complex and we deduce that 
$\chi(\of_C(2))=20-15+1=6$. After a twist 
by $\of_{\Gr(2,4)}(-1)$, all the terms are acyclic by Borel--Weil--Bott Theorem except 
$S^2\mU(-1)$ and $S^2\mQ^\vee(-1)$ which have a one 
dimensional second cohomology group, hence 
$\chi(\of_C(1))=6-2=4$. 
After another twist, 
all the terms are acyclic except $\mU\otimes \mQ^\vee(-2)$  which has a one 
dimensional third cohomology group, hence 
$\chi(\of_C)=1+1=2$. This is more than 
enough to deduce that the Hilbert polynomial 
of $C$ is $\chi(\of_C(k))=2k+2$, so that 
$C$ must be the union of two disjoint lines $d$ and $d'$. 
Note the symmetry of the picture:
\begin{equation*}
\xymatrix{
 & E \ar[dl]_{\PP^2} \ar@{^{(}->}[r] & Y  \ar[ld]^{\PP^1}\ar[rd]_{\PP^1} & F \ar[dr]^{\PP^2} \ar@{_{(}->}[l]\\
 \delta\cup \delta'\ar@{^{(}->}[r] &  Q & & \Gr(2,4) & d\cup d'\ar@{_{(}->}[l].
 }
 \end{equation*}
 
\medskip
Now we turn to the fourfold $X$ defined as the hyperplane section of $Y$ defined by a section of 
$\of(1,1)$, hence a tensor $\gamma \in V_6^\vee\otimes\wedge^2V_4^\vee$. The general
fibers of the two projections are projective lines
cut by a hyperplane, hence single points, so that
the two projections are birational. 
Over $\Gr(2,4)$, the exceptional fibers
are obtained over a locus $S$ that can be described as before. The sections that define 
$X$ induce a morphism 
$$\mF:= \mU\oplus \mQ^\vee\oplus\of_{\Gr(2,4)}(-1)
\longrightarrow V_6^\vee \otimes \of_{\Gr(2,4)}.$$
Generically, this morphism is surjective, its kernel is a line in $V_6$ defining the point in the preimage. The rank drops in codimension two,  
over a smooth surface $S$ over which the fibers of the projection are projective lines; so the projection to $\Gr(2,4)$ should be the blowup of $S$. 
Note that $S$ has to contain the two lines $d$ and $d'$. 
\begin{equation*}
\xymatrix{
 & E_X \ar[dl]_{\PP^1} \ar@{^{(}->}[r] & X  \ar[ld]^{\PP^1}\ar[rd]_{\PP^0} & F_X\ar[dr]^{\PP^1} \ar@{_{(}->}[l]\\
 \delta\cup \delta'\subset T\quad\ar@{^{(}->}[r] &  Q & & \Gr(2,4) & \quad S\supset d\cup d'\ar@{_{(}->}[l].
 }
 \end{equation*}
 
\begin{proposition}
$S$ is the blow up at two points of a K3 surface 
of genus seven,  $d$ and $d'$  being the 
two exceptional curves of the surface. 
\end{proposition}

\begin{proof}
The ideal of $S$ in $\Gr(2,4)$ is resolved by  the Eagon--Northcott complex
$$0\longrightarrow \mF\otimes\det(\mF) \longrightarrow 
V_6\otimes \det(\mF)\longrightarrow\of_{\Gr(2,4)}
\longrightarrow\of_S\longrightarrow 0.$$
We have $\det(\mF)=\of_{\Gr(2,4)}(-3)$,  so $\mF\otimes \det(\mF)$ has a factor $\of_{\Gr(2,4)}(-4)=K_{\Gr(2,4)}$ which implies that 
$h^2(\of_S)=1$. We compute from the previous complex that the Hilbert polynomial of $S$ is 
$\chi(\of_S(k))=5k^2-k+2,$ so that $H^2=10$ and $H\omega_S=2$. Moreover, 
dualizing the Eagon--Northcott complex we get 
$$0\longrightarrow \of_{\Gr(2,4)} \longrightarrow 
V_6(-1)\longrightarrow \mF^\vee(-1)=\mU\oplus \mQ^\vee\oplus\of_{\Gr(2,4)}
\longrightarrow\omega_S\longrightarrow 0.$$
In particular $\omega_S$ admits a canonical section, which vanishes exactly on the locus 
where the $V_6(-1)\longrightarrow \mU\oplus \mQ^\vee$ is not surjective. But this locus is 
exactly $C=d\cup d'$! We conclude that $\omega_S=d+d'$. Then by the genus formula $d$ and $d'$ are 
two $(-1)$-curves, and after contracting them we get a surface $S_{min}$ with trivial canonical bundle 
and  $h^2(\of_{S_{min}})=1$, hence a K3 surface.
\renewcommand\qedsymbol{$\square$}
\end{proof}

Note that a K3 surface of degree $12$, hence genus $7$, can be embedded in $\PP^7$; projecting 
from two of its points, one obtains in $\PP^5$ the blow up of the surface at the two chosen points.

\medskip 
Now we turn to the other projection to the quadric $Q$ defined by $\alpha$ and $\beta$. 
The fiber over $l\in Q$ is the set of planes $P$ such that $\alpha(l)\subset P\subset\beta(l)^\perp$
and $\gamma(l,p\wedge p')=0$ for any $p,p'\in P$, where $\gamma\in V_6^\vee\otimes\wedge^2V_4^\vee$ 
is the tensor defining the section of $\of(1,1)$. For $l$ general, the first condition defines a 
$\PP^1$, on which the second condition, being linear, cuts a single point; the whole $\PP^1$ is in
the fiber exactly when  $\gamma(l,\alpha(l)\wedge \beta(l)^\perp)=0$, a codimension two condition 
that defines the surface $T$ over which $X$ has non-trivial fibers. Note that the lines $\delta$
and $\delta'$ are contained in $T$. 

Generically, $\alpha(l)\wedge \beta(l)^\perp$ is a dimension two subspace of $\wedge^2 V_4$, but since
the rank jumps over $\delta$ and $\delta'$ we will need to blow up these two lines in $Q$; denote
by $\tilde{Q}$ the resulting fourfold, with $\Delta$ and $\Delta'$ the two exceptional divisors.
Let $L$ be the pullback of the hyperplane bundle on $Q$. On $\tilde{Q}$ we have a complex 
$$L^\vee(\Delta)\hookrightarrow V_4\otimes \of_{\tilde{Q}}\longrightarrow L(-\Delta'),$$
where the two morphisms have constant rank one, so the kernel of the second one is a rank three
bundle $\mM$ containing the image of $L^\vee(\Delta)$ via the first one. As a consequence, we get a 
rank two vector sub-bundle $\mP=L^\vee(\Delta)\wedge \mM$ of $\wedge^2 V_4\otimes \of_{\tilde{Q}}$, and 
then a surface $\tilde{T}$ in $\tilde{Q}$ defined by the vanishing of a section of $\HOM(\mP,L)$.
By construction, $\tilde{T}$ is a smooth surface dominating $T$. Using the fact that $\det(\mP)=
L^\vee(\Delta)\otimes\det(\mM)$ we compute, by adjunction, that $\omega_{\tilde{T}}=C+C'$ is the 
restriction of $\Delta+\Delta'$. 

Now consider the Koszul complex that resolves the structure sheaf of $\tilde{T}$, namely
$$0\longrightarrow \of(-4L+\Delta+\Delta') \longrightarrow \mP(-L)
\longrightarrow\of_{\tilde{Q}}
\longrightarrow\of_{\tilde{T}}\longrightarrow 0.$$
Using the fact that $\Delta\Delta'=0$ in the Chow ring of $\tilde{Q}$, since these two divisors
do not intersect, we deduce that the fundamental class of $\tilde{T}$ is given by 
$$[\tilde{T}]=5L^2-3L(\Delta+\Delta')+(\Delta+\Delta')^2.$$
Denoting by $\lambda$ the restriction of $L$, we deduce that $\lambda^2=5L^4=10$. 
Moreover $\lambda C=[\tilde{T}]L\Delta=5L^3\Delta-3L^3\Delta^2+L\Delta^3=L\Delta^3$; indeed,
$\Delta$ dominates a line in $Q$ so $L^2\Delta=0$. Of course $L\Delta^3$ can be computed by 
restricting to $\Delta$: since $L$ cuts a single point of $\delta$ over which $\Delta$ is just $\of_{\PP^2}(-1)$, we get $\lambda C=L\Delta^3=1$. Again restricting to $\Delta$ and applying the usual projection formula for projective bundles, we get that $\Delta^4=s_1(N_\delta)=s_1(TQ_{|\delta})-s_1(T\delta)=4-2=2$. This yields 
$$C^2=[\tilde{T}]\Delta^2=5L^2\Delta^2-3L\Delta^3+\Delta^4=5\times 0-3\times 1+2=-1.$$
Note that $\lambda C=1$ means that $C$ is projecting isomorphically to $\delta$ in $T$, so that 
$\tilde{T}$ is in fact isomorphic to $T$. Moreover, $C$ and $C'$ are two $(-1)$-curves and since 
$\omega_{\tilde{T}}=C+C'$, we get a genuine K3 surface $T_{min}$ after contracting them.  
We have finally proved:

\begin{proposition}
The first projection of $X$ to the smooth quadric $Q$ is also a doubly blown up K3 surface $T$ 
of genus seven.
\end{proposition}
 
We finally get a non-trivial 
birational transformation of the four-dimensional quadric, as the  composition  of the blow up 
of $S$ with the inverse of the blow up of $T$. In fact, we can compute the canonical bundle of 
$X$ in terms of the two hyperplane divisors as $\omega_X=-H-H'$, but also in terms of the 
exceptional divisors as $\omega_X=-4H+E_X=-4H'+F_X$; in particular $H'=3H-E_X$, which means that 
our birational involution of $Q$ is defined by the linear system of cubics containing $T$ (or $S$). 
Note the striking similarity with the birational involution of $\PP^3$ defined by the linear 
system of cubics containing a curve of degree $6$ and genus $3$ \cite{katz-veryspecial}. Note also that since $\chi(T)=-23$, 
we expect a family with $23=19+2+2$ parameters, meaning that the K3 surface and its two points 
should be generic. 
\end{fano}
 
 \begin{rmk}
The two K3 surfaces $S_{min}$ and $T_{min}$ are derived equivalent by \cref{lem:dereq}. We cannot prove that they are not isomorphic,
but note that a generic K3 surface of genus seven has 
exactly one Fourier-Mukai partner \cite{oguiso02}, which can be constructed in
terms of (HP) dual sections to the spinor tenfold. Presumably $S_{min}$ and $T_{min}$ should be non-isomorphic Fourier-Mukai partners. Moreover, our construction indicates that generically, choosing two points on 
$S_{min}$ also yields  two points on $T_{min}$, in perfect agreement with the fact that the punctual Hilbert schemes $S_{min}^{[2]}$ and $T_{min}^{[2]}$ are birationally equivalent \cite{meachan-mongardi-yoshioka}.

 \end{rmk}

\medskip\begin{fano}\fanoid{2-26-95-2} 
\label[mystyle]{26-95-2-1-23-22-1-4-2-5}
 $\mZ(\PP^3 \times \Gr(2,5),\of(0,1)^{\oplus 2} \oplus \mU^{\vee}_{\Gr(2,5)}(1,0) \oplus \of(1,1))$.

\subsubsection*{Invariants}  $h^0(-K)=26, \ (-K)^4=95, \  h^{1,1}=2, \ h^{3,1}=1, \ h^{2,2}=23$, $-\chi(T)=22$.
\subsubsection*{Description} 
\begin{itemize}
    \item $\pi_1 \colon X \to \PP^3$ is a conic bundle over $\PP^3$ with discriminant locus $S_8^{(1)}$,  the blow up of a degree 8 K3 surface in one point.
    \item $\pi_2 \colon X \to \Gr(2,5)$ is a blow up $\Bl_{S_8^{(1)}} X_5$, with $S_8^{(1)}$ as above, and $X_5$ a Fano fourfold of index $3$ and degree $5$.
    \item Rationality: yes.
\end{itemize}

\subsubsection*{Semiorthogonal decompositions}
\begin{itemize}
    \item 4 exceptional objects and $\Db(\PP^3,\mC_0)$ from $\pi_1$.
    \item 7 exceptional objects and $\Db(S_8)$ from $\pi_2$.
\end{itemize}

\subsubsection*{Explanation} 
Let us start with the fivefold $Y$ defined by two sections of $\of(0,1)$, that is 
two skew-symmetric forms on $V_5$, and a section of $\mU^\vee(1,0)$, that is a tensor $\alpha\in V_4^\vee\otimes V_5^\vee$. So $Y$ parametrizes the pairs $(l, P)$, with $P$ a doubly 
isotropic plane in $V_5$, and $l$ a line in $V_4$ such that the linear form $\alpha(l)$
vanishes on $P$. For $\alpha$ generic, $\alpha(l)$ is never zero, and the fiber of $Y$ over $l$ is a codimension two linear section of $\Gr(2,\alpha(l)^\perp)$. On the other hand,
if $P$ is fixed there are two possibilities. Generically, $\alpha^t(P)$ is a plane in $V_4^\vee$
and the fiber of $Y$ over $P$ is a projective line. But if $P$ contains the kernel of $\alpha^t$, 
the fiber is a projective plane: this happens over a projective line $d$. 
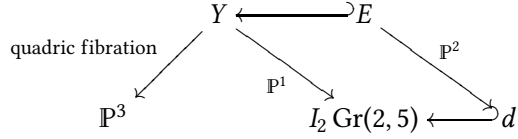
\begin{equation*}
\xymatrix{
   & Y  \ar[ld]_{\mathrm{quadric\;fibration}} \ar[rd]_{\PP^1} & E\ar@{_{(}->}[l] 
 \ar[rd]^{\PP^2}\\
   \PP^3 & & I_2\Gr(2,5) &  d\ar@{_{(}->}[l].
 }
 \end{equation*}

Now we turn to the fourfold $X$ defined in $Y$ by a section of $\of(1,1)$, that is a
tensor $\theta\in V_4^\vee\otimes\wedge^2V_5^*$. The fibers of the projection to $Z=I_2\Gr(2,5)$
are defined by the image of the induced morphism 
$$\mU_Z\oplus \of_Z(-1)\longrightarrow V_4^\vee\otimes \of_Z,$$
which drops rank along a smooth surface $S$ whose structure sheaf is resolved by the 
Eagon--Northcott complex
$$0\longrightarrow \mU_Z\oplus \of_Z(-1)\longrightarrow V_4^\vee\otimes \of_Z\longrightarrow \of_Z(2)\longrightarrow \of_S(2)\longrightarrow 0.$$
On the Grassmannian $\Gr(2,5)$, Borel--Weil--Bott Theorem implies that $\mU(-k)$ is acyclic for $0\le k\le 4$. 
The Koszul complex of $Z$ in $\Gr(2,5)$ allows us to deduce that $\mU_Z(-k)$ is acyclic for $0\le k\le 2$.
Moreover $\chi(\of_Z(2))=31$, hence $\chi(\of_S(2))=27$. Also 
$\chi(\of_S(1))=\chi(\of_Z(1))=8$, and $\chi(\of_S)=\chi(\of_Z)
+\chi(\of_Z(-3))=1+1=2$. We deduce the Hilbert polynomial
$$\chi(\of_Z(k))=\frac{13k^2-k}{2}+2.$$
Using the same kind of arguments as we already did many times, we get:

\begin{proposition}
$S$ is a K3 surface of genus $8$ blown up at a single point, and $d$ is its exceptional curve. 
\end{proposition}

 Now we turn to the conic bundle of $X$ over $\PP^3$. 
 
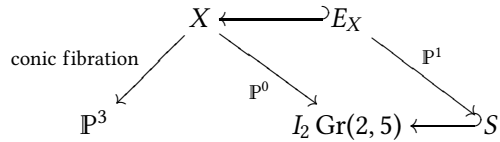
\begin{equation*}
\xymatrix{
   & X  \ar[ld]_{\mathrm{conic\;fibration}} \ar[rd]_{\PP^0} & E_X\ar@{_{(}->}[l] 
 \ar[rd]^{\PP^1}\\
   \PP^3 & & I_2\Gr(2,5) &  S\ar@{_{(}->}[l].
 }
 \end{equation*}
 
 To understand it better, notice how the fourfold $X$ can be equivalently described as $\mZ(\Fl(1,3,5),\of(1,0) \oplus \of(0,1)^{\oplus 2} \oplus \of(1,1))$. The first bundle cuts the flag in $\mathbb{G}_{\PP^3}(2, \mQ(-1) \oplus \of(-1))$. Therefore this becomes a quadric fibration if we cut with $\of(0,1)^{\oplus 2}$, and a conic bundle with discriminant locus $S$ when we cut with the last $\of(1,1)$.
\end{fano}

\medskip\begin{fano}\fanoid{2-25-90-2-B} 
\label[mystyle]{25-90-2-1-24-23-2-4-2-4} $\mZ(\Gr(2,4)_1 \times \Gr(2,4)_2,\of(0,1) \oplus \mU^{\vee}_{\Gr(2,4)_2}(1,0) \oplus \of(1,1))$.
\subsubsection*{Invariants}  $h^0(-K)=25, \ (-K)^4=90, \  h^{1,1}=2, \ h^{3,1}=1, \ h^{2,2}=24$, $-\chi(T)=23$.
\subsubsection*{Description} 
\begin{itemize}
    \item $\pi_1 \colon X \to \Gr(2,4)_1$ is birational, with base locus a genus 6 K3 surface with two double points.
    \item $\pi_2 \colon X \to \Gr(2,4)_2$ is a conic bundle over $\QQ^3$, with discriminant a quartic hypersurface with $20$ ODP.
    \item Rationality: yes.
\end{itemize}

\subsubsection*{Semiorthogonal decompositions}
\begin{itemize}
    \item 4 exceptional objects and $\Db(\QQ^3,\mC_0)$ from $\pi_2$. There is no hint for the K3 category.
\end{itemize}

\subsubsection*{Explanation}
Over $\Gr(2,U_4)\times \Gr(2,V_4)$, with tautological bundles $\mU$ and $\mV$, we have 
the line bundles $\det(\mV)^\vee$ and $\det(\mU)^\vee\otimes\det(\mV)^\vee$, and the rank two bundle 
$\det(\mU)^\vee \otimes \mV$. Their global sections are given by tensors $\omega\in \wedge^2V_4^\vee$, 
$\theta\in \wedge^2U_4^\vee\otimes \wedge^2V_4^\vee$ and $\alpha\in \wedge^2U_4^\vee\otimes V_4$.
The sections $\omega$ and $\alpha$ define a fivefold $Y$ parametrizing the pairs $U_2,V_2$
of planes such that $V_2$ is $\omega$-isotropic and contains $\alpha(\wedge^2U_2)$. 
So $V_2$ has to belong to $IG_\omega(2,V_4)\simeq\QQ^3$ and its fiber in $Y$ is a codimension
two linear section of $\Gr(2,U_4)$. So $Y$ is a quadratic surface bundle over $\QQ^3$, with a quadric hypersurface  for discriminant.

When we take into account the remaining section $\theta$, we get our Fano fourfold $X$ 
as a conic bundle over $\QQ^3$, with discriminant $\Delta$. 

\begin{lemma}
In general, the discriminant $\Delta\subset\QQ^3$ is a quartic surface with 
exactly $20$ ODP. 
\end{lemma}

\begin{proof} The conic bundle over $\QQ^3\subset \Gr(2,V_4)$ is contained in $\PP(\mE)$,
for $\mE$ a rank three vector bundle defined by an exact sequence 
$$ 0\lra \mU\oplus \of_{\QQ^3}(-1)\lra \wedge^2U_4^\vee\otimes \of_{\QQ^3}\lra\mE^\vee\lra 0. $$
The conic bundle is then defined by the natural morphism $S^2\mE\ra S^2(\wedge^2U_4)\otimes \of_{\QQ^3}\ra \det(U_4)\otimes \of_{\QQ^3} $.
The Chern class of $\mE$ is given in terms of the Schubert cycles $\sigma_1$ and $\sigma_2$ by 
$$c(\mE)=\frac{s(\mU^\vee)}{1-\sigma_1}=\frac{1+\sigma_1+\sigma_2}{1-\sigma_1}.$$
We deduce that $c_1(\mE)=2\sigma_1$, $c_2(\mE)=2\sigma_1^2+\sigma_2$,  $c_3(\mE)=2\sigma_1^3+\sigma_1\sigma_2$, hence $c_1(\mE)c_2(\mE)=2c_3(\mE)$.
Applying \cref{lem:conicbundle},
we compute the number $N$ of singular points in $\Delta$, as 
$$N=4\int_{\QQ^3}(c_1(\mE)c_2(\mE)-c_3(\mE))=4\int_{\Gr(2,4)}\sigma_1(2\sigma_1^3+\sigma_1\sigma_2)=4(2\times 2+1)=20.$$
Moreover these must be ordinary double points in general.
\renewcommand\qedsymbol{$\square$}
\end{proof}

\medskip
On the other hand, if  $\alpha(\wedge^2U_2)$ is non-zero, the $\omega$-isotropic planes in $V_4$ containing this line form a projective line. Generically, $\alpha(\wedge^2U_2)=0$ for exactly 
two transverse planes $P_1$ and $P_2$, defining two points $p_1, p_2$ of $\Gr(2,U_4)$ whose fibers $Q_1$ and $Q_2$ are copies of   $IG_\omega(2,V_4)\simeq\QQ^3$. 
Hence the diagram:
\begin{equation*}
\xymatrix{
 &  Q_1\cup Q_2\ar[ld]_{\mathbb{Q}^3}\ar@{^{(}->}[r]  & Y \ar[ld]^{\PP^1}
 \ar[rrd]^{\mathrm{\quad quadric\;fibration}} &&\\
 p_1\cup p_2\ar@{^{(}->}[r] & \Gr(2,U_4) & & & \QQ^3 
 }
 \end{equation*}

\smallskip
When we cut $Y$ by the hyperplane section defined by $\theta$, this shows that 
our fourfold
$X$ projects birationally to $\Gr(2,U_4)$. Where are the exceptional fibers? 
They occur exactly when the projective line in $\Gr(2,V_4)$ defined by the 
conditions that $\alpha(\wedge^2U_2)\subset V_2\subset \alpha(\wedge^2U_2)^\perp$ 
is contained in the hyperplane defined by $\theta(\wedge^2U_2)$, which means that 
$\theta(\wedge^2U_2)$ vanishes on the linear space $P_U=\alpha(\wedge^2U_2)\wedge \alpha(\wedge^2U_2)^\perp
\subset \wedge^2V_4$. But this is meaningful only when $\alpha(\wedge^2U_2)\ne 0$; 
in other words, $P_U$ does not define a vector bundle on the Grassmannian. 

In order to overcome this problem, we need to pass to the blow up $\tilde{G}$ of
$\Gr(2,U_4)$ at the two points $p_1=[P_1]$ and $p_2=[P_2]$. Denote by $E=E_1+E_2$ the 
exceptional divisor, and by $H$ the pull-back of the hyperplane divisor. 
Over $\tilde{G}$, $\alpha$ and its contraction with $\omega$ define 
two embeddings of vector bundles 
$$\mathcal{L}\simeq\of_{\tilde{G}}(E-H)\hookrightarrow V_4\otimes \of_{\tilde{G}}, 
\qquad 
\mathcal{L'}\simeq\of_{\tilde{G}}(E-H)\hookrightarrow V_4^\vee\otimes \of_{\tilde{G}}.$$ 
Dualizing the second one, we get a rank three sub-bundle $\mathcal{M}$ of $V_4\otimes \of_{\tilde{G}}$, and by the skew-symmetry of $\omega$ this bundle contains 
$\mathcal{L}$ as a sub-bundle. So $\mathcal{P}:=\mathcal{L}\wedge \mathcal{M}$ is now 
a genuine rank two sub-bundle of $\wedge^2V_4\otimes \of_{\tilde{G}}$. As we 
observed, the exceptional locus of our projection is now described by the vanishing 
of the induced section of $\mathcal{P}^\vee(H)$. This must be a smooth surface $S$, 
whose structure sheaf is resolved by the Koszul complex 
$$0\longrightarrow \wedge^2\mathcal{P}(-2H)\longrightarrow \mathcal{P}(-H)\longrightarrow \of_{\tilde{G}}\longrightarrow \of_S\longrightarrow 0.$$
Note that $\wedge^2\mathcal{P}\simeq \mathcal{L}^2$, and that $\mathcal{L}^2(-2H)=\omega_{\tilde{G}}(-E)$, which will imply that $h^2(\of_S)=1$. Moreover, by adjunction $\omega_S=\of_S(E)$.

In order to compute the cohomology of $\mathcal{P}$, observe that it is a sub-bundle of $\mathcal{K}:=\mathcal{L}\wedge V_4$, 
with quotient $\mathcal{L}\otimes (V_4/\mathcal{M})$, which is trivial. We thus get the 
exact sequence 
$$0\longrightarrow \mathcal{P}\longrightarrow \mathcal{K}\longrightarrow \of_{\tilde{G}}
\longrightarrow 0,$$
while the rank three vector bundle $\mathcal{K}$ is given by the simple sequence 
$$0\longrightarrow \mathcal{L}^2\longrightarrow \mathcal{L}\otimes V_4\longrightarrow \mathcal{K}
\longrightarrow 0.$$
From the latter we readily deduce that $ \mathcal{K}(-H)$ is acyclic, and then that 
$ \mathcal{P}(-H)$ is acyclic as well, and therefore that $\chi(\of_S)=2$. 
After a twist by $H$, $\chi(\mathcal{P})=-1$ yields $\chi(\of_S(H)))=6-(-1)=7$.
After another twist,  $\chi(\mathcal{K}(H))=4$ implies $\chi(\mathcal{P}(H))=-2$
and then $\chi(\of_S(H)))=20-(-2)=22$. Therefore we get the Hilbert polynomial 
$$\chi(\of_S(kH)))= 5k^2+2.$$

\begin{proposition}
$S$ is a K3 surface of genus six blown up at four points.
\end{proposition}

Observe that in the Chow ring of $\tilde{G}$ we have $HE=0$. It is then easy to compute that the class of $S$ is 
$$[S]=c_2( \mathcal{P}^\vee(H))=\frac{(1+2H-E)^4}{(1+H)(1+3H-2E)}_{|\deg=2}=5H^2+2E^2.$$
Since $K_S=E_S$, we deduce that $K_S^2=[S]E^2=2E^4=-4$. This is explained as follows. 
Over $p_1$ or $p_2$, the exceptional divisor is a copy of $\PP^3$ and the restriction 
of $\mathcal{L}$ is the tautological line bundle. Therefore, the restriction of $\mathcal{K}$ is the so-called {\it null-correlation bundle}  $\mathcal{E}:=\mathcal{L}\wedge\mathcal{L}^\perp$. There is an exact sequence 
$$0 \longrightarrow \mathcal{E}\longrightarrow T_{\PP^3}(-2)=\mathcal{L}\wedge V_4\longrightarrow \mathcal{L}\otimes (V_4/\mathcal{L}^\perp)=\of_{\PP^3}\longrightarrow 0,$$
which implies that $h^1(\mathcal{E})=1$, and $h^k(\mathcal{E})=0$ for $k\ne 1$. Consider  then a general 
section of the globally generated vector bundle $\mathcal{E}^\vee$. It vanishes on a smooth 
curve $C$ such that $K_C=\of_C(-2)$ since $\det(\mathcal{E})=\of_{\PP^3}(-2)$.
Moreover the Koszul complex 
$$0 \longrightarrow \of_{\PP^3}(-2)\longrightarrow \mathcal{E} \longrightarrow \of_{\PP^3}\longrightarrow \of_C\longrightarrow 0$$
shows that $h^0(\of_C)=h^0(\of_{\PP^3})+h^1(\mathcal{E})=2$. We conclude that the general section of the dual of the null-correlation bundle vanishes on the 
union of two skew lines. 

This gives us four lines in $S$, which are four $(-1)$-curves on $S$. Contracting them we get a genuine K3 surface $S_{min}$ of genus $6$. Finally, note that these four lines are contracted  in pairs to 
$p_1$ and $p_2$ in $\Gr(2,U_4)$; so the image of $S$, which is the base
locus of the birational projection of $X$ to $\Gr(2,U_4)$, is a non-normal image of $S_{min}$ obtained 
by identifying two pairs of points. 
\end{fano}

We remark that there exists a link between this Fano variety and \cref{30-114-2-1-24-24-2-4-2-5}, investigated in the recent \cite{nodalK3}.

\medskip\begin{fano}\fanoid{3-19-60-2} 
\label[mystyle]{19-60-3-5-1-22-23-1-2-1-5} $\mZ(\PP^1 \times \PP^1 \times \PP^4,\of(0,0,3) \oplus \of(1,1,1))$.

\subsubsection*{Invariants}  $h^0(-K)=19, \ (-K)^4=60, \  h^{1,1}=3, \ h^{2,1}=5, \ h^{3,1}=1, \ h^{2,2}=22$, $-\chi(T)=23$.
\subsubsection*{Description} 
\begin{itemize}
    \item $\pi_{12} \colon X \to \PP^1 \times \PP^1$ is a $\dP_3$-fibration.

    \item $\pi_{13},\pi_{23} \colon X \to \PP^1 \times W_3$ are  blow ups $\Bl_{S_{12}}(\PP^1 \times W_3)$, for a general cubic threefold $W_3$, along a K3 surface of degree 12 and Picard number 2.
    \item Rationality: unknown.
\end{itemize}

\subsubsection*{Semiorthogonal decompositions}
\begin{itemize}
    \item 4 exceptional objects and an unknown category from $\pi_{12}$.
    \item 4 exceptional objects, two $[5/3]$-CY categories and $\Db(S_{12})$ from $\pi_{13}$ or  $\pi_{23}$.
\end{itemize}

\subsubsection*{Explanation} We use \cref{lem:blowupcodim2} to get the blow up 
structure given by $\pi_{23}$ and $\pi_{13}$. 
We check directly that $\pi_{12}$ has general fiber a $\dP_3$.
\end{fano}

\medskip\begin{fano}\fanoid{3-23-80-2-A} 
\label[mystyle]{23-80-3-2-1-33-20-1-2-1-2-1-6} $\mZ(\PP^1 \times \PP^1 \times \PP^5,\of(0,0,2)^{\oplus 2} \oplus \of(1,1,1))$.

\subsubsection*{Invariants}  $h^0(-K)=23, \ (-K)^4=80, \  h^{1,1}=3, \ h^{2,1}=2, \ h^{3,1}=1, \ h^{2,2}=33$, $-\chi(T)=20$.
\subsubsection*{Description}
\begin{itemize}
    \item $\pi_{12} \colon X \to \PP^1 \times \PP^1$ is a $\dP_4$-fibration.
    \item $\pi_{13},\pi_{23} \colon X \to \PP^1 \times \PP^5$ are $\Bl_{S_{16}}(\PP^1 \times W_4)$, where $W_4$ is the intersection of two quadrics in $\PP^5$.
\item Rationality: Yes. 
\end{itemize}

\subsubsection*{Semiorthogonal decompositions}
\begin{itemize}
    \item 4 exceptional objects and $\Db(Z, \mC_0)$, for $Z \to \PP^1 \times \PP^1$ a $\PP^1$-bundle from $\pi_{12}$.
    \item 4 exceptional objects, two copies of $\Db(C)$ for $C$ a genus 2 curve, and $\Db(S_{16})$ from $\pi_{13},\pi_{23}$.
\end{itemize}

\subsubsection*{Explanation} Simple application of \cref{lem:blowupcodim2}. The K3 surface in $\PP^1 \times \PP^5$ can be checked to have degree 16 with respect to $\of(1,1)$, and of course Picard rank $ 2$.
\end{fano}

\medskip\begin{fano}\fanoid{3-27-100-2-A} 
\label[mystyle]{27-100-3-1-22-18-1-2-1-2-2-5} $\mZ(\PP^1 \times \PP^1 \times \Gr(2,5),\of(0,0,1)^{\oplus 3} \oplus \of(1,1,1) )$.
\subsubsection*{Invariants}  $h^0(-K)=27, \ (-K)^4=100, \  h^{1,1}=3, \ h^{3,1}=1, \ h^{2,2}=22$, $-\chi(T)=18$.
\subsubsection*{Description} 
\begin{itemize}
    \item $\pi_{12} \colon X \to \PP^1 \times \PP^1$ 
    is a $\dP_5$-fibration,
    \item $\pi_{13},\pi_{23} \colon X \to \PP^1 \times \Gr(2,5)$ are blow ups $\Bl_{S_{20}}(\PP^1 \times W_5)$, with $S_{20}$ a K3 surface of degree 20 and generic Picard rank 2, and $W_5$ a Fano threefold of index 2 and degree 5.
    \item Rationality: yes.
\end{itemize}

\subsubsection*{Semiorthogonal decompositions}
\begin{itemize}
    \item 8 exceptional objects and $\Db(Z)$, for $Z \to \PP^1 \times \PP^1$ a flat degree 5 cover, from $\pi_{12}$.
    \item 8 exceptional objects and $\Db(S_{20})$ from $\pi_{13}$ or $\pi_{23}$.
    \item The two decompositions are not known to coincide but it is natural to imagine  that $Z$ should be a K3 surface, derived equivalent to $S_{20}$.
\end{itemize}

\subsubsection*{Explanation}
It suffices to apply \cref{lem:blowupcodim2} to understand the second and third projections.

For the projection to $\PP^1\times\PP^1$, just observe that the section of $\of(1,1,1)$ gives an extra  linear condition on each fiber. 
\end{fano}

\medskip\begin{fano}\fanoid{3-29-110-2-A} 
\label[mystyle]{29-110-3-1-24-20-1-2-1-3-2-4} $\mZ(\PP^1 \times \PP^2 \times \Gr(2,4),\mQ_{\PP^2}(0,0,1) \oplus \of(1,1,1))$.

\subsubsection*{Invariants}  $h^0(-K)=29, \ (-K)^4=110, \  h^{1,1}=3, \ h^{3,1}=1, \ h^{2,2}=24$, $-\chi(T)=20$.
\subsubsection*{Description} 
\begin{itemize}
    \item $\pi_{12} \colon X \to \PP^1 \times \PP^2 $ is a conic bundle.
    \item $\pi_{13} \colon X \to \PP^1 \times \Gr(2,4)$ is a blow up of a conic in a  hypersurface of bidegree $(1,2)$,
    itself \cref{39-160-2-1-22-21-1-2-2-4} a blow up of $\Gr(2,4)$ along a K3 surface of degree $8$. 
     \item $\pi_{23} \colon X \to \PP^2 \times  \Gr(2,4)$ is a blow up $\Bl_{S_{22}}Y$, where $Y \subset \PP^2 \times \Gr(2,4)$ is the blow up $\Bl_{C_2}\Gr(2,4)$ of a smooth conic $C_2$.
    \item Rationality: yes.
\end{itemize}

\subsubsection*{Semiorthogonal decompositions}
\begin{itemize}
    \item 6 exceptional objects and $\Db(\PP^1\times \PP^2, \mC_0)$ from $\pi_{12}$.
    \item 10 exceptional objects and $\Db(S_8)$ from $\pi_{13}$.
    \item 10 exceptional objects and $\Db(S_{22})$ from $\pi_{23}$. 
    \item We deduce that $S_8$ and $S_{22}$ are derived equivalent (see \cref{lem:dereq}).
\end{itemize}

\subsubsection*{Explanation}
First we use \cref{lem:blowhighercod} to identify $\mZ(\PP^2 \times \Gr(2,4),\mQ_{\PP^2}(0,1) \oplus \of(1,1))$ with the blow up $Y$ of $\Gr(2,4)$ along a conic $C_2$. Then we use  \cref{lem:blowupcodim2} again to cut a (generically of Picard rank 2) K3 surface of degree 22 with respect to $\of(1,1)$ in $Y$. We can project further to $\Gr(2,4)$ to obtain a K3 surface of degree 8 containing the conic $C_2$. 

To understand better the other maps, we consider the equivalent description of this Fano fourfold as $\mZ(\PP^1 \times \Fl(1,4,6),\mQ_2^{\oplus 3} \oplus \of(0;2,0) \oplus \of(1;1,1)).$ We first notice that $\mZ(\Fl(1,4,6),\mQ_2^{\oplus 3} \oplus \of(2,0))$ is a quadric surface bundle over $\PP^2$: by  \cref{lem:conicbundle} the discriminant divisor $\Delta$ is a quartic plane curve. We then need to apply  \cref{lem:blowupcodim2} once again. The codimension 2 blown up locus is  the double cover of $\PP^2$ branched in the said quartic, which is a degree 2 K3 surface.
\end{fano}

\medskip\begin{fano}\fanoid{3-22-74-2} 
\label[mystyle]{22-74-3-1-32-29-1-2-1-2-1-6} $\mZ(\PP^1_1 \times \PP^1_2 \times \PP^5,\of(0,0,2) \oplus \of(0,1,1) \oplus \of(1,0,2))$.

\subsubsection*{Invariants}  $h^0(-K)=22, \ (-K)^4=74, \  h^{1,1}=3, \ h^{3,1}=1, \ h^{2,2}=32$, $-\chi(T)=29$.
\subsubsection*{Description} 
\begin{itemize}
    \item $\pi_{12} \colon X \to \PP^1_1 \times \PP^1_2$ is a $\dP_4$-fibration.
    \item $\pi_{13} \colon X \to \PP^1_1 \times \PP^5$ is a blow up $\Bl_{S_8^{(8)}}Y$, with $S_8^{(8)}$  a degree 8 K3 surface blown up in 8 points, and $Y$ a blow up of $\QQ^4$ along $\PP^1 \times \PP^1$. 
     \item $\pi_{23} \colon X \to \PP^1_2 \times \PP^5$ is a blow up $\Bl_{\Sigma}Z$, where $\Sigma$ is the blow up of $\PP^1 \times \PP^1$ in 8 points, and $Z$ is \cref{39-160-2-1-22-21-1-2-2-4}, the blow up of $\QQ^4$ along a K3 surface of degree 8.
    \item Rationality: yes.
\end{itemize}

\subsubsection*{Semiorthogonal decompositions}
\begin{itemize}
    \item 8 exceptional objects and $\Db(T,\alpha)$, for $W \to \PP^1 \times \PP^1$ the discriminant double cover from $\pi_{12}$.
    \item 18 exceptional objects and $\Db(S_8)$ from $\pi_{13}$ and $\pi_{23}$.
    \item This suggests that the $T$ may be the blow up of a K3 surface in 10 points, and that the Brauer class may be trivial.
\end{itemize}

\subsubsection*{Explanation}
We need to apply  \cref{lem:blowupcodim2} twice.
Start with $Y\subset \PP^1\times \mathbb{Q}^4$ defined by a section of $\of(1,2)$. The projection to $\mathbb{Q}^4$ is the blow up of a surface
$S$ defined as the intersection of two other quadrics, hence a K3 surface of degree 8. 
Then $X\subset \PP^1\times Y$ being defined by a section of $\of(1,1)$
is a blow up of $Y$ along a surface $\Sigma$ which is a codimension two linear section.
This linear section is to be taken on the $\PP^5$ factor, so the projection of 
$\Sigma$ to $\mathbb{Q}^4$ is the blow up of a quadric surface along a linear 
section of $S$ hence $8$ points. 

The projection to the other $\PP^1 \times \PP^5$ is  dealt with by reversing the order of the blow ups.
\end{fano}

\medskip\begin{fano}\fanoid{3-39-160-2} 
\label[mystyle]{39-160-3-1-22-18-1-2-1-2-1-4} $\mZ(\PP^1 \times \PP^1 \times \PP^3,\of(1,1,2))$.

\subsubsection*{Invariants}  $h^0(-K)=39, \ (-K)^4=160, \  h^{1,1}=3, \ h^{3,1}=1, \ h^{2,2}=22$, $-\chi(T)=18$.
\subsubsection*{Description} 
\begin{itemize}
    \item $\pi_{12} \colon X \to \PP^1 \times \PP^1$ is a quadric fibration with discriminant locus a divisor of bidegree $(4,4)$. 
    \item $\pi_{13}, \pi_{23} \colon X \to \PP^1 \times \PP^3$ are blow ups of $\PP^1\times\PP^3$ along a bielliptic K3 surface.

    \item Rationality: yes.
\end{itemize}

\subsubsection*{Semiorthogonal decompositions}
\begin{itemize}
    \item 8 exceptional objects and $\Db(S_8,\alpha)$ from $\pi_{12}$.
    \item 8 exceptional divisors and $\Db(S)$ from $\pi_{13}$ or  $\pi_{23}$.
\end{itemize}

\subsubsection*{Explanation} The projection $\pi_{23}$ to $\PP^1\times\PP^3$ is obviously birational, and its exceptional locus is defined by the vanishing of the induced morphism $\of(-1,-2)\rightarrow V_2^\vee \otimes \of$. We get a surface $S$ which is a general complete 
intersection of two divisors of bidegree $(1,2)$, hence an elliptic K3 surface $S$. If
the equations of the two divisors are $sq_0+tq_1=0$ and $sq_2+tq_3=0$, note that $S$ projects
isomorphically to the quartic surface in $\PP^3$ with equation $q_0q_3=q_1q_2$. 

Note also that by replacing  $\pi_{23}$ with  $\pi_{13}$ we get a birational isomorphism $\varphi$
of $\PP^1\times\PP^3$ with itself. Normalizing the defining equation of $X$ to $s\sigma q_0+t\sigma q_1+s\tau q_2+t\tau q_3=0$, this birational map is defined by 
$$\phi([s,t],x)=([sq_0+tq_1, sq_2+tq_3],x), \qquad \phi^{-1}([\sigma,\tau],x)=([\sigma q_3-\tau q_1, \tau q_0-\sigma q_2],x). $$
Note that the exceptional locus of $\pi_{13}$ is the intersection in $\PP^1\times\PP^3$ of the
two divisors of equations $\sigma q_3-\tau q_1=0$ and $\tau q_0-\sigma q_2=0$, which defines 
another elliptic fibration on the same quartic surface $S$. 

Finally, $\pi_{12}$ is a quadric fibration over $P=\PP^1\times\PP^1$, defined by a map
$\of_P(-1,-1)\ra \Sym^2V_4^\vee\otimes\of_P$. Its discriminant is given by a section of $\of_P(4,4)$. 
\end{fano}

\begin{rmk}
This Fano fourfold already appeared in \cite{eg2}, under the name (B1).
\end{rmk}

\medskip\begin{fano}\fanoid{3-28-104-2} 
\label[mystyle]{28-104-3-1-27-24-1-2-1-3-2-4} $\mZ(\PP^1 \times \PP^2 \times \Gr(2,4),\mQ_{\Gr(2,4)}(0,1,0) \oplus \of(1,0,2))$.
\subsubsection*{Invariants}  $h^0(-K)=28, \ (-K)^4=104, \  h^{1,1}=3, \ h^{3,1}=1, \ h^{2,2}=27$, $-\chi(T)=24$.
\subsubsection*{Description} 
\begin{itemize}
    \item $\pi_{12} \colon X \to \PP^1 \times \PP^2$ is a conic bundle with discriminant 
    of bidegree $(3,4)$. 
    \item $\pi_{13} \colon X \to \PP^1 \times \Gr(2,4)$  is a blow up 
    $\Bl_{\dP_5}Y$, where $Y= \Bl_{S_8}\Gr(2,4)$.
     \item $\pi_{23} \colon X \to \PP^2 \times \Gr(2,4)$  is a blow up 
     $\Bl_{S_8^{(4)}}Z$, where $Z =\Bl_{\PP^2}\Gr(2,4)$ and $S_8^{(4)}$ is the blow up of the octic K3 surface $S_8$ in 4 points.
    \item Rationality: yes.
\end{itemize}

\subsubsection*{Semiorthogonal decompositions}
\begin{itemize}
    \item 13 exceptional objects and $\Db(S_8)$ from the three maps.
\end{itemize}

\subsubsection*{Explanation}
Our Fano fourfold $X$ is defined by two tensors $\alpha\in \Hom(V_3,V_4)$ and $\beta\in V_2^\vee\otimes S^2(\W^2V_4)^\vee$. It parametrizes the set of triples $(x,y,P)$ such that $P\supset\alpha(y)$ and 
$\beta(x,\bullet)=0$ on the Pl\"ucker representative of $P$. 

In order to understand $\pi_{12}$, we observe that for $(x,y)$ given, the set of planes  $P$ containing  $\alpha(y)$ is a projective plane, on which $\beta$ cuts a conic through a morphism $\of_P(-1,-2)\rightarrow S^2(V_4/\of_P(0,-1))^\vee$ on $P=\PP^1\times\PP^2$. We deduce that the
discriminant is given by a map $\of_P(-3,-6)\rightarrow \det ((V_4/\of_P(0,-1)^\vee)^{\otimes 2}$, that is 
a section of $\of_P(3,4)$. 

To describe $\pi_{23}$ we observe that since $\alpha$ is the injection of a hyperplane, 
the condition $P\supset\alpha(y)$ determines uniquely $y$ from $P$ unless $P\subset\alpha(V_3)$. 
So the image of  $\pi_{23}$ is $\Bl_{\PP^2}\Gr(2,4)$. 
Then we apply \cref{lem:blowupcodim2} to this fourfold. The complete intersection of two $\of(0,2)$ divisors on $Z$ cuts a K3 surface that intersects the specified $\PP^2$ in four points. It follows that $S$ is isomorphic to the blow up of an octic K3 surface in these four  points.

Finally consider $\pi_{13}$, whose image is obviously a divisor of bidegree $(1,2)$ in  $\PP^1 \times \Gr(2,4)$, hence the blow up of $\Gr(2,4)$ along the same octic K3 surface as before. For $(x,P)$ given, $y$ is uniquely determined by the condition that $\alpha(y)=P\cap \alpha(V_3)$, except when $P\subset \alpha(V_3)$, in which case we get a projective line as fiber. This implies that $\pi_3$ blows up
a surface which is a $(1,2)$-divisor in  $\PP^1 \times \PP^2$, hence a $\dP_5$. 
\end{fano}

\medskip\begin{fano}\fanoid{3-32-125-2} 
\label[mystyle]{32-125-3-1-24-20-1-3-1-3-1-5}$\mZ(\PP^2_1 \times \PP^2_2 \times \PP^4,\of(1,0,1) \oplus \of(1,1,1) \oplus \mQ_{\PP^2_2}(0,0,1))$.
\subsubsection*{Invariants}  $h^0(-K)=32, \ (-K)^4=125, \  h^{1,1}=3, \ h^{3,1}=1, \ h^{2,2}=24$, $-\chi(T)=20$.
\subsubsection*{Description} 
\begin{itemize}
    \item $\pi_{12} \colon X \to \PP^2 \times \PP^2$ is a blow up $\Bl_{{S}_{16}^{(1)}}(\PP^2 \times \PP^2)$ of a K3 surface of degree 16 blown up in one point.
    \item $\pi_{23} \colon X \to \PP^2 \times \PP^4$ is a blow up 
    $\Bl_{{S}_{12}^{(1)}}Y$ of $Y= \Bl_{\PP^1}\PP^4$ along a K3 surface of degree 12 blown up in one point. 
     \item $\pi_{13} \colon X \to \PP^2 \times \PP^4$ is a blow up of a Hirzebruch surface $\Sigma_1$ in a
     Fano variety belonging to the \cref{39-160-2-1-22-21-1-3-1-5} family, a complete intersection of bidegree $(1,1),(1,2)$.
    \item Rationality: yes. 
\end{itemize}

\subsubsection*{Semiorthogonal decompositions}
\begin{itemize}
    \item 10 exceptional objects and $\Db(S_{16})$ from $\pi_{12}$.
    \item 10 exceptional objects and $\Db(S_{12})$ from $\pi_{13}$.
    \item We deduce that $S_{16}$ and $S_{12}$ are derived equivalent (see \cref{lem:dereq}).
\end{itemize}

\subsubsection*{Explanation}
First consider $\pi_{12}$. 
Writing $\mZ(\PP^2_1 \times \PP^2_2 \times \PP^4,\mQ_{\PP^2_2}(0,0,1))$ as  $\PP^2_1 \times \PP_{\PP^2_2}(\of(0,-1) \oplus \of^{\oplus 2})$,
we see that $X$ is the blow up of $P=\PP^2_1 \times \PP^2_2$ along a smooth surface $S$ whose structure sheaf is resolved by the Eagon--Northcott complex
$$0 \longrightarrow \of_P(-3,-2) \oplus \of_P(-3,-3) \longrightarrow \of_P(-2,-1) \oplus \of_P(-2,-2)^{\oplus 2} \longrightarrow \of_P \longrightarrow \of_S \longrightarrow 0,$$
from which we calculate the Hilbert polynomial $\chi(\of_S(k)) = \frac{15k^2-k}{2}+2$. Dualizing 
the complex, we get 
$$0 \longrightarrow \of_P(-3,-3) \longrightarrow \of_P(-1,-2) \oplus \of_P(-1,-1)^{\oplus 2} \longrightarrow \of_P(0,-1) \oplus\of_P\longrightarrow \omega_S \longrightarrow 0.$$
So once again $\omega_S$ admits a canonical section, vanishing on the locus where the 
morphism $$\of_P(-1,-2) \oplus \of_P(-1,-1)^{\oplus 2} \longrightarrow \of_P(0,-1)$$ is not surjective. 
This is the common zero locus of a section of $\of_P(1,1)$ and two sections of $\of_P(1,0)$, 
hence a line $D$ projecting to a point on $\PP_1^2$. Once again this line is a $(-1)$-curve on 
$S$, after contracting which we get a genuine K3 surface $S_{16}$ of degree $16$. 

\smallskip
Now consider the projection $\pi_{23}: X \to \PP^2_2 \times \PP^4$.
The zero locus of $\mQ_{\PP^2_2} (0,1)$ is $Y=\Bl_{\PP^1}\PP^4$.
We identify $\PP^2_1 \times Y$ with $p:\PP(U_3 \otimes \of_Y) \to Y$ so that $X$ the zero locus of a bundle $p^* \mF \otimes \of_p(1)$, where $\mF= \of(1,0)\oplus\of(1,1)$ is a vector bundle of rank 2 on $Y$ (the restriction of $\of(1,0)\oplus\of(1,1)$ from $\PP^2_2 \times \PP^4$).
So $X$ is the blow up of a smooth surface $T \subset Y$ whose structure sheaf is resolved by the Eagon--Northcott complex
$$0 \longrightarrow \of_Y(-2,-3) \oplus \of_Y(-1,-3) \longrightarrow \of_Y(-1,-2)^{\oplus 3} \longrightarrow \of_Y \longrightarrow \of_T \longrightarrow 0.$$
We deduce that the Hilbert polynomial is $\chi(\of_T(k)) = \frac{11k^2-k}{2}+2$. Moreover, since 
$\of_Y(-2,-3)=K_Y$, the same arguments as before yield that  $T ={S}_{12}^{(1)}$ is the
blow up in one point of a K3 surface of degree 12.

\smallskip
Finally consider $\pi_{13}: X \to \PP^2_1 \times \PP^4$. The image is contained 
in a divisor of bidegree $(1,1)$, over which we get two additional morphisms
$\alpha: \of(-1,-1)\ra V_3^\vee$ and $\beta: \of(0,-1)\ra V_3$, if $\PP^2_2=\PP(V_3)$.
There is a non-trivial fiber when these two morphisms are compatible, which yields
another divisor of bidegree $(1,2)$. We conclude that the image of $\pi_{13}$ is the 
intersection of these two divisors. Moreover, non-trivial fibers occur only when
$\alpha=0$, which can be identified with a complete intersection of a divisor of bidegree $(1,1)$ and three divisors of bidegree $(0,1)$ in $\PP^2 \times \PP^4$, which is
 a surface of type $\Sigma_1$ over which the fibers are projective lines, proving 
that $\pi_{13}$ is a blow up. 
\end{fano}

\medskip\begin{fano}\fanoid{3-25-89-2} 
\label[mystyle]{25-89-3-1-30-27-1-2-1-3-1-5}
$\mZ(\PP^1 \times \PP^2 \times \PP^4,\of(0,1,1)\oplus\of(0,1,2)\oplus \of(1,0,1))$.

\subsubsection*{Invariants}  $h^0(-K)=25, \ (-K)^4=89, \  h^{1,1}=3, \ h^{3,1}=1, \ h^{2,2}=30$, $-\chi(T)=27$.
\subsubsection*{Description} 
\begin{itemize}
    \item $\pi_{12} \colon X \to \PP^1 \times \PP^2$ is a conic bundle.
     \item $\pi_{13} \colon X \to \PP^1 \times \PP^4$ is a blow up $\Bl_{S_8^{(8)}}Y$, where $Y$ is $\Bl_{\PP^2}\PP^4$ and $S_8^{(8)}$ is
      a K3 surface of degree 8 blown up in 8 points.
     \item $\pi_{23} \colon X \to \PP^2 \times \PP^4$ is $\Bl_{\dP_2}Z$
     where $Z$ is the Fano fourfold from \cref{39-160-2-1-22-21-1-3-1-5}.
     \item Rationality: yes
\end{itemize}

\subsubsection*{Semiorthogonal decompositions}
\begin{itemize}
    \item 6 exceptional objects and $\Db(\PP^1\times \PP^2, \mC_0)$ from $\pi_{12}$.
    \item 16 exceptional divisors and $\Db(S_8)$ from $\pi_{13}$ and $\pi_{23}$.
\end{itemize}

\subsubsection*{Explanation}
Let us start with $\pi_{23}$. By \cref{lem:blowupcodim2},
$X$ is the blow up of $Z$ in a surface which is a complete intersection of bidegrees $(1,1),(1,2)$ in $\PP^2 \times \PP^2$, that is a del Pezzo surface of degree 2. In fact, this a double cover of $\PP^2$ ramified in 4 points, a section $\mL^2 \boxtimes \of(1)$ over $\PP_{\PP^2}(\Omega^1(1))$.

The other two maps are obtained by direct calculation.
Consider first the projection to $Y= \mZ(\PP^1 \times \PP^4,\of(1,1))$, i.e.\ $\Bl_{\PP^2}\PP^4$. We have a projective bundle over $Y$ given by $\PP_Y(\of^{\oplus 3})$. The exceptional locus is a surface 
$T$ whose structure sheaf is resolved by the Eagon--Northcott complex
\[
 0 \to \of_Y(0, -1) \oplus \of_Y(0, -2) \to \of_Y^3 \to \of_Y(0,3) \to \of_{T}(0,3) \to 0.
\]
If we restrict the projection $Y \to \PP^4$ to $T$ we can see $T$ as the blow up of the surface $S$ whose structure sheaf is resolved by another Eagon--Northcott complex:
\[
 0 \to \of_{\PP^4}(-5) \oplus \of_{\PP^4}(-4) \to \of_{\PP^4}^3(-3) \to \of_{\PP^4} \to \of_{S} \to 0.
\]
We readily deduce that the Hilbert polynomial of $S$ is $\frac{7k^2-k}{2}+2$, so this is a degree 7 surface in $\PP^4$ which by \cite{Oko84} must be an inner projection of a K3 of degree 8 in $\PP^5$. It follows, that a general plane in $\PP^4$ meets $S$ in $7$ points and that $T\to S$ is the blow up of $7$ additional points. Finally, $T=S_8^{(8)}$ is a K3 surface of degree eight blown up in $8$ points.
\end{fano}

\medskip\begin{fano}\fanoid{3-31-119-2} 
\label[mystyle]{31-119-3-1-24-21-1-3-1-3-1-5} $\mZ(\PP^2_1 \times \PP^2_2 \times \PP^4,\of(0,1,2) \oplus \of(1,1,0) \oplus \mQ_{\PP^2_1}(0,0,1))$.
\subsubsection*{Invariants}  $h^0(-K)=31, \ (-K)^4=119, \  h^{1,1}=3, \ h^{3,1}=1, \ h^{2,2}=24$, $-\chi(T)=21$.
\subsubsection*{Description} 
\begin{itemize}
    \item $\pi_{12} \colon X \to \PP^2_1 \times \PP^2_2$ is a conic bundle over a $(1,1)$ divisor $Z \subset \PP^2 \times \PP^2$ with discriminant a $(2,3)$-divisor in $Z$. 
    \item $\pi_{13} \colon X \to \PP^2_1 \times \PP^4$ is a blow up 
    $\Bl_{{S}_8^{(1)}}Y$, where $Y= \Bl_{\PP^1}\PP^4$ and ${S}_8^{(1)}$ is a K3 surface of degree 8 blown up in one point.
     \item $\pi_{23} \colon X \to \PP^2_2 \times \PP^4$ is a blowup of \cref{39-160-2-1-22-21-1-3-1-5} along 
     a quadratic surface.
    \item Rationality: yes.
\end{itemize}

\subsubsection*{Semiorthogonal decompositions}
\begin{itemize}
    \item 6 exceptional objects and $\Db(Z,\mC_0)$ from $\pi_{12}$.
    \item 10 exceptional objects and $\Db(S_8)$ from $\pi_{13}$.
    \item 10 exceptional objects and $\Db(S_2,\alpha)$ or $\Db(S_8)$ (see \cref{39-160-2-1-22-21-1-3-1-5}) from $\pi_{23}$. 
\end{itemize}

\subsubsection*{Explanation}
Let $\alpha, \beta, \gamma$ denote the three sections defining $X$. In particular $\gamma$ can be seen as a morphism
with a two-dimensional kernel $K_2$. 

The projection $\pi_{12}$ maps $X$ to the $(1,1)$-divisor defined by $\beta$. When $y$ is given in this divisor, 
$z$ must be contained in $\gamma^{-1}(x)\simeq K_2\oplus x$, and $\alpha(y,z,z)=0$ defines a conic in the corresponding
projective plane. So $\pi_{12}$ is a conic bundle, with discriminant given by a divisor of type $(2,3)$. 

\smallskip
Now we describe the projection $\pi_{13}$ to $Y=\Bl_{\PP^1}\PP^4$. We apply \cref{lem:EN} with $\mE= \of^{\oplus 3}$ and $\mF=\of(0,2) \oplus \of(1,0)$, where the bi-grading is with respect to the embedding of $Y$ in $\PP^2 \times \PP^4$ as zero locus of $\mQ(0,1)$. The Eagon--Northcott complex thus reads 
\[
0 \to \of_Y(-1, -4) \oplus \of_Y(-2,-2) \to \of_Y^3(-1, -2) \to \of_Y \to \of_S \to 0.
\]
In order to compute $H^i(\of_Y(a,b))$ we use the Koszul complex for $Y$ in $\PP^2 \times \PP^4$.
We get that 
$$\chi(\of_S(a,b))=2a^2+6ab+\frac{7b^2}{2}-a-\frac{b}{2}+2.$$
In particular $\chi(\of_S(0,b))=\frac{7b^2}{2}-\frac{b}{2}+2$, so that if the projection of $S$ to $\PP^4$ is an isomorphism, we can conclude (by \cite{Oko84}) that $S$ is a K3 surface of degree $8$ blown up at one point.
This is the case because $S$ is defined by the condition that the two linear forms $\alpha(\bullet,z,z)$
and $\beta(x,\bullet)$ are proportional, where moreover $\gamma(z)$ must be contained in $x$. Over the 
blown up line where $z\subset \Ker\gamma$, the linear form  $\alpha(\bullet,z,z)$ is non-zero, so $x$ is uniquely
determined since $\beta$ has maximal rank. As a consequence, $S$ is mapped isomorphically to a surface in $\PP^4$
containing the blown up line, and this surface must be a non-minimal K3 of degree seven. 

\smallskip
Finally, the projection $\pi_{23}$ maps $X$ to the intersection of the $(1,2)$-divisor defined by $\alpha$, with the $(1,1)$ divisor defined by $\beta(\gamma(z),y)=0$. The latter contains $\PP^2\times \PP(K_2)$, which gives the locus
over which $\pi_{23}$ has non-trivial fibers, namely projective lines. This is actually a copy of $\PP^1\times \PP^1$
embedded by $\of(2,1)$, and $\pi_{23}$ is finally the blowup of \cref{39-160-2-1-22-21-1-3-1-5} along a quadratic surface. 
\end{fano}

\medskip\begin{fano}\fanoid{3-33-130-2} 
\label[mystyle]{33-130-3-1-24-20-1-2-1-4-1-6} $\mZ(\PP^1 \times \PP^3 \times \PP^5,\mQ_{\PP^3}(0,0,1) \oplus \of(0,1,1) \oplus \of(1,1,1))$.

\subsubsection*{Invariants}  $h^0(-K)=33, \ (-K)^4=130, \  h^{1,1}=3, \ h^{3,1}=1, \ h^{2,2}=24$, $-\chi(T)=20$.
\subsubsection*{Description} 
\begin{itemize}
    \item $\pi_{12} \colon X \to\PP^1 \times \PP^3$ is a blow up of an elliptic K3 surface blown up in two points. 
    \item $\pi_{13} \colon X \to \PP^1 \times \PP^5$ is a blow up along a quadratic surface of \cref{39-160-2-1-22-21-1-2-2-4}, the intersection of two 
    divisors of bidegree $(0,2)$ and $(1,2)$, which is itself a blow up of a quadric in $\PP^5$ along a degree $8$ 
    K3 surface containing a line. 
     \item $\pi_{23} \colon X \to \PP^3 \times \PP^5$ is a blow up $\Bl_{S_{26}}Y$, where $Y$ is the blow up of $\Gr(2,4)$ in a line, and $S_{26}$ is a K3 of degree 26 and generic Picard rank 2 containing the said line.
    \item Rationality: yes.
\end{itemize}

\subsubsection*{Semiorthogonal decompositions}
\begin{itemize}
    \item 10 exceptional objects and $\Db(S_{16})$ from $\pi_{12}$.
    \item 10 exceptional objects and $\Db(S_8)$ from $\pi_{13}$.
    \item 10 exceptional objects and $\Db(S_{26})$ from $\pi_{23}$.
    \item We deduce that $S_{16}$, $S_8$ and $S_{26}$ are derived equivalent (see \cref{lem:dereq}).
\end{itemize}

\subsubsection*{Explanation} Our fourfold $X$ is defined by tensors $\alpha\in \Hom(V_6,V_4)$, $\beta\in V_4^\vee\otimes V_6^\vee$ and 
 $\gamma\in  V_2^\vee\otimes V_4^\vee\otimes V_6^\vee$; it parametrizes the triples $(x,y,z)$ such that $\alpha(z)\subset y$,
 $\beta(y,z)$ and $\gamma(x,y,z)=0$. 
 
 Let us describe $\pi_{12}$. 
 When $x$ and $y$ are fixed, $z$ must belong to $\alpha^{-1}(y)=A_2\oplus y$, which defines a projective plane to which 
 $\beta$ and $\gamma$ impose two linear conditions. So the projection $\pi_{12}$ is birational, and blows up a surface $S$ inside $P=\PP^1\times\PP^3$
 whose structure sheaf is resolved by the Eagon--Northcott complex
 $$0\longrightarrow \of_P(0,-1)\oplus\of_P(-1,-1)\longrightarrow (A_2 \otimes \of_P)\oplus\of_P(0,1)\longrightarrow \of_P(1,3)
 \longrightarrow \of_S(1,3)\longrightarrow 0. $$
 Dualizing this complex we get 
 $$0\longrightarrow \of_P(-2,-4)\longrightarrow (A_2^\vee \otimes \of_P(-1,-1))\oplus\of_P(-1,-2)
\longrightarrow  \of_P(-1,0)\oplus\of_P \longrightarrow \omega_S\longrightarrow 0. $$
 So $\omega_S$ admits a canonical section whose zero locus is the curve $D$ along which the 
 morphism $(A_2^\vee \otimes \of_P(-1,-1))\oplus\of_P(-1,-2)
\longrightarrow  \of_P(-1,0)$ vanishes, which is the pull-back of two points from $\PP^3$. 
So $D$ is just the union of two lines, and $S$ is an elliptic K3 surface blown up at two points.
The first  complex above also yields the Hilbert polynomial $\chi(\of_S(k,k))=7k^2-t+2$, so that the degree of $S$ with respect to $\of_S(1,1)$ is $16$.

 \smallskip
 Now we turn to $\pi_{13}$. 
 When $x$ and $z$ are fixed, $y=\alpha(z)$ when this is non-zero, which implies that $\pi_{13}$ blows up the 
 locus where $\alpha(z)=0$. The image is the intersection $Y$ of the two divisors of equations $\beta(\alpha(z),z)=0$ 
 and $\gamma(x,\alpha(z),z)=0$. The image of $Y$ is $\PP^5$ is the quadric $Q$ of equation $\beta(\alpha(z),z)=0$,
 which contains the line $l= \PP(\Ker\alpha)$, and the projection of $Y$ to $Q$ blows up the intersection of $Q$ with 
 two other quadrics also containing $l$. 
 
 \smallskip
 Finally we describe $\pi_{23}$. 
We use \cref{lem:blowhighercod} twice: first on $\PP^3 \times \PP^5$ to get $\Bl_{\PP^1} \PP^5$, which is then cut with a quadric containing the line. The blown up locus is a K3 surface containing a line 
(whose projection in $\Gr(2,4)$ is the intersection of three quadrics). This K3 surface 
has Picard rank 2, and its degree with respect to $\of(1,1)$ can be checked to be 26.
\end{fano}

\medskip\begin{fano}\fanoid{3-27-100-2-B} 
\label[mystyle]{27-100-3-1-26-22-1-2-1-4-1-4} $\mZ(\PP^1 \times \PP^3 \times \PP^3,\of(0,1,1)^{\oplus 2} \oplus \of(1,1,1))$.
\subsubsection*{Invariants}  $h^0(-K)=27, \ (-K)^4=100, \  h^{1,1}=3, \ h^{3,1}=1, \ h^{2,2}=26$, $-\chi(T)=22$.
\subsubsection*{Description} 
\begin{itemize}
    \item $\pi_{12},\pi_{13} \colon X \to \PP^1 \times \PP^3$ are 
    blow ups 
    $\Bl_{S_{20}^{(4)}}(\PP^1 \times \PP^3)$, where $S_{20}^{(4)}$ is a K3 surface of 
    degree 20 and Picard number 2, blown up in 4 points.
     \item $\pi_{23} \colon X \to \PP^3 \times \PP^3$ is a blow up $\Bl_{S_4}Y$, where $Y \subset \PP^3 \times \PP^3$ is a codimension 2 linear section, and $S_4$ a determinantal quartic K3 surface.
    \item Rationality: yes.
 \end{itemize}

\subsubsection*{Semiorthogonal decompositions}
\begin{itemize}
    \item 12 exceptional objects and $\Db(S_{20})$ from $\pi_{12}$ or  $\pi_{13}$.
    \item 12 exceptional objects and $\Db(S_4)$ from $\pi_{23}$.
    \item We deduce that $S_{20}$ and $S_4$ are derived equivalent (see \cref{lem:dereq}).
\end{itemize}

\subsubsection*{Explanation}
The projection $\pi_{23}$ can be  understood using \cref{lem:blowupcodim2}. In fact one can check that the fourfold $Y$ has $h^{1,1}=2$, $h^{2,2}=6$, and we can give an explicit interpretation. Indeed, first identify $Y$ with $\mZ(\Fl(1,3,4),\of(1,1))$. This equivalent to writing $Y= \mZ(\PP_{\PP^3}(\mQ^{\vee}(-1)),\mL)$, where we identified $\mU$ on $\Gr(3,4)$ with $\mQ^{\vee}$ on $\PP^3$.
By the Cayley trick, the projection $\pi:Y \to \PP^3$ is therefore generically a $\PP^1$-bundle, with special fibers $\PP^2$ over 4 points. This is enough to compute the Hodge numbers for $Y$, for example via a Grothendieck ring calculation.
This construction shows as well that $X$ is rational, since $Y$ is.
The locus which is blown up is a K3 surface by adjunction. In \cite[Theorem 1]{iliev2019hyperkahler}
this K3 is shown to be quartic determinantal. 

The two other projections $\pi_{12},\pi_{13}$ are birational,
    with base locus inside $P=\PP^1 \times \PP^3$ given by the degeneracy locus of the induced morphism $\of_P(0,-1)^{\oplus 2}\oplus \of_P(-1,-1)\rightarrow V_4^\vee \otimes \of_P$. This is a surface $S$ whose ideal sheaf is resolved by the Eagon--Northcott complex 
    $$0\longrightarrow\of_P(-1,-4)^{\oplus 2}\oplus \of_P(-2,-4)\rightarrow V_4^\vee \otimes \of_P(-1,-3)\longrightarrow\of_P
    \longrightarrow\of_S\longrightarrow 0.$$
    It follows that the Hilbert polynomial of $S$ is $8k^2 - 2k +2$. Dualizing the previous complex we get 
     $$0\longrightarrow \of_P(-2,-4)\rightarrow V_4 \otimes \of_P(-1,-1)\longrightarrow
     \of_P(-1,0)^{\oplus 2}\oplus \of_P
    \longrightarrow\omega_S\longrightarrow 0.$$
    So $\omega_S$ admits a canonical section, vanishing exactly along the curve $D$ where the 
    morphism $V_4 \otimes \of_P(-1,-1)\longrightarrow \of_P(-1,0)^{\oplus 2}$ degenerates. 
    This is just the preimage of the locus in $\PP^3$ where the morphism $V_4 \otimes \of_{\PP^3}(-1)\rightarrow \of_{\PP^3}^{\oplus 2}$, that is, four points; so $D$ is the 
    union of four lines $L_1,\ldots, L_4$. Then $\omega_S=\of_S(L_1+\cdots+L_4)$ and we conclude as usual that $S$ is a degree $20$ K3 surface blown up in 4 points.
    Moreover the class of $S$ in the Chow ring is $[S]=4l h+6h^2$, if $l$ and $h$ denote the hyperplane classes of $\PP^1$ and $\PP^3$, respectively, so that $S$ projects to an elliptic quartic  surface in $\PP^3$. 
\end{fano}

\medskip\begin{fano}\fanoid{3-31-120-2-A} 
\label[mystyle]{31-120-3-1-22-18-1-2-1-2-4} $\mZ(\PP^1 \times \Fl(1,2,4),\of(0;0,1) \oplus \of(1;1,1))$.
\subsubsection*{Invariants}  $h^0(-K)=31, \ (-K)^4=120, \  h^{1,1}=3, \ h^{3,1}=1, \ h^{2,2}=22$, $-\chi(T)=18$.
\subsubsection*{Description} 
\begin{itemize}
    \item $\pi_{12} \colon X \to \PP^1 \times \PP^3$ is a blow up $\Bl_{S_{14}}(\PP^1 \times \PP^3)$, where $S_{14}$ is a K3 surface of Picard rank  2.
    \item $\pi_{13} \colon X \to \PP^1 \times \Gr(2,4)$ is a blow up 
    $\Bl_{S_{16}}(\PP^1 \times \QQ_3)$, where $S_{16}$ is a K3 surface of Picard rank 2.
     \item $\pi_{23} \colon X \to \Fl(1,2,4)$ is a blow up 
     $\Bl_{S_{24}}\PP_{\QQ^3}(\mU|_{\QQ_3})$, where $S_{24}$ is a K3 of Picard rank  2 and $\mU|_{\QQ^3}$ is the restriction of the tautological bundle from $\QQ^4=\Gr(2,4)$.
    \item Rationality: yes.
\end{itemize}

\subsubsection*{Semiorthogonal decompositions}
\begin{itemize}
    \item 8 exceptional objects and $\Db(S_{14})$ from $\pi_{12}$.
    \item 8 exceptional objects and $\Db(S_{16})$ from $\pi_{13}$.
    \item 8 exceptional objects and $\Db(S_{24})$ from $\pi_{23}$.
    \item We deduce that $S_{14}$, $S_{16}$ and $S_{24}$ are derived equivalent (see \cref{lem:dereq}).
\end{itemize}

\subsubsection*{Explanation}
We start with $\pi_{23}$. 
First interpret $\mZ(\Fl(1,2,4),\of(0,1))$ as $ \PP_{\QQ^3}(\mU|_{\QQ^3})$, simply because $\Fl(1,2,4) \cong \PP_{\Gr(2,4)}(\mU)$. Then apply \cref{lem:blowupcodim2}. Two copies of $\of(1,1)$ on the flag manifold cut a K3 which has degree 24 with respect to this bundle, and Picard rank 2.

We turn to $\pi_{13}$.
Consider the projection $\Fl(1,2,4) \to \Gr(2,4)$ that identifies $\Fl(1,2,4)$ with $\PP_{\Gr(2,4)}(\mU)$. Again, we interpret $\mZ(\Fl(1,2,4),\of(0,1))$ as $ \PP_{\QQ^3}(\mU|_{\QQ^3})$, so that we are left with the question of 
understanding the zero locus of $\of(1;1,0)$ inside $\PP^1 \times \PP_{\QQ^3}(\mU|_{\QQ^3}(-1))$. Since $\mZ(1,0)$ inside $\PP_{\QQ^3}(\mU|_{\QQ^3})$ is the blow up of the elliptic curve given by the 
zero locus $\mZ(\QQ^3,\mU^\vee(1))$, it is easy to see that $X$ is the blow up of the zero locus of $\mZ(\PP^1 \times \QQ^3,\mU^\vee(1,1))$, which is a K3 surface $S_{16}$ of Picard rank 2 and degree $16$.

Finally we describe $\pi_{12}$. We consider the projection $\pi: \Fl(1,2,4) \to \PP^3$ that identifies $\Fl(1,2,4)$ with $\PP_{\PP^3}(\mQ^\vee(-1))$. The relative hyperplane section $\of_{rel}(1)$ of this projective bundle is the line bundle $\of(0,1)$ on $\Fl(1,2,4)$, that is, $\of(0;1,0)$ on $\PP^1 \times \Fl(1,2,4)$. Setting $\mL:=\of(1;1,0)$, our variety $X$ is then given by sections of $\of_{rel}(1) \otimes (\of \oplus \mL)$, so the projection is a birational map which blows up  a smooth surface $S$, the degeneracy locus of a map $\of_P(-1,-1)\oplus \of_P \to \mQ_{\PP^3}^\vee(0,1)$ on $P=\PP^1 \times \PP^3$. 
The structure sheaf of $S$ is resolved by the Eagon--Northcott complex
$$0 \longrightarrow \of_P(-1,-3) \oplus \of_P(-2,-4) \longrightarrow \mQ_{\PP^3}^\vee(-1,-2) \longrightarrow \of_P \longrightarrow \of_S \longrightarrow 0.$$
We deduce that the Hilbert polynomial of $S$ is $7k^2 + 2$. 
Dualizing the complex, we get 
$$0 \longrightarrow  \of_P(-2,-4) \longrightarrow \mQ_{\PP^3}(-1,-2) \longrightarrow 
\of_P(-1,-1) \oplus \of_P\longrightarrow \omega_S \longrightarrow 0.$$
In particular, $ \omega_S$ admits a canonical section, vanishing  where the
morphism $ \mQ_{\PP^3}(-1,-2) \longrightarrow \of(-1,-1)$ is not surjective, that is, on the 
pull-back of the zero locus $Z$ of a section of  $ \mQ_{\PP^3}(-1)$ on $\PP^3$. But a general section of this bundle is actually a two-form in four variables, 
and since a general such
form does not vanish on any hyperplane, we can conclude that $Z=\emptyset$ and $\omega_S$ is trivial. 
Finally, $S$ is a degree $14$ K3 surface.
\end{fano}

\medskip\begin{fano}\fanoid{3-27-100-2-C} 
\label[mystyle]{27-100-3-1-24-20-1-2-1-3-1-5} $\mZ(\PP^1 \times \PP^2 \times \PP^4,\of(0,0,2) \oplus \of(0,1,1) \oplus \of(1,1,1))$.

\subsubsection*{Invariants}  $h^0(-K)=27, \ (-K)^4=100, \  h^{1,1}=3, \ h^{3,1}=1, \ h^{2,2}=24, \ -\chi(T)=20$.
\subsubsection*{Description} 
\begin{itemize}
    \item $\pi_{12} \colon X \to \PP^1 \times \PP^2$ is a conic bundle with discriminant a singular $(2,4)$ divisor. 
    \item $\pi_{13} \colon X \to \PP^1 \times \QQ^3$ is a blow up $\Bl_{S_{20}^{(2)}}(\PP^1 \times \QQ^3)$.
     \item $\pi_{23} \colon X \to \PP^2 \times \QQ^3$ is a blow up $\Bl_{S_{20}}Y$, where $Y$ is a $(1,1)$-divisor in $\PP^2\times \QQ^3$.
    \item Rationality: yes.
\end{itemize}

\subsubsection*{Semiorthogonal decompositions}
\begin{itemize}
    \item 6 exceptional objects and $\Db(\PP^1 \times \PP^2, \mC_0)$ from $\pi_{12}$.
    \item 10 exceptional objects and $\Db(S_{20})$ from $\pi_{13}$ and $\pi_{23}$.
\end{itemize}

\subsubsection*{Explanation}
As usual the projection to $\PP^2\times\QQ^3$ is the blow up of the locus in the image cut out by two linear sections: hence the intersection in $\PP^2\times\QQ^3$ of three divisors of bidegree $(1,1)$. This is a K3 surface of degree $20$, the degree of $\of(1,0)$ being two and that of 
 $\of(0,1)$ being six. Note that the image $Y$ of $\pi_{23}$ is a hyperplane section with respect to the Segre-like embedding of $\PP^2 \times \QQ^3 \to \PP^{14}$ given by $\of(1,1)$. 
 
 To construct the claimed semiorthogonal decomposition, it is enough to show that $\Db(Y)$ is generated by 10 exceptional objects. The general fiber of the projection $\pi: Y \to \QQ^3$ is a hyperplane section of $\PP^2$, hence a $\PP^1$. However, there is a line in $\PP^4 \supset \QQ^3$ where the corresponding linear form is degenerate, and hence there are two points $p,q$ of $\QQ^3$ where the fiber is the whole $\PP^2$. Then $\Db(X)$ admits a decomposition by two copies of $\Db(\QQ^3)$, and by one copy of $\Db(p)$ and $\Db(q)$, see \cite[Appendix B]{BFM}.

 When we project to $P=\PP^1\times\QQ^3$, we get a birational map with $\PP^1$-fibers over the degeneracy
 locus of a morphism $\of_P(0,-1)\oplus \of_P(-1,-1)\rightarrow V_3^\vee\otimes\of_P$. This is a smooth 
 surface $S$ whose structure sheaf is resolved by the Eagon--Northcott complex
 $$0\longrightarrow\of_P(-1,-3)\oplus\of_P(-2,-3) \longrightarrow V_3^\vee\otimes\of_P(-1,-2)\longrightarrow\of_P\longrightarrow\of_S\longrightarrow 0.$$
 We deduce that $S$ has degree $18$ with respect to $\of_P(1,1)$. Dualizing the complex we get 
 $$0\longrightarrow\of_P(-2,-3) \longrightarrow V_3\otimes\of_P(-1,-1)\longrightarrow\of_P(-1,0)\oplus\of_P\longrightarrow\omega_S\longrightarrow 0.$$
 So $\omega_S$ has a canonical section whose zero locus is the degeneracy locus $D$ of the morphism 
 $V_3\otimes\of_P(-1,-1)\longrightarrow\of_P(-1,0)$. This locus is the pull-back from $\QQ^3$ of the 
 common zero locus to three sections of $\of_{\QQ^3}(1)$, hence two points. We conclude that $D$
 is the union of two disjoint lines, that $\omega_S=\of_S(D)$, and that $S$ is finally a blow up in two points 
 of a K3 surface of degree $20$. 
 
 Finally, the sections of  $\of(0,1,1)$ and $\of(1,1,1)$ define a $\PP^2$-bundle over $\PP^1\times\PP^2$, 
 corresponding to a family of projective planes in $\PP^4$, and the quadric in $\PP^4$ cuts out a conic
 bundle over $\PP^1\times\PP^2$. The discriminant is a divisor of bidegree $(2,4)$. \end{fano}

\medskip\begin{fano}\fanoid{3-43-180-2} 
\label[mystyle]{43-180-3-1-22-18-1-4-1-5-1-5} $\mZ(\PP^3 \times \PP^4 \times \PP^4,\of(1,1,1) \oplus \mQ_{\PP^3}(0,1,0) \oplus \mQ_{\PP^3}(0,0,1))$.

\subsubsection*{Invariants}  $h^0(-K)=43, \ (-K)^4=180, \  h^{1,1}=3, \ h^{3,1}=1, \ h^{2,2}=22$,$-\chi(T)=18$.
\subsubsection*{Description} 
\begin{itemize}
    \item $\pi_{12}, \pi_{13} \colon X \to \PP^3 \times \PP^4$ are blow ups $\Bl_{S_{22}}Y$, where $S_{22}$ is a K3 surface of degree 22 and $Y$ is a blow up of $\PP^4$ at one point. 
     \item $\pi_{23} \colon X \to \PP^4 \times \PP^4$ has image a rational singular fourfold $Y^\circ$ and contracts a $\PP^2$.
    \item Rationality: yes.
\end{itemize}

\subsubsection*{Semiorthogonal decompositions}
\begin{itemize}
    \item 8 exceptional objects and $\Db(S_{22})$ from $\pi_{12}$ or $\pi_{13}$.
\end{itemize}

\subsubsection*{Explanation} The sections that define $X$ are given by $\alpha \in \Hom(U_5,U_4)$, $\beta\in \Hom(V_5,U_4)$ and $\gamma \in U_4^\vee\otimes U_5^\vee\otimes V_5^\vee$. 
The fourfold $X$ parametrizes the triples of lines $x,y,z$ in $\PP^3\times \PP^4\times \PP^4= \PP(U_4) \times \PP(U_5)\times \PP(V_5)$
such that $$\alpha(y)\subset x, \qquad \beta(z)\subset x, \qquad \gamma(x,y,z)=0.$$
The projections to the two $\PP^4$'s are clearly birational.

The set of pairs $(x,y)$ such that $\alpha(y)\subset x$ is the blow up $Y$ of $\PP^4$ at the point $y_0=\PP(\Ker(\alpha))$. The extra point $z$ is then uniquely determined in the projective line 
$\PP(\beta^{-1}(x))$ by the linear condition given by $\gamma$, except if this condition is 
trivial. This happens on the locus defined by the vanishing of a morphism $\of_Y(-1,-1)\rightarrow (\of_Y\oplus\of_Y(-1,0))^\vee$. 
This locus is a surface $S$ cut out by two divisors of bidegrees $(1,1)$ and $(2,1)$ in $Y$, 
whose canonical divisor is $\of_P(-3,-2)$. So $S$ is a K3 of degree 22 and $X$ is be the blow up 
of $Y$ along this K3. Note that the image of $S$ in $\PP^4$ is the intersection of a quadric through $z_0$, with a cubic which is nodal at $z_0$, in particular this is a nodal surface. The preimage
of $z_0$ in $S$ is a conic, which must be a $(-2)$-curve. 

The set of pairs $(y,z)$ in $\PP^4\times\PP^4$ such that $\alpha(y)$ is collinear to $\beta(z)$, 
and $\gamma (\alpha(y),y,z)=\gamma(\beta(z),y,z)=0$ must be a fourfold $Y^\circ$ singular at $(y_0,z_0)$. 
The extra point $x$ is uniquely determined as $\alpha(y)$ or $\beta(z)$, except when both vanish, 
that is at the unique point $(y_0,z_0)$, over which the condition $\gamma(x,y_0,z_0)=0$ defines a 
projective plane contracted by $\pi_{23}$. In particular the latter is a small contraction. 

The projection of $Y^\circ$ to $\PP V_5 = \PP^4$ is birational, the fiber of $z\ne z_0$ consisting of those $y\subset \alpha^{-1}(\beta(z))$ such that $\gamma(\beta(z),y,z)=0$. This is a point in general, and a projective line when this linear condition 
is trivial on $\alpha^{-1}(\beta(z))\simeq \mathbb{C}\oplus z$. The latter condition is satisfied over the intersection 
of a quadric in $\PP^4$ containing $z_0$ and a cubic which is singular at $z_0$. Finally,
the fiber of $z_0$ is the quadric of equation $\gamma(\alpha(y),y,z_0)=0$. 

To summarize, the birational projection of $Y$ to $\PP^4$ contracts two divisors: a quadric onto a point,
and a (generically) $\PP^1$-bundle to a singular K3. Blowing up $\PP^4$ at $z_0$ we get a smooth K3 surface $S'$ as the 
intersection of two divisors of type $2H-E$ and $3H-2E$. This surface $S'$  cuts the exceptional divisor 
along a conic, and the natural projection to $\PP^3$ exhibits $S'$ as a smooth quartic containing a conic. 
\end{fano}

\medskip\begin{fano}\fanoid{3-23-80-2-B} 
\label[mystyle]{23-80-3-1-30-26-1-2-1-3-1-4}
$\mZ(\PP^1 \times \PP^2 \times \PP^3,\of(0,1,2)\oplus \of(1,1,1))$.

\subsubsection*{Invariants}  $h^0(-K)=23, \ (-K)^4=80, \  h^{1,1}=3, \ h^{3,1}=1, \ h^{2,2}=30, \ -\chi(T)=26$.
\subsubsection*{Description} 
\begin{itemize}
    \item $\pi_{12} \colon X \to \PP^1 \times \PP^2$ is a conic bundle, with discriminant a divisor of bidegree $(2,8)$. 
    \item $\pi_{13} \colon X \to \PP^1 \times \PP^3$ is a blow up of a special quartic K3 surface 
 blown up at eight points. 
    \item $\pi_{23} \colon X \to \PP^2 \times \PP^3$ is $\Bl_{S_{16}} Y$ where $S_{16}$ is a special quartic K3 surface and $Y=\mZ(\PP^2 \times \PP^3,\of(1,2))$ is a rational fourfold.
    \item Rationality: yes.
\end{itemize}

\subsubsection*{Semiorthogonal decompositions}
\begin{itemize}
    \item 6 exceptional objects and $\Db(\PP^1\times \PP^2, \mC_0)$ from $\pi_{12}$.
    \item 16 exceptional objects and $\Db(S_4)$ from $\pi_{13}$.
    \item 8 exceptional divisors, an unknown category and $\Db(S_4)$ from $\pi_{23}$.
\end{itemize}

\subsubsection*{Explanation}
Choose coordinates $[x_1,x_2]$ on $\PP^1$ and $[y_1,y_2,y_3]$ on $\PP^2$. The two defining sections of $X$ can be written 
as $\sum_ky_kq_k$ and $\sum_{i,j}x_iy_j\ell_{ij}$ for some linear forms $\ell_{ij}$ and quadratic forms $q_k$ on $V_4$.

\smallskip
Over $\PP^1 \times \PP^2$, the second section defines a projective bundle $\PP(\mE)$, where $\mE$ is given by the 
sequence $0\ra \of(-1,-1)\ra V_4^\vee \otimes \of \ra \mE^\vee\lra 0$. The first section then gives $X$  as a conic bundle in $\PP(\mE)$ defined 
by the composition $\of(0,-2)\ra S^2V_4^\vee \otimes \of \ra S^2\mE^\vee$. Its discriminant is given by the induced cubic map from 
$\of(0,-6)$ to $\det(\mE^\vee)^2=\of(2,2)$. 

\smallskip
The image of $X$ in $\PP^2\times \PP^3$ is a $(1,2)$-divisor $Y$. Note that the projection of $Y$ to $\PP^3$ 
has projective lines for fibers, except eight projective planes over the eight intersection points $z_1,\ldots , z_8$ 
of the quadrics $q_1,q_2,q_3$. 
As usual $X$ is a blow up of $Y$ along a codimension two linear section, that is a K3 surface $S\subset \PP^2\times \PP^3$ defined by sections of $\of(1,2)$ and $\of(1,1)^{\oplus 2}$. 
The projection of $S$ to $\PP^3$ is an isomorphism
with a quartic K3 of Picard number two, whose equation $Q_4=0$ is obtained as
$$Q_4=\det \begin{pmatrix} \ell_{11}&\ell_{12}&\ell_{13} \\\ell_{21}&\ell_{22}&\ell_{23} \\ q_1&q_2&q_3\end{pmatrix}.$$
In particular this quartic 
surface contains the twisted cubic defined by the condition 
$$ \mathrm{rank}\begin{pmatrix} 
 \ell_{11}&\ell_{12}&\ell_{13} \\\ell_{21}&\ell_{22}&\ell_{23} \end{pmatrix}\le 1.  $$
 Note that quartic surfaces containing a twisted cubic span
 a Noether--Lefschetz divisor in the moduli space.

Observe that $Y$ can be seen as a hyperplane section of $\PP^2 \times \PP^3$ under the embedding $\PP^2 \times \PP^3 \to \PP(V_3 \otimes S^2 V_4)$ given by the line bundle $\of(1,2)$. A Lefschetz decomposition of $\PP^2 \times \PP^3$ with respect to such a line bundle can be given by $3$ copies of $\Db(\PP^3)$. The section $Y$ would inherit then $8$ exceptional objects, and an unknown residual category. Comparing with the decomposition from $\pi_{13}$ we expect the latter to be generated by $8$ exceptional objects.

\smallskip
Finally, the projection to $\PP^1\times \PP^3$ is birational, and $X$ is the blow up of the surface $T\subset 
\PP^1\times \PP^3$ defined by the condition
$$ \mathrm{rank}\begin{pmatrix} q_1&q_2&q_3\\
 x_1\ell_{11}+x_2\ell_{21}&x_1\ell_{12}+x_2\ell_{22}&x_1\ell_{13}+x_2\ell_{23}  \end{pmatrix}\le 1.$$
The projection of $T$ to $\PP^3$ maps it to the quartic $S$, and blows up the eight points $z_1,\ldots , z_8$. 
\end{fano}

\medskip\begin{fano}\fanoid{3-29-110-2-B} 
\label[mystyle]{29-110-3-1-24-20-1-3-1-3-1-4} $\mZ(\PP^2 \times \PP^2 \times \PP^3,\of(0,1,1) \oplus \of(1,0,1) \oplus \of(1,1,1)).$
\subsubsection*{Invariants}  $h^0(-K)=29, \ (-K)^4=110, \  h^{1,1}=3, \ h^{3,1}=1, \ h^{2,2}=24, \ -\chi(T)=20$.
\subsubsection*{Description} 
\begin{itemize}
    \item $\pi_{12} \colon X \to \PP^2 \times \PP^2$ is a blow up $\Bl_{S_{20}^{(1)}}(\PP^2 \times \PP^2)$ along a K3 surface of degree $20$ blown up in one point. 
    \item $\pi_{13},\pi_{23} \colon X \to \PP^2 \times \PP^3$ are blow ups $\Bl_{S_{20}^{(1)}} Y$ of a $(1,1)$ divisor $Y \subset \PP^2 \times \PP^3$ along  a K3 surface of degree 20 blown up in one point.
    \item Rationality: yes.
\end{itemize}

\subsubsection*{Semiorthogonal decompositions}
\begin{itemize}
    \item 10 exceptional objects and $\Db(S_{20})$ from all the maps. 
\end{itemize}

\subsubsection*{Explanation}
The situation being symmetrical, we can describe $\pi_{23}$ and $\pi_{13}$ using the same arguments.
The line bundle $\of(0,1,1)$ cuts out $\PP^2 \times Y$, for $Y$ a $(1,1)$ divisor. Then 
$X$ is the blow up of $Y$ along the degeneracy locus $T$ of the map of bundles $\of_Y(0,-1) 
\oplus \of_Y(-1,-1)\to \of_Y^{\oplus 3}$. The Eagon--Northcott complex on $Y$ reads
$$0 \longrightarrow \of_Y(-1,-3)\oplus\of_Y(-2,-3) \longrightarrow \of_Y(-1,-2)^{\oplus 3} \longrightarrow \of_Y \longrightarrow \of_T\longrightarrow 0$$
and can be dualized, since $\omega_Y=\of_Y(-2,-3)$, into 
$$0 \longrightarrow \of_Y(-2,-3) \longrightarrow \of_Y(-1,-1)^{\oplus 3} 
\longrightarrow \of_Y(-1,0)\oplus\of_Y \longrightarrow \omega_T\longrightarrow 0.$$
So $\omega_T$ admits a canonical section, vanishing exactly on the locus where 
the morphism $ \of_Y(-1,-1)^{\oplus 3} 
\longrightarrow \of_Y(-1,0)$ is not surjective, hence on the common zero locus $D$ of three 
sections of $\of_Y(0,1)$. This curve is a line $D$ (projecting to a point in $\PP^3$), and since 
$\omega_T=\of_T(D)$ we deduce that $D$ is a $(-1)$-curve. Finally, after contracting 
this curve we get a genuine K3 surface of degree $20$. 

\smallskip 
We turn to the description of $\pi_{12}$. On $\PP^2 \times \PP^2$ we have a map of bundles 
$$\of_{\PP^2 \times \PP^2}(0,-1) \oplus \of_{\PP^2 \times \PP^2}(-1,0) \oplus \of_{\PP^2 \times \PP^2}(-1,-1) \longrightarrow \of_{\PP^2 \times \PP^2}^{\oplus 4},$$
and $\pi_{12}$ is the blow up of its degeneracy locus $S$, whose structure sheaf is resolved by the Eagon--Northcott complex
$$0 \longrightarrow \of_{\PP^2 \times \PP^2}(-2,-3) \oplus \of_{\PP^2 \times \PP^2}(-3,-2) \oplus \of_{\PP^2 \times \PP^2}(-3,-3) \longrightarrow \of_{\PP^2 \times \PP^2}(-2,-2)^{\oplus 4} \longrightarrow \of_{\PP^2 \times \PP^2} \longrightarrow \of_{S} \longrightarrow 0.$$
We deduce that the Hilbert polynomial of $S$ with respect to $\of(1,1)$ is $8k ^2 - k +2$.
Moreover, dualizing as before, we get the complex 
$$0 \longrightarrow  \of_{\PP^2 \times \PP^2}(-3,-3) \longrightarrow \of_{\PP^2 \times \PP^2}(-1,-1)^{\oplus 4} \longrightarrow
\of_{\PP^2 \times \PP^2}(-1,0) \oplus \of_{\PP^2 \times \PP^2}(0,-1) \oplus\of_{\PP^2 \times \PP^2} \longrightarrow \omega_{S} \longrightarrow 0.$$
So once again $\omega_{S}$ admits a canonical section, vanishing on the locus where the morphism 
$ \of_{\PP^2 \times \PP^2}(-1,-1)^{\oplus 4} \longrightarrow
\of_{\PP^2 \times \PP^2}(-1,0) \oplus \of_{\PP^2 \times \PP^2}(0,-1)$ is not surjective.
This is a curve $D$ whose class is given, by the Thom--Porteous formula, as 
$$[D]=s_3(\of_{\PP^2 \times \PP^2}(1,0) \oplus \of_{\PP^2 \times \PP^2}(0,1)-\of_{\PP^2 \times \PP^2}^{\oplus 4})=h_1^2h_2+h_1h_2^2,$$
where $h_1$ and $h_2$ denote the hyperplanes classes in the two copies of $\PP^2$ respectively. 
This means that 
the projections of $D$ to these two planes are two lines $D_1$ and $D_2$. Moreover, the Eagon--Northcott 
complex for $D$ reads
$$0 \longrightarrow  \of(-3,-1) \oplus \of(-2,-2) \oplus \of(-1,-3)  \longrightarrow
\of(-2,-1)^{\oplus 4}\oplus 
\of(-1,-2)^{\oplus 4}\longrightarrow
\of(-1,-1)^{\oplus 6}  \longrightarrow \of
\longrightarrow \of_{D} \longrightarrow 0.$$
We readily deduce that the $D$ must be connected, from which we can finally conclude that $D$ is
a projective line mapping isomorphically to both $D_1$ and $D_2$. And once again, $D$ is a $(-1)$-curve
on $S$, after contracting which we get a genuine K3 surface.  
\end{fano}

\medskip\begin{fano}\fanoid{3-36-145-2} 
\label[mystyle]{36-145-3-1-23-19-1-3-1-4-1-5} $\mZ(\PP^2 \times \PP^3 \times \PP^4,\mQ_{\PP^3}(0,0,1) \oplus \of(1,1,0) \oplus \of(1,1,1))$.

\subsubsection*{Invariants}  $h^0(-K)=36, \ (-K)^4=145, \  h^{1,1}=3, \ h^{3,1}=1, \ h^{2,2}=23,  \ -\chi(T)=19$.
\subsubsection*{Description} 
\begin{itemize}
    \item $\pi_{12} \colon X \to \PP^2 \times \PP^3$ is a blow up $\Bl_{S_{16}}Y$, where $Y$ is a $(1,1)$ divisor 
    in $\PP^2 \times \PP^3 $. 
    \item $\pi_{13} \colon X \to \PP^2 \times \PP^4$ is a blow up $\Bl_{\PP^2}Z$, where $Z$ is a special FK3 \cref{39-160-2-1-22-21-1-3-1-5} given by the intersection of two divisors of bidegrees $(1,1)$ and $(1,2)$ in $\PP^2\times\PP^4$.
     \item $\pi_{23} \colon X \to \PP^3 \times \PP^4$ is a blow up $\Bl_{S_{26}^{(1)}}\Bl_p\PP^4$, where $S_{26}^{(1)}$ is a degree 26 K3 surface blown up in one point.
    \item Rationality: yes.
\end{itemize}

\subsubsection*{Semiorthogonal decompositions}
\begin{itemize}
    \item 9 exceptional objects and $\Db(S_{16})$ from $\pi_{12}$.
    \item 9 exceptional objects and $\Db(S_8)$ or $\Db(S_2,\alpha)$ (see \cref{39-160-2-1-22-21-1-3-1-5}) from $\pi_{13}$.
    \item 9 exceptional objects and $\Db(S_{26})$ from $\pi_{23}$.
    \item We deduce that $S_{16}$, $S_8$ and $S_{26}$ are derived equivalent (see \cref{lem:dereq}).
\end{itemize}

\subsubsection*{Explanation} The sections of the three bundles are given by tensors
$\alpha\in \Hom(V_5,V_4)$, $\beta\in V_3^\vee\otimes V_4^\vee$ and  $\gamma\in V_3^\vee\otimes V_4^\vee\otimes V_5^\vee$. The fourfold $X$ parametrizes the triples of lines $(x,y,z)$ 
such that $$\alpha(z)\subset y, \qquad \beta(x,y)=0, \qquad \gamma(x,y,z)=0.$$

\smallskip Let us start with $\pi_{23}$.  
The set of pairs $(y,z)$ such that  $\alpha(z)\subset y$ is the blow up $P:= \Bl_{z_0}\PP^4$ of $\PP^4$ at the point $z_0=\Ker(\alpha)$. For a general $z$, $y$ is uniquely determined, 
and there are two linear conditions on $x$ that also determine it uniquely in general. So 
the projection of $X$ to $\PP^4$ is birational. The exceptional locus $S$ of the projection to
$P$ is the degeneracy locus of the morphism $\of_P(-1,0)\oplus\of_P(-1,-1)\rightarrow V_3^\vee$ induced by $\beta$ and $\gamma$. The structure sheaf of this smooth surface is resolved by the 
Eagon--Northcott complex 
$$0\longrightarrow \of_P(-3,-1)\oplus\of_P(-3,-2)\longrightarrow V_3^\vee \otimes \of_P(-2,-1)\longrightarrow \of_P\longrightarrow \of_S\longrightarrow  0. $$
Since $\of_P(-3,-2)$ is the canonical bundle of $P$, one easily deduces that 
$h^2(\of_S)=1$. Moreover, the Hilbert polynomial for $S$ is $\chi(\of_S(t))=\frac{25k^2-k}{2}+2$.
Dualizing the complex we get 
$$0\longrightarrow \of_P(-3,-2)\longrightarrow V_3 \otimes \of_P(-1,-1)\longrightarrow 
\of_P(0,-1)\oplus\of_P\longrightarrow \omega_S\longrightarrow  0. $$
This implies that $\omega_S$ admits a canonical section, that vanishes on the locus where 
$V_3 \otimes \of_P(-1,-1) \longrightarrow \of_P(0,-1)$ is not surjective. This locus is 
the common zero locus of three general sections of $\of_P(1,0)$, hence the preimage of a point 
in $\PP^3$, which is a line $D$ inside $P$. We conclude that $\omega_S=\of_S(D)$, and by the 
genus formula $D$ is a $(-1)$-curve in $S$; after contracting it, we get a genuine K3 surface 
$S_{min}$ of degree $26$. 
 Moreover $X$ is the blow up of $P$ along the surface $S$. 

\smallskip
We now turn to $\pi_{12}$. 
The set of pairs $(x,y)$ such that $\beta(x,y)=0$ is a divisor $Y$ of bidegree $(1,1)$ in 
$\PP^2\times \PP^3$. The extra point $z$ has to belong to $\alpha^{-1}(y)\simeq z_0\oplus y$, 
which defines a projective line on which $\gamma$ cuts out a unique point. The projection 
of $X$ to $Y$ is 
thus birational. The exceptional locus is the zero locus of the induced morphism 
$\of_Y(-1,-1)\rightarrow (z_0\oplus \of_Y(0,-1))^\vee$, hence a surface
$T$ which is a general complete intersection of divisors of bidegrees $(1,1)$ and $(1,2)$ in $Y$. Being a complete intersection in $\PP^2\times \PP^3$ of three divisors of bidegrees $(1,1)$, 
$(1,1)$ and $(1,2)$, $T=S_{16}$ has trivial canonical bundle and is a K3 surface of degree 16 
(isomorphic to $S_{min}$). 
Moreover $X$ is the blow up of $Y$ along this surface. 

\smallskip Finally we describe $\pi_{13}$. Its image  $Z$ parametrizes 
the set of pairs $(x,z)$ such that $\beta(x,\alpha(z))=0$ and $\gamma(x,\alpha(z),z)=0$.
It is the intersection in $\PP^2\times\PP^4$ of two divisors of bidegrees $(1,1)$ and $(1,2)$, both containing the plane $\Pi=\PP^2\times \{z_0\}$. 
For a given $(x,z)$ in $Z$,  $y$ is uniquely determined by the condition that $\alpha(z)\subset y$, except if $z=z_0$, for which we get a $\PP^1$-fibration
over $\Pi$. We conclude that the projection of $X$ to $Z$
is just the blow up of $\Pi$.  
\end{fano}

\medskip\begin{fano}\fanoid{3-34-133-2} 
\label[mystyle]{34-133-3-1-25-23-1-3-1-4-1-5} $\mZ(\PP^2 \times \PP^3 \times \PP^4,\of(0,1,2) \oplus \mQ_{\PP^2}(0,0,1) \oplus \mQ_{\PP^2}(0,1,0))$.

\subsubsection*{Invariants}  $h^0(-K)=34, \ (-K)^4=133, \  h^{1,1}=3, \ h^{3,1}=1, \ h^{2,2}=25, \ -\chi(T)=23$.
\subsubsection*{Description} 
\begin{itemize}
    \item $\pi_{12} \colon X \to \PP^2 \times \PP^3$ is a conic bundle over $\Bl_p \PP^3$.
    \item $\pi_{13} \colon X \to \PP^2 \times \PP^4$ is a blow up $\Bl_{S_6^{(2)}}Y$, where $S_6^{(2)}$ is a sextic K3 surface in $\PP^4$ blown up at two points, and $Y$ is the blow up $\Bl_{\PP^1}\PP^4$.
     \item $\pi_{23} \colon X \to \PP^3 \times \PP^4$ is a small contraction of two projective planes to a fourfold $Y^\circ$, singular at two points.
    \item Rationality:  yes.
\end{itemize}

\subsubsection*{Semiorthogonal decompositions}
\begin{itemize}
    \item 6 exceptional objects and $\Db(\Bl_p\PP^3,\mC_0)$ from $\pi_{12}$.
    \item 11 exceptional objects and $\Db(S_6)$ from $\pi_{13}$.
\end{itemize}

\subsubsection*{Explanation}
The fourfold $X\subset\PP^2\times\PP^3\times\PP^4$ is the zero locus of sections of $\of(0,1,2)$, $\mQ_{\PP^2}(0,1,0)$ and $\mQ_{\PP^2}(0,0,1)$ defined by tensors $\alpha\in V_4^\vee\otimes S^2V_5^\vee$, $\beta\in \Hom(V_4,V_3)$ and $\gamma\in \Hom(V_5,V_3)$. The last two sections vanish on those triples $(x,y,z)$ of lines such that $\beta(y)\subset x$ and $\gamma(z)\subset x$. 

We start with $\pi_{12}$. For $x,y$ fixed, $z$ must belong to $\gamma^{-1}(x)$
which defines a projective plane, on which $\alpha$
traces a conic. Note that $x$ is defined by $y$ except when $\beta(y)=0$. The kernel of $\beta$ defines a point $y_0\in\PP^3$ and we conclude that $\pi_{12}$ is a conic bundle over $\Bl_{y_0}\PP^3$.
To understand its discriminant locus, notice (in the flag interpretation) that the fiber of $(U_1\subset V_2)$ in $P$ 
is the set of lines $l\in \beta^{-1}(V_2)$ such that $\alpha(l)(U_1)=0$, which defines a conic
in a projective plane. So the projection 
$X\rightarrow P$ is a conic bundle,  defined by a morphism 
$$U_1\rightarrow S^2(V_2\oplus\of_P)^\vee.$$
Its discriminant locus $\Delta$ is defined by a morphism $U_1^3\rightarrow (\det(V_2)^\vee)^2$ so a section of $\of_P(3,2)$. 
Since $K_P=\of_P(-2,-3)$, we get $K_\Delta=\of_\Delta(1,-1)$. 

\smallskip 
Now we turn to $\pi_{13}$. The image $Y$ of $X$ in $\Gr(2,4)\times\PP^4$ is in fact contained in $\PP^2 \times \PP^4$ (since $\mZ(\Fl(1,2,4),\mQ_2)$ base changes the second projection to $\PP^2$) and is the set of pairs $(V_2,l)$ such that $V_2$ contains  $L_1$ and $\beta(l)$. The fibers of the projection $Y\rightarrow\PP^4$ are non-trivial when $\beta(l)\subset L_1$, so we get the blow up of the line $d=\PP(\beta^{-1}(L_1))$. The fibers  of the projection $X\rightarrow Y$ are non-trivial when 
$\alpha(l)$ vanishes on $V_2$. This condition defines a surface $S$ whose image in $\PP^4$ is defined by the conditions that 
$$\alpha(l)(L_1)=\alpha(l)(\beta(l))=0.$$
We get the intersection $S_6$ of a quadric and
a cubic, hence a K3 surface of degree $6$. Note that the line $d$ is bisecant to $S_6$, so $S=S_6^{(2)}$ is the 
blow up of  $S_6$ at these two points.

 Finally we describe $\pi_{23}$. 
One can understand $Y^\circ$ as follows. Given $y,z$ such that $\alpha(y,z,z)=0$ and $\beta(y)$ and $\gamma(z)$ are collinear, $x$ is uniquely defined except if $\beta(y)$ and $\gamma(z)$ both vanish, which happens for exactly two points 
 $q_1,q_2$ in $P=\PP^3\times \PP^4$, over which the fiber of $\pi_{23}$ is a projective plane. The collinearity condition is expressed by the degeneracy of the morphism 
 $$\varphi: \of_P(-1,0)\oplus \of_P(0,-1)\rightarrow U_3 \otimes \of_P$$ defined by $\beta$ and $\gamma$, restricted to the general $(1,2)$-divisor
 cut out by $\alpha$. The fourfold $Y^\circ=D_1(\phi)$ is the image of $\pi_{23}$; its singular locus is  $D_0(\phi)=\{q_1,q_2\}$. We conclude that $\pi_{23}$
 is a small contraction of two projective planes to  a fourfold $Y^\circ$ with two singular points. 
 When we 
 blow up these two points, the exceptional divisors 
 are copies of $\PP^1\times\PP^2$ and the $\PP^1$-factor can be contracted to the smooth fourfold $X$. 
 
 Finally, for a given $x$, $y$ has to vary in a projective line and $z$ in a projective plane, and $\alpha$ defines a divisor of bidegree $(1,2)$ in the corresponding $\PP^1\times\PP^2$. This means that 
 the projection to $\PP^2$ is a $\dP_5$-fibration. The projections to $\PP^3$ and $\PP^4$ have already been described. 
\end{fano}

\medskip\begin{fano}\fanoid{3-37-150-2} 
\label[mystyle]{37-150-3-1-22-18-1-2-1-5-1-6} $\mZ(\PP^1 \times \PP^4 \times \PP^5,\mQ_{\PP^4}(0,0,1) \oplus \of(0,2,0) \oplus \of(1,1,1))$.
\subsubsection*{Invariants}  $h^0(-K)=37, \ (-K)^4=150, \  h^{1,1}=3, \ h^{3,1}=1, \ h^{2,2}=22, \ -\chi(T)=18$.
\subsubsection*{Description} 
\begin{itemize}
    \item $\pi_{12} \colon X \to \PP^1 \times \PP^4$ is a blow up $\Bl_{S_{14}}(\PP^1 \times \QQ^3)$ along  a K3 surface $S_{14}$ of genus eight.
    \item $\pi_{13} \colon X \to \PP^1 \times \PP^5$ is a divisorial contraction to a complete 
    intersection $Y^\circ$ of divisors of bidegrees $(0,2)$ and $(1,2)$, singular along a line. 
     \item $\pi_{23} \colon X \to \PP^4 \times \PP^5$ is $\Bl_{S_{30}}Y$, where $S_{30}$ is a K3 surface of      degree 30 and $Y$ is a $\PP^1$-bundle over $\QQ^3$.
    \item Rationality: yes.
\end{itemize}

\subsubsection*{Semiorthogonal decompositions}
\begin{itemize}
    \item 8 exceptional objects and $\Db(S_{14})$ from $\pi_{12}$.
    \item 8 exceptional objects and $\Db(S_{30})$ from $\pi_{23}$.
    \item We deduce that $S_{14}$ and $S_{30}$ are derived equivalent (see \cref{lem:dereq}).
\end{itemize}

\subsubsection*{Explanation}
There are two line bundles involved, plus the rank four bundle $\mQ_{\PP^4}(0,0,1)$ whose space
of global sections is $\Hom(V_6,V_5)$. Generically,  a morphism $\theta
\in \Hom(V_6,V_5)$ can be reduced to 
the projection from $V_6=K_1\oplus V_5$ to $V_5$, where $K_1$ denotes the kernel. Moreover the zero locus $Z$ of the corresponding 
section is the blow up of $\PP^5$ at the point $v=[K_1]$, which is also $\PP_{\PP^4}(\of\oplus \of(-1))$. This implies that the image of $\pi_{23}$ is just $Y=\PP_Q(\of\oplus \of(-1))$, for $Q\subset\PP^4$
a smooth quadric. Moreover $\pi_{23}$ blows up a smooth surface $S$ in $Y$ obtained as a double section 
of $\of_Y(1,1)$. Since by adjunction, $\omega_Y=\of_Y(-2,-2)$, this must be a K3 surface. 
If we denote by $b,c$ the hyperplane classes from $\PP^4$ and $\PP^5$, the fundamental class
of $Z$ is $c_4(\mQ(0,1))=b^4+b^3c+b^2c^2+bc^3+c^4$. We readily deduce that 
$$\deg(S)=\int_S (b+c)^2=\int_Z 2b(b+c)^4=30, \qquad \int_S b^2 =\int_Z 2b^2(b+c)^2=6.$$
So $S$ is mapped to a special K3 surface of degree $6$ in $\PP^4$.

\smallskip
Now we turn to $\pi_{12}$. Its image is obviously $P=\PP^1\times \QQ^3$, while the fibers are defined
by the image of the morphism $\mQ\oplus \of(-1,-1)\rightarrow V_6^\vee \otimes \of$ (where we simply 
denote by $\mQ$ the restriction of $\mQ_{\PP^4}$ to the quadric 
$Q\subset \PP^4$. This morphism degenerates 
over a smooth surface $T$, whose structure sheaf is resolved by the Eagon--Northcott complex
$$0 \longrightarrow \mQ^\vee(-1,-2) \oplus \of_P(-2,-3) \longrightarrow V_6^\vee\otimes\of_P(-1,-2) 
\longrightarrow \of_P \longrightarrow \of_T \longrightarrow 0,$$
from which we calculate the Hilbert polynomial $\chi(\of_T(k,k)) = 7k^2+2$. Dualizing 
the complex, we get 
$$0 \longrightarrow \of_P(-2,-3) \longrightarrow V_6\otimes \of_P(-1,-1) \longrightarrow  \mQ(-1,-1)\oplus 
\of_P\longrightarrow \omega_T \longrightarrow 0.$$
So $\omega_T$ admits a canonical section whose zero locus is the degeneracy locus of the morphism
$V_6\otimes \of(-1,-1) \longrightarrow \mQ(-1,-1)$, which is just a twist of the surjection $V_6\rightarrow V_5\rightarrow \mQ$. So $T$ is a K3 surface of degree $14$, projecting in $\QQ^3$ to a special K3 
of degree $6$ in $\PP^4$ (its Picard rank is two, generically, and the projection from $T$ to 
 $\PP^1$ is an elliptic fibration). 

\smallskip
Finally, consider the projection $\pi_{13}$ to $R=\PP^1\times\PP^5$. The fiber is defined by a quadratic
form on $V_5$ and two morphisms $\of_R(0,-1)\ra V_5 \otimes \of_R$ and $\of_R(-1,-1)\ra V_5^\vee \otimes \of_R$. When the first 
morphism is non-zero, it defines a point in $\PP(V_5)$ which has to be the unique point of the 
fiber, in case it is compatible with the other sections. The compatibility with the quadratic
form is verified on $\PP^1\times\QQ^4_0$, where $\QQ^4_0$ is a corank one quadric in $\PP^5$;
let $v$ denote its vertex. The 
compatibility with the second morphism holds true over a divisor of type $(1,2)$ containing 
the line $d=\PP^1\times v$. The image of $\pi_{13}$ is therefore a complete intersection of 
hypersurfaces of bidegree $(0,2)$ and $(1,2)$, singular along $d$. 

Non-trivial fibers can only occur when the morphism $\of(0,-1)\ra V_5 \otimes \of$ is zero, that is, 
over $d$. The fibers are then defined by a linear form and a quadratic form on $\PP(V_5)$. 
So $\pi_{13}$ is finally a divisorial contraction, the exceptional divisor being a quadric
fibration over $d$. 
\end{fano}

\medskip\begin{fano}\fanoid{4-35-140-2} 
\label[mystyle]{35-140-4-1-24-17-1-2-1-2-1-3-1-4} $\mZ(\PP^1 \times \PP^1 \times \PP^2 \times \PP^3,\mQ_{\PP^2}(0,0,0,1) \oplus \of(1,1,1,1)).$

\subsubsection*{Invariants}  $h^0(-K)=35, \ (-K)^4=140, \  h^{1,1}=4, \ h^{3,1}=1, \ h^{2,2}=24, \ -\chi(T)=17$.
\subsubsection*{Description} 
\begin{itemize}
    \item $\pi_{123} \colon X \to \PP^1 \times \PP^1 \times \PP^2$ is the blow up of a K3 surface obtained by intersecting divisors of multidegrees $(1,1,1)$ and $(1,1,2)$.
    \item $\pi_{124} \colon X \to \PP^1 \times \PP^1 \times \PP^3$ is a blow up of 
    \cref{39-160-3-1-22-18-1-2-1-2-1-4} along a line.
      \item $\pi_{134}=\pi_{234} \colon X \to \PP^1 \times \PP^2 \times \PP^3$ is a blow up $\Bl_{S_{28}}(\PP^1 \times \Bl_p \PP^3)$, where $S_{28}$ is a degree 28 K3 surface with Picard rank $3$.
    \item Rationality: yes.
\end{itemize}

\subsubsection*{Semiorthogonal decompositions}
\begin{itemize}
    \item 12 exceptional objects and a $\Db(S)$ from each  map.
\end{itemize}

\subsubsection*{Explanation} In order to understand the last two projections, we first apply \cref{lem:blowhighercod}, identifying $\mZ(\PP^2 \times \PP^3,\mQ_{\PP^2}(0,1))$ with $\Bl_p \PP^3$. We then apply in two different ways \cref{lem:blowupcodim2} again. For the second projection, the line is $\PP^1 \times p$, with the point $p$ contained in the K3 surface. For the third and fourth, we only need to check that the intersection of two divisors of degree $\of(1,1,1)$ gives a K3 surfaces $S$ such that $p \in S$. The generic Picard rank of $S$ is 3, with degree 28 with respect to $\of(1,1,1)$.
\end{fano}

\medskip\begin{fano}\fanoid{4-34-134-2} 
\label[mystyle]{34-134-4-1-24-18-1-2-1-3-6} $\mZ(\PP^1 \times \Fl(1,3,6),\mQ_2^{\oplus 2} \oplus \of(0;0,2) \oplus \of(1;2,0))$.
\subsubsection*{Invariants}  $h^0(-K)=34, \ (-K)^4=134, \  h^{1,1}=4, \ h^{3,1}=1, \ h^{2,2}=24, \ -\chi(T)=18$.
\subsubsection*{Description} 
\begin{itemize}
    \item $\pi_{12} \colon X \to \PP^1 \times \PP^5$ is a divisorial contraction, whose image is 
    a complete intersection of bidegrees $(1,2)$ and $(0,2)$, singular along a twisted cubic.
    \item $\pi_{13} \colon X \to \PP^1 \times \Gr(3,6)$ is a conic bundle over $\PP^1 \times \PP^1 \times \PP^1$ with discriminant locus a divisor of multidegree $(3,2,2)$.
     \item $\pi_{23} \colon X \to \Fl(1,3,6)$ is a blow up $\Bl_{{S}_8^{(2)}}Y$, where ${S}_8^{(2)}$ is an octic K3 surface blown up at two points $p,q$ and $Y=\Bl_{\PP^1}CQ$, the latter being a four dimensional quadric cone with $\PP^1$ as a vertex.
    \item Rationality: yes.
\end{itemize}

\subsubsection*{Semiorthogonal decompositions}
\begin{itemize}
    \item 8 exceptional objects and $\Db(\PP^1 \times \PP^1 \times \PP^1,\mC_0)$ from $\pi_{13}$.
    \item 12 exceptional objects and $\Db(S_8)$ from $\pi_{23}$.
\end{itemize}

\subsubsection*{Explanation}
To describe $\pi_{23}$, we apply \cref{lem:blowupcodim2}. We need to understand $Y:=\mZ(\Fl(1,3,6),\mQ_2^{\oplus 2} \oplus \of(0,2))$. The latter by \cite[Corollary 2.7]{DFT} is  a special quadratic section of $\PP_{\PP^3}(\of^{\oplus 2} \oplus \of(-1,-1))$, which is of course the blow up of $\PP^5$ in a line $l \cong \PP(U_2)$. The quadratic section in this blow up is given by an element in $\Sym^2(V_6/U_2)^{\vee}$, hence describing a quadric cone with vertex $l$, of which $Y$ is the blow up.

\smallskip
The above description gives us $\pi_{13}$ as well. We start by identifying $X$ with a $(1,2)$ divisor in the product $\PP^1 \times \PP_{\PP^1 \times \PP^1}(\of^{\oplus 2} \oplus \of(-1,-1))$. In turn this can be rewritten as  $\mZ(\PP_{\PP^1\times\PP^1\times\PP^1}(\of(-1,0,0)^{\oplus 2} \oplus \of(-1,-1, -1)),\mL^2)$.
Over $\PP^1\times T$ where $T=\PP^1\times \PP^1$ we get a conic bundle structure given by a morphism 
$\of(-1,0)\to \Sym^2V_3^{\vee}$.
So the discriminant is a map $\of(-3,0)\to\det(V^{\vee}_3)^2$, and the
discriminant
locus is a divisor of multidegree $(3,2,2)$ in $\PP^1\times\PP^1\times\PP^1$.

\smallskip
Finally consider the projection $\pi_{12}$ to $\PP^1\times\PP^5$. When the line $U_1\subset V_6$ is fixed outside
$V_2$, then $U_3=U_1\oplus V_2$ is uniquely determined. So the section of $\det(U_3^\vee)^{\otimes 2}$ yields a section of $(U_1^\vee)^{\otimes 2}$ defining a quadric whose singular locus is the line $\PP(V_2)$. Cutting 
we the hypersurface of bidegree $(1,2)$, we get a twisted cubic for singular locus. 
\end{fano}

\medskip\begin{fano}\fanoid{4-28-104-2} 
\label[mystyle]{28-104-4-1-28-22-1-2-1-2-1-2-1-4} $\mZ(\PP^1_1 \times \PP^1_2 \times \PP^1_3 \times \PP^3,\of(0,0,1,1) \oplus \of(1,1,0,2)).$
\subsubsection*{Invariants}  $h^0(-K)=28, \ (-K)^4=104, \  h^{1,1}=4, \ h^{3,1}=1, \ h^{2,2}=28, \ -\chi(T)=22$.
\subsubsection*{Description} 
\begin{itemize}
    \item $\pi_{123} \colon X \to \PP^1_1 \times \PP^1_2 \times \PP^1_3$ is a conic bundle over $\PP^1\times\PP^1\times\PP^1$, with discriminant a $(3,3,2)$-divisor. 
    \item $\pi_{124} \colon X \to \PP^1_1 \times \PP^1_2 \times \PP^3$ is a blow up $\Bl_{\dP_4}Y$, where $Y$ is a blow up of $\PP^1\times\PP^3$ along a bielliptic K3 surface $S$.
     \item $\pi_{134},\pi_{234} \colon X \to \PP^1_{1,2} \times \PP^2_3 \times \PP^3$ are blow ups $\Bl_SY$, where $S$ is the same bielliptic K3 surface and $Y$ is $\PP^1 \times \Bl_{\PP^1}\PP^3$.
    \item Rationality: yes.
\end{itemize}

\subsubsection*{Semiorthogonal decompositions}
\begin{itemize}
    \item 8 exceptional objects and $\Db(\PP^1_1 \times \PP^1_2 \times \PP^1_3, \mC_0)$ from $\pi_{123}$.
    \item 16 exceptional objects and $\Db(S)$ from $\pi_{124}$ and from $\pi_{134}$ or  $\pi_{234}$.
\end{itemize}

\subsubsection*{Explanation}
We consider vector spaces $U_2, V_2, W_2, U_4$, and tensors $\alpha\in W_2^\vee\otimes U_4^\vee$ and 
$\beta\in U_2^\vee\otimes V_2^\vee\otimes S^2W_2^\vee$. The fourfold $X$ parametrizes the fourtuples
of lines $(a,b,c,d)$ such that $\alpha(c,d)=0$ and $\beta(a,b,d,d)=0$. 

For a given triple $(a,b,c)$, the condition $\alpha(c,d)=0$ defines a projective plane 
in which $\beta$ cuts a conic, so the projection of $X$ to $\PP^1\times\PP^1\times\PP^1$ is a  
conic bundle defined by a morphism $\of(-1,-1,0)\rightarrow S^2(\of^{\oplus 2}
\oplus \of(0,0,-1))^\vee$. Its discriminant is given by the induced morphism 
$\of(-3,-3,0)\rightarrow S_{222}(\of^{\oplus 2}
\oplus \of(0,0,-1))^\vee= \of(0,0,2)$, hence by a section of $ \of(3,3,2)$,
which has nef canonical divisor. 

The set of triples $(a,b,d)$ such that $\beta(a,b,d,d)=0$ is a divisor $Y$ of multidegree $(1,1,2)$
in $\PP^1\times \PP^1\times \PP^3$, whose projection to each $\PP^1\times \PP^3$ is the blowup 
of a bielliptic K3 surface $S$. The extra line $c$ is then uniquely determined by the 
condition that $\alpha(c,d)=0$, except if $\alpha(W_2,d)=0$. This
condition defines a line $\delta$ 
in $\PP^3$, whose preimage in $Y$ is a divisor of multidegree $(1,1,2)$
in $\PP^1\times \PP^1\times \PP^1$, hence a $\dP_4$. 

The set of pairs $(c,d)$ such that $\alpha(c,d)=0$ is the blow up $P$ of $\PP^3$ along the line
$\delta$. The triple $(a,c,d)$ being given in $Y=\PP^1\times P$, the extra line $b$ is uniquely
determined except if $\beta(a,V_2,d,d)=0$, hence on a surface defined as the 
intersection of two divisors
of multidegree $(1,0,2)$ in $Y$, whose image in $\PP^3$ is again the bielliptic quartic surface $S$. 
\end{fano}

\medskip\begin{fano}\fanoid{4-31-120-2} 
\label[mystyle]{31-120-4-1-24-17-1-2-1-2-1-3-1-3} $\mZ(\PP^1_1 \times \PP^1_2 \times \PP^2_1 \times \PP^2_2,\of(0,0,1,1) \oplus \of(1,1,1,1)).$
\subsubsection*{Invariants}  $h^0(-K)=31, \ (-K)^4=120, \  h^{1,1}=4, \ h^{3,1}=1, \ h^{2,2}=24$, $-\chi(T)=17$.
\subsubsection*{Description} 
\begin{itemize}
    \item $\pi_{123}, \pi_{124} \colon X \to \PP^1 \times \PP^1 \times \PP^2$ are blow ups $\Bl_{S_{20}}(\PP^1\times\PP^1\times \PP^2)$, where $S_{20}$ is a degree 20 K3 surface of 
    Picard rank 3.
    \item $\pi_{134},\pi_{234} \colon X \to \PP^1 \times \PP^2 \times \PP^2$ are blow ups $\Bl_S (\PP^1\times F)$, where $F \subset \PP^2\times \PP^2$ is the complete flag manifold (a Fano threefold of Picard rank 2, index 2 and degree 6)
along a K3 surface $S$ isomorphic to $S_{20}$.
    \item Rationality: yes.
\end{itemize}

\subsubsection*{Semiorthogonal decompositions}
\begin{itemize}
    \item 12 exceptional objects and $\Db(S_{20})$ from $\pi_{123}=\pi_{124}$ and $\pi_{134}=\pi_{234}$.
\end{itemize}

\subsubsection*{Explanation} Here we consider vector spaces $U_2, V_2, U_3, V_3$ and tensors
$\alpha\in U_3^\vee\otimes V_3^\vee$ and $\beta\in U_2^\vee\otimes V_2^\vee\otimes 
U_3^\vee\otimes V_3^\vee$, that define the fourfold $X$ in $\PP^1\times\PP^1\times \PP^2\times\PP^2$. Obviously the projection to $\PP^1\times\PP^1$ is a $\dP_6$-fibration, and the projection to $\PP^2\times\PP^2$ is a conic bundle over a hyperplane section of the latter, which is known to be isomorphic to the complete flag $F=\PP (T_{\PP^2})$. 

The fibers of the projection $\pi_{123}$ of $X$ to  $Y=\PP^1\times\PP^1\times \PP^2$ are non-trivial
when the natural morphism $\of_Y(0,0,-1)\oplus \of_Y(-1,-1,-1)\rightarrow V_3^\vee \otimes \of_Y$
drops rank, which happens along a surface $S$ whose structure sheaf is resolved by the Eagon--Northcott complex 
$$0\longrightarrow \of_Y(-1,-1,-3)\oplus \of_Y(-2,-2,-3) \longrightarrow  V_3^\vee \otimes \of_Y(-1,-1,-2)\longrightarrow \of_Y \longrightarrow  \of_S\longrightarrow 0.  $$
The 
Hilbert polynomial of $S$ is $\chi(\of_S(k))= 10k^2+2$. Dualizing the complex we get 
$$0\longrightarrow \of_Y(-2,-2,-3) \longrightarrow  V_3 \otimes \of_Y(-1,-1,-1)
\longrightarrow \of_Y(-1,-1,0)\oplus\of_Y \longrightarrow  \omega_S\longrightarrow 0.  $$
So $\omega_S$ admits a canonical section, vanishing on the locus where the 
morphism $V_3 \otimes \of_Y(-1,-1,-1)\longrightarrow \of_Y(-1,-1,0)$ is not surjective, 
hence on the common zero locus of three sections of $\of_Y(0,0,1)$. Clearly this locus 
is empty, so $\omega_S$ is trivial and we can conclude that 
$S$ must be a K3 surface of degree 20. Moreover $X$ must be the blow up of $Y$ along $S$. Note that the class of $S$ in the Chow ring of 
$Y$ can be obtained as the degree two part of $ch(\of_S)$. If we denote by $a,b,c$ the hyperplane classes of the three projective spaces, the previous complex yields 
$[S]= 2ab+3ac+3bc+3c^2$. We deduce that the Picard group of $S$ contains a copy of the lattice
$$\begin{pmatrix} 0&3&3 \\ 3&0&3 \\ 3&3&2
\end{pmatrix}. $$

The projection $\pi_{234}$ to  $\PP^1\times\PP^2\times \PP^2$ maps $X$ to $\PP^1\times F$. 
The fibers are non-trivial over the intersection of two divisors of multidegree $(1,1,1)$,
which is again a K3 surface $T$. Here the  Picard group of $T$ contains a copy of the lattice
$$\begin{pmatrix} 0&3&3 \\ 3&2&4 \\ 3&4&2
\end{pmatrix}. $$
\end{fano}

\medskip\begin{fano}\fanoid{5-31-120-2} 
\label[mystyle]{31-120-5-1-26-16-1-2-1-2-1-2-1-2-1-2} $\mZ((\PP^1)^5,\of(1,1,1,1,1))$.
\subsubsection*{Invariants}  $h^0(-K)=31, \ (-K)^4=120, \  h^{1,1}=5, \ h^{3,1}=1, \ h^{2,2}=26, \ -\chi(T)=16$.
\subsubsection*{Description} 
\begin{itemize}
    \item $X \to (\PP^1)^4$ is a blow up $\Bl_{S_{24}}(\PP^1)^4$ in five different ways, where $S_{24}$ is a K3 surface of degree 24 with  Picard rank $4$. 
    \item Rationality: yes.
\end{itemize}

\subsubsection*{Semiorthogonal decompositions}
\begin{itemize}
    \item 16 exceptional objects and $\Db(S_{24})$ from each map.
\end{itemize}

\subsubsection*{Explanation} The projection to each partial product $\PP^1\times \PP^1\times \PP^1\times \PP^1$ is a blow up of a K3 surface defined as the intersection of two 
divisors of multidegree $(1,1,1,1)$. 
This surface has degree $24$ in the Segre embedding.\end{fano}

\begin{rmk}
This fourfold was already considered in \cite{kuchle} and labeled as (d3), and its birational structure was described in \cite{kuznetsovKuchlePicard}.
\end{rmk}

\section{Rogue Fano fourfolds}\label{sect:rogue}
\renewcommand\thefano{R--\arabic{fano}}

This section contains the remaining four families of Fano fourfolds of K3 type that we obtained, which do not belong to the lists of the three previous sections. The origin of their K3 structures is therefore a bit more mysterious. 

\medskip\begin{fano}\fanoid{2-27-99-2} 
\label[mystyle]{27-99-2-1-25-25-1-3-2-6} $\mZ(\PP^2 \times \Gr(2,6),\mU^{\vee}_{\Gr(2,6)}(0,1) \oplus \mQ_{\PP^2} \boxtimes \mU^{\vee}_{\Gr(2,6)})$.

\subsubsection*{Invariants}  $h^0(-K)=27, \ (-K)^4=99, \  h^{1,1}=2,  \ h^{3,1}=1, \ h^{2,2}=25, \ -\chi(T)=25$.
\subsubsection*{Description} 
\begin{itemize}
    \item $\pi_1 \colon X \to \PP^2$ is a $\dP_5$ fibration.
    \item $\pi_2 \colon X \to \Gr(2,6)$ is birational to a singular fourfold, contracting three projective planes.
    \item Rationality: yes.
   
\end{itemize}

\subsubsection*{Semiorthogonal decompositions}
\begin{itemize}
    \item 6 exceptional objects and $\Db(Z)$, for $Z \to \PP^2$ a degree $5$ flat map, from $\pi_1$. One can expect $Z$ to be the blow up of a K3 in 3 points.
\end{itemize}

\subsubsection*{Explanation} The fourfold $X$ is defined by two sections $\alpha\in \Hom(V_6,V_3)$ and $\beta\in S_{2,1}V_6^\vee\subset V_6^\vee\otimes \wedge^2V_6^\vee$. It parametrizes the pairs $(l,P)\in \PP^2\times \Gr(2,6)$ such that $\alpha(P)\subset l$ and $\beta(p,p_1\wedge p_2)=0$ for all $p\in P=\langle p_1, p_2\rangle$.  
The first condition implies that $P$ meets the kernel 
$K_3$ of $\alpha$. (Hence a projection to $\PP(K_3)$ which will induce a rational quadric fibration on $X$.)

Projecting to $\PP^2$, the fiber over $l$ is the set of planes in $\Gr(2,\alpha^{-1}(l))$ where a section of $\mQ(1)$ vanishes, that is a $\dP_5$. Notice that this description implies the rationality of this family,  since we already
mentioned that every del Pezzo of degree $5$ is rational over its field of definition, hence $X$ is $k(\PP^2)$-rational.

Finally, projecting to $\Gr(2,6)$, we get non-trivial fibers when $\alpha(P)=0$. Since a section of $\mQ(1)$ on a projective plane vanishes at three points, we conclude that this projection is a birational morphism contracting three projective planes -- hence a small contraction. 
\end{fano}

\medskip\begin{fano}\fanoid{2-33-129-2} 
\label[mystyle]{33-129-2-1-23-23-1-4-2-7} $\mZ(\PP^3 \times \Gr(2,7),\of(0,1)^ {\oplus 2} \oplus \of(1,1) \oplus \mQ_{\PP^3}\boxtimes \mU^{\vee}_{\Gr(2,7)})$.

\subsubsection*{Invariants}  $h^0(-K)=33, \ (-K)^4=129, \  h^{1,1}=2, \ h^{3,1}=,1 \ h^{2,2}=23, \  -\chi(T)=23$.
\subsubsection*{Description} 
\begin{itemize}
    \item $\pi_1 \colon X \to \PP^3$ is a conic bundle with discriminant a quintic surface, singular in 16 points.
    \item $\pi_2 \colon X \to \Gr(2,7)$ contracts a $\PP^2$ to a point of a singular fourfold  $Y^\circ$.
    \item rationality: unknown. However, see \cref{rmkKuz} for a connection with a (possibly special) family of \cref{33-129-2-1-23-23-1-4-2-7}.
\end{itemize}

\subsubsection*{Semiorthogonal decompositions}
\begin{itemize}
    \item 4 exceptional objects and $\Db(\PP^3,\mC_0)$ from $\pi_1$.
    \item There is no obvious hint about the K3 structure (but compare with \cref{32-124-2-1-28-28-1-2-8} or \cref{31-120-3-1-2-22-20-1-6-1-8} and the possible correspondence between singular intersections of three quadrics and singular $(2,2)$ divisors in $\PP^2 \times \PP^3$; see also \cref{rmk11,rmk12}).
\end{itemize}

\subsubsection*{Explanation}
The fourfold $X$ is defined by three line bundles and a rank six bundle.
First consider the sevenfold $Z$ defined by a section of the latter, which is just a morphism
$\theta\in V_4\otimes V_7^\vee=\Hom(V_7,V_4)$. Generically, one can suppose that $V_7=K_3\oplus 
V_4$ and that $\theta$ is the projection to $V_4$. Our $Z$ is just the set of pairs $(l, U_2)$
such that $\theta(U_2)\subset l$. In particular $U_2$ has to meet $K_3$ non-trivially. This 
suggests to define 
$$\tilde{Z}=\{U_1\subset U_2\subset U_4\subset V_7, \quad  U_1\subset K_3\subset U_4\},$$
a $\PP^2$-bundle over $\PP(K_3)\times \PP(V_7/K_3)=\PP^2\times \PP^3$. 
Sending the flag $U_1\subset U_2\subset U_4$ to
$(U_4/K_3, U_2)$, one obtains a birational isomorphism $\eta$ onto $Z$, since $U_1$ can be 
recovered as $U_1=U_2\cap K_3$ when $U_2$ is not contained in $K_3$. The exceptional locus
is $\Fl(1,2,K_3)\times \PP^3$, which is contracted to the fivefold $\Gr(2,K_3)\times\PP^3\subset \Gr(2,V_7)\times \PP^3$ by forgetting $U_1$. 

Now we take into account the sections of the line bundles, which are defined by two 
skew-symmetric forms $\omega_1$, $\omega_2$ on $V_7$, and a tensor $\alpha\in V_4^\vee\otimes\wedge^2V_7^\vee$. They define a fourfold $X$ in $Z$, whose preimage
in $\tilde{Z}$ is the subvariety $\tilde{X}$ defined by the conditions 
$$\omega_1(U_1,U_2)=\omega_2(U_1,U_2)=\alpha(U_4/K_3, U_1,U_2)=0.$$
Note that $U_4$ is not completely determined by $U_2$ exactly when $U_2\subset K_3$. Under this condition, 
$U_2$ is uniquely determined by $\omega_1$ and $\omega_2$, and $\alpha$ cuts out a projective plane
in $\PP^3$. We conclude that $\pi_2$ is a birational morphism that contracts a plane $\Pi$ to a 
singular point $v$ of the image. 

Let us give a description of this image. 
Over $\PP(K_3)\times \PP(V_7/K_3)$, the conditions above cut out a linear space whose 
dimension is governed by the rank of the induced map 
$$U_1^{\oplus 2}\oplus U_1\otimes U_4/K_3\longrightarrow (U_4/U_1)^\vee.$$
The projection from $\tilde{X}$ to $\PP(K_3)\times \PP(V_7/K_3)$ has non-trivial 
fibers when this morphism degenerates, and this condition defines a determinantal hypersurface $H$
of bidegree $(2,2)$, birationally isomorphic to $X$. Moreover there exist 
positive dimensional fibers, parametrized by the curve where the morphism degenerates further, 
and the morphism from $\tilde{X}$  to $H$ is a small resolution of singularities. Note also 
that the projection $\tilde{X}\rightarrow X$ has positive dimensional fibers over the locus of pairs $(U_2,U_4)$ over which the conditions on $U_1$ are empty, which exactly means that $U_2\subset K_3$;
so $\tilde{X}$ is the blow up of $X$ along the projective plane $\Pi$. 

\smallskip Let us now turn to $\pi_1$. When $U_4$ is fixed, $U_2\subset U_4$ has to verify
three linear conditions, which defines a conic bundle over $\PP^3$. More precisely,  $\omega_1$,  $\omega_2$ and $\alpha$ induce over $\PP^3$ a morphism 
$$\of_{\PP^3}^{\oplus 2}\oplus \of_{\PP^3}(-1) \longrightarrow \W^2(K_3^\vee\oplus \of_{\PP^3}(-1))^\vee,$$
which degenerates in codimension $6-3+1=4$, hence nowhere. The orthogonal to the image is then 
a rank three vector bundle $\mE\subset  \W^2(K_3^\vee\oplus \of_{\PP^3}(-1))$, with a natural map
$S^2\mE\longrightarrow \mK=\of_{\PP^3}(-1)$ that defines our conic bundle inside $\PP(\mE)$.

\begin{lemma}
The discriminant $\Delta$ is a quintic surface with $16$ double points.
\end{lemma} 

\begin{proof}
The discriminant is given by the induced map $(\det\mE)^2\rightarrow \of_{\PP^3}(-3)$. Since 
$\det\mE=\of_{\PP^3}(-4)$, we get a quintic surface in $\PP^3$, with a number of double points given by \cref{lem:conicbundle}. 
In order to get it, we simply compute 
$$c(\mE)=\frac{(1+h)^3}{1-h}=1+4h+7h^2+8h^3.  $$
Since $k=c_1(\mK)=-1$, we finally get $[\Delta_{sing}]=16h^3$.
\let\oldqedbox\qedsymbol
\newcommand{\twoqedbox}{$\square$ \, \oldqedbox}
\renewcommand{\qedsymbol}{\twoqedbox}
\end{proof}
\let\qed\relax
\end{fano}

We recall that the link between \cref{33-129-2-1-23-23-1-4-2-7} and \cref{40-163-2-1-23-24-1-4-1-6}, see \cref{rmkKuz}, is investigated in the recent This is investigated in the recent \cite{nodalK3}.

\begin{rmk}
\label{rmk11}
We point out the probable connection with the recent work by Huybrechts \cite{huybrechts2021nodal}, where a link between 16-nodal quintic surfaces and cubic fourfolds is made explicit. In particular, it is explained how a double cover of a 16-nodal quintic is a surface of general type $F$, whose anti-invariant part $H^2(F, \Z)^{-}$ of the cohomology is a K3 structure of rank 23, that in turn could explain the K3 structure in our Fano  fourfold.\end{rmk}

\begin{rmk}
\label{rmk12}
At the level of rationality, there is an interesting link between this Fano variety and a recent work of Hassett, Pirutka and Tschinkel \cite{HPT-22}. In fact, they showed how a specific family of smooth complete intersections of bidegree $(2,2)$ in $\PP^2 \times \PP^3$ has the very general fiber which is not stably rational, whereas some special fibers are rational. In our case, we have seen how our Fano variety is birational to a \emph{singular} complete intersection of the same type. It would be very interesting to have an analogy with \cite{HPT-22}, and produce some special subfamilies of this family of Fano,
for which not all the fibers are rational.  This  would be a counterexample to the generalization of the Kuznetsov conjecture to all FK3 fourfolds. However, we have not been able to understand this yet.
\end{rmk}

\medskip\begin{fano}\fanoid{3-31-120-2-B} 
\label[mystyle]{31-120-3-1-2-22-20-1-6-1-8} $\mZ(\PP^5 \times \PP^7,\mQ_{\PP^5}(0,1) \oplus \of(0,2) \oplus \of(2,0)^{\oplus 2}).$

\subsubsection*{Invariants}  $h^0(-K)=31, \ (-K)^4=120, \  h^{1,1}=3, \ h^{3,1}=1, \ h^{2,1}=2, \ h^{2,2}=22, \ -\chi(T)=20$.
\subsubsection*{Description} 
\begin{itemize}
    \item $\pi_1 \colon X \to \PP^5$ is a conic bundle over $Z$, the complete intersection of two quadrics in $\PP^5$, with discriminant locus $\Delta=S_8$ a smooth K3 surface of genus 5.  
    \item $\pi_2 \colon X \to \PP^7$ is the blow up $\Bl_{p,q} Y^\circ$, where $Y^\circ \subset \PP^7$ is the singular complete intersection of three quadrics, $\Q_1, \Q_2^{\circ}, \Q_3^{\circ}$, where $\Q_1$ is of maximal rank, and $\Q_2^{\circ}$ and $\Q_3^{\circ}$ are double cones over quadrics in $\PP^5$. In particular, $\Sing(\Q_2^{\circ})=\Sing(\Q_3^{\circ}) \cong \PP^1$, and $\Sing(Y^\circ)= \Q_1 \cap \PP^1$ is made of the two points $p$ and $q$.
    \item Rationality: yes.
\end{itemize}

\subsubsection*{Semiorthogonal decompositions}
\begin{itemize}
    \item 4 exceptional objects, 2 genus 2 curves and $\Db(S_8)$ from $\pi_1$, see \cref{subsub-conic}.
\end{itemize}

\subsubsection*{Explanation}

Here we have a quadric $\QQ^6$ in $\PP^7$, the intersection $Z$ of two quadrics in $\PP^5$, 
and $X$ is defined in $Z\times \QQ^6$ by the vanishing of a section of $\mQ_5(0,1)$. This section 
is defined by a morphism $\theta\in \Hom(V_8,V_6)$, with a two-dimensional kernel $K_2$, 
and it vanishes at $(l,m)$ if 
$m\subset \theta^{-1}(l)$. Note that if $K_2$ is the kernel of $\theta$ that we identify
with the projection from $V_8=K_2\oplus V_6$ to $V_6$, then $ \theta^{-1}(l)=K_2\oplus l$.
We get the conic bundle structure on $X$ by restricting $\QQ^6$ to the planes $\PP(K_2\oplus l)$,
for $l \in Z$.  The discriminant 
locus $\Delta$ is given by a section of $\det ( \theta^{-1}(l))^{-2}\simeq l^{-2}$, which means 
that $\Delta$ is a quadric section of $Z$, that is, a K3 surface of genus $5$. Note that contrary to 
the general situation the discriminant is smooth: if the restriction of $\mQ$ to $K_2$ is non-degenerate, the rank does never drop to one on $K_2\oplus l$.
The above description lets us construct a section (actually, two) of the conic bundle $\pi_1$ : just note that $\PP(K_2)$ is a trivial $\PP^1$-bundle over $Z$, and we are taking a quadratic section.
This gives the required semiorthogonal decomposition.
This variety appears also in  \cite{secci}, where it is labelled as $X_4^{1,2}$. The given description realizes $X$ as a conic bundle over $Z$ factoring through the blow up of $\PP_Z(\of_Z \oplus \of_Z(-1))$ with $\Delta$ as center.

\smallskip
What is the image of $X$ in $\QQ^6$? It consists of the lines $m$ such that $\theta(m)$, when non-zero,
is contained in $Z$. This condition defines in $\PP^7$ an intersection $Z'$ of two quadrics, singular
along the line $\PP(K_2)$. The intersection $Z'\cap \QQ^6$ has two singular points, and 
the projection from $X$ to $Z'\cap \QQ^6$ is a birational morphism that contracts exactly two 
copies of $Z$ to these two singular points.
\end{fano}

\medskip\begin{fano}\fanoid{3-25-90-2} 
\label[mystyle]{25-90-3-5-1-22-23-1-5-1-7} $\mZ(\PP^4 \times \PP^6,\mQ_{\PP^4}(0,1) \oplus \of(3,0) \oplus \of(0,2))$.
\subsubsection*{Invariants}  $h^0(-K)=25, \ (-K)^4=90, \  h^{1,1}=3, \ h^{2,1}=5, \ h^{3,1}=1, \ h^{2,2}=22, \ -\chi(T)=23$.
\subsubsection*{Description} 
\begin{itemize}
    \item $\pi_1 \colon X \to \PP^4$ is a conic bundle over $W_3$, with discriminant a   K3 surface 
    of degree $6$.
    \item $\pi_2 \colon X \to \PP^6$ contracts two copies of $W_3$ to the two singular points of $Y^\circ=\mathcal{CC}(W_3) \cap \QQ^5$, with $\mathcal{CC}(W_3)$ the double projective cone over a cubic threefold $W_3$.
    \item Rationality: unknown.
\end{itemize}

\subsubsection*{Semiorthogonal decompositions}
\begin{itemize}
    \item 4 exceptional objects, 2 $[5/3]$-CY categories and $\Db(S_6)$ from $\pi_1$.
\end{itemize}

\subsubsection*{Explanation}
Our fourfold $X$ parametrizes the pairs $(x,y)$, for $x\in W_3\subset\PP^4$ and $y\in\QQ^5\subset\PP^6$,
such that $\alpha(y)\subset x$ for some given morphism $\alpha\in \Hom(V_7,V_5)$.
Denote by $K_2$ the kernel of $\alpha$. 

We first describe $\pi_1$. For $x$ given, $y$ must be contained in $\alpha^{-1}(x)\simeq K_2\oplus y$, 
which gives a projective plane on which $\QQ^5$ cuts a conic. So $\pi_1$ is a conic bundle over
$W_3$, whose discriminant is defined by a section of $\of(2)$, which gives a K3 surface of degree 6. 
(Note that the discriminant is smooth, the quadratic form being non-degenerate on $K_2$). Moreover, $\PP (K_2)$ is a trivial $\PP^1$-bundle over $W_3$, hence the fibration has two sections given by the intersection of the fibers with this constant line. The above description of the semiorthogonal decomposition of $\Db(X)$ follows (see \cref{subsub-conic}).

Now we turn to $\pi_2$. Observe that the cubic condition on $x$ implies that the image of $\pi_2$ is 
$\mathcal{CC}(W_3)\cap \QQ^5$, which is singular at the two points of $\PP(K_2)\cap \QQ^5$. For $y$ given, 
$x$ is uniquely defined by the condition $\alpha(y)\subset x$,
except when  $\alpha(y)=0$, in which case the fiber is the whole 
$W_3$. 

This variety appears also in  \cite{secci}, where it is labelled as $X_3^{1,2}$. The given description realizes $X$ as a conic bundle over $W_3$ factoring through the blow up of $\PP_{W_3}(\of_Z \oplus \of_Z(-1))$ with $\Delta$, the K3 surface of degree 6 above, as center. 
\end{fano}

\appendix \section{Some blow ups of FK3}\label{sect:blowups}

\renewcommand\thefano{A--\arabic{fano}}
In this appendix we collect a few families of  Fano fourfolds which are not FK3 themselves, but whose Hodge structure in middle dimension contains a proper K3 substructure. A typical example is the blow up of a cubic fourfold in a K3 surface.

\medskip\begin{fano}\fanoid{2-12-27-2} 
\label[mystyle]{12-27-2-7-79-58-1-2-1-6} $\mZ(\PP^1 \times \PP^5,\of(0,3) \oplus \of(1,2))$. 
\subsubsection*{Invariants}  $h^0(-K)=12, \ (-K)^4=27, \  h^{1,1}=2, \ h^{3,1}=7, \ h^{2,2}=79, \ -\chi(T)=58$.
\subsubsection*{Description} 
\begin{itemize}
    \item $\pi_1 \colon X \to \PP^1$ is a fibration in $W_6$, the latter being a $(2,3)$ complete intersection in $\PP^5$. 
    \item $\pi_2 \colon X \to \PP^5$ is the blow up $\Bl_{S_{12}} X_3$ of a cubic fourfold, where $S_{12}$ is a degree 12 canonical surface given by a $(2,2,3)$ complete intersection.
    \item Rationality: unknown. Same as the general cubic fourfold.
\end{itemize}

\subsubsection*{Semiorthogonal decompositions}
\begin{itemize}
    \item 2 exceptional objects and an unknown category from $\pi_1$.
    \item 3 exceptional objects, the K3 category of $X_3$ ans $\Db(S_{12})$ from $\pi_2$.
\end{itemize}

\subsubsection*{Explanation} It suffices to apply \cref{lem:blowupcodim2}.
\end{fano}

\medskip\begin{fano}\fanoid{2-16-45-2} 
\label[mystyle]{16-45-2-3-51-44-1-3-1-6} $\mZ(\PP^2 \times \PP^5,\of(0,3) \oplus \of(1,1) \oplus \of(1,1)).$

\subsubsection*{Invariants}  $h^0(-K)=16, \ (-K)^4=45, \  h^{1,1}=2, \ h^{3,1}=3, \ h^{2,2}=51, \ -\chi(T)=44$.
\subsubsection*{Description} 
\begin{itemize}
    \item $\pi_1 \colon X \to \PP^2$ is a fibration in cubic surfaces.
    \item $\pi_2 \colon X \to \PP^5$ is a blow up $\Bl_{T_9}X_3$ of a cubic fourfold along an elliptic surface of degree $9$ and Kodaira dimension one. 
    \item Rationality: unknown, the cubic fourfold being general. 
\end{itemize}

\subsubsection*{Semiorthogonal decompositions}
\begin{itemize}
    \item 3 exceptional objects and an unknown category from $\pi_1$.
    \item 3 exceptional objects, the K3 category of $X_3$ and $\Db(T_9)$ from $\pi_2$.
\end{itemize}

\subsubsection*{Explanation} We have two sections of $\of(1,1)$ over $\PP^2\times X_3$,
where $X_3$ is a cubic fourfold. The projection $\pi: X\rightarrow X_3$ has fibers given by two linear conditions, hence points in general. So $\pi$ is birational and its
base locus is the degeneracy locus of the induced morphism $\of_{X_3}(-1)^2\rightarrow V_3^\vee\otimes \of_{X_3}$. This is a smooth surface $S$, with a morphism to $\PP^1$ defined by the kernel of the morphism. In fact the degeneracy locus of the corresponding morphism 
 $\of_{\PP^5}(-1)^2\rightarrow V_3^\vee\otimes \of_{\PP^5}$ is a copy of $\PP^1\times \PP^2$, so 
 $S$ is the intersection of the latter with $X_3$. The projection to $\PP^1$ shows 
 that $S$ is an elliptic surface. Its canonical bundle is 
 $\of_{\PP^1\times \PP^2}(1,0)$, which explains that $h^2(\of_S)=2$, and 
 since it is globally generated, we must have $\kappa(S)=1$. 
\end{fano}

\medskip\begin{fano}\fanoid{2-17-51-2} 
\label[mystyle]{17-51-2-2-41-36-1-3-1-6} $\mZ(\PP^2 \times \PP^5,\of(1,1) \oplus \mQ_{\PP^2}(0,2)).$
\subsubsection*{Invariants}  $h^0(-K)=17, \ (-K)^4=51, \  h^{1,1}=2, \ h^{3,1}=2, \ h^{2,2}=41, \ -\chi(T)=36$.
\subsubsection*{Description} 
\begin{itemize}
    \item $\pi_1 \colon X \to \PP^2$ is a fibration in $\dP_4$.
    \item $\pi_2 \colon X \to \PP^5$ is a blow up $\Bl_{S_8} X_3$ with $S_8$ a octic K3 surface.
    \item Rationality: depends on the cubic fourfold. A cubic fourfold contains an $S_8$ if and only if it  contains a plane, see \cref{subsubC8}.
\end{itemize}

\subsubsection*{Semiorthogonal decompositions}
\begin{itemize}
    \item 3 exceptional objects and $\Db(Z,\mathcal{C}_0)$, for $Z \to \PP^2$ a $\PP^1$-bundle, from $\pi_1$.
    \item 3 exceptional objects, $\Db(S_2,\alpha)$ and $\Db(S_8)$ from $\pi_2$.
\end{itemize}

\subsubsection*{Explanation}
 The description of  $\pi_1$ is obvious.
 
 In order to describe $\pi_2$  we first apply \cref{lem:higherpeco} and identify $\mZ(\PP^2 \times \PP^5,\mQ_{\PP^2}(0,2))$ with the blow up of $\PP^5$ with center the intersection $S_8$ of three quadrics. This is a Fano fivefold $Y$ of index $\iota_Y=2$. By definition our Fano fourfold $X$ is a half-anticanonical section of $Y$. Via the blow up map the anticanonical class is identified as $6H-2S_{8}$; therefore a $(1,1)$ section on $Y$ cuts a (Hodge-special) cubic fourfold containing $S_8$, see also \cite[4.11]{kuznetsovKuchlePicard}, and the result follows. 
\end{fano}

 \begin{rmk}
 Observe that the \emph{double K3 structure} is in line with the prediction of \cite[2.1]{eg2}. In fact $Y$ is an index $2$ Fano fivefold with $H^5(Y)=0$, but with $H^4(Y)$ having a K3 structure coming $H^2(S_8)$, which is inherited by $Z$ by the Lefschetz Theorem on hyperplane sections.
\end{rmk}

\medskip\begin{fano}\fanoid{2-26-93-2} 
\label[mystyle]{26-93-2-2-41-39-1-6-1-7} $\mZ(\PP^5 \times \PP^6,\of(1,1) \oplus \mQ_{\PP^5}(0,1) \oplus \of(3,0))$.

\subsubsection*{Invariants}  $h^0(-K)=26, \ (-K)^4=93, \  h^{1,1}=2, \ h^{3,1}=2, \ h^{2,2}=41, \ -\chi(T)=39$. 
\subsubsection*{Description} 
\begin{itemize}
    \item $\pi_1 \colon X \to \PP^5$ is a blow up $\Bl_{S_6} X_3$, where $S_6$ is a sextic K3 surface.
    \item $\pi_2 \colon X \to \PP^6$ is birational onto a singular fourfold, the intersection $C(X_3) \cap \QQ^5$ of the cone over $X_3$ with a smooth quadric passing through; moreover it contracts a cubic threefold to the
    the (singular) vertex. 
    \item Rationality: unknown. Same as the general cubic fourfold, which always contains a sextic K3 as an intersection with a quadric and a hyperplane.
\end{itemize}

\subsubsection*{Semiorthogonal decompositions}
\begin{itemize}
    \item 3 exceptional objects, the K3 category of $X_3$ and $\Db(S_6)$ from $\pi_1$.
\end{itemize}

\subsubsection*{Explanation}
The zero locus $\mZ(\PP^5 \times \PP^6,\mQ_{\PP^5}(0,1))$ is the blow up $\Bl_v \PP^6$ of $\PP^6$ at one point. This is enough to explain both projections. In fact if we consider the projection to $\PP^6$ we have to consider the intersection between a general cubic equation in six variables (hence, a projective cone over a smooth cubic fourfold) and a smooth quadric, both passing through $v$.

On the other hand if we consider the projection to $\PP^5$ this fourfold becomes $\mZ(\PP_{X_3}(\of \oplus \of(-1)),\mL \boxtimes \of_{X_3}(1))$. We conclude using the Cayley trick, with the blown up locus $\mZ(X_3,\of(2) \oplus \of(1))$ being a K3 of degree 6.
\end{fano}

\medskip\begin{fano}\fanoid{2-17-50-2} 
\label[mystyle]{17-50-2-2-42-38-1-2-2-5}$\mZ(\PP^1 \times \Gr(2,5),\of(0,1) \oplus \of(0,2) \oplus \of(1,1)).$

\subsubsection*{Invariants}  $h^0(-K)=17, \ (-K)^4=50, \  h^{1,1}=2, \ h^{3,1}=2, \ h^{2,2}=42, \ -\chi(T)=38$.
\subsubsection*{Description} 
\begin{itemize}
    \item $\pi_1 \colon X \to \PP^1$ is a fibration in Gushel--Mukai threefolds.
    \item $\pi_2 \colon X \to \Gr(2,5)$ is a blow up $\Bl_{
    S_{10}}X_{10}$ of a degree 10 K3 surface in a Gushel--Mukai fourfold.
    \item Rationality: unknown. Same as the general Gushel--Mukai fourfold, since it contains degree 10 K3 surfaces  as double hyperplane sections.
\end{itemize}

\subsubsection*{Semiorthogonal decompositions}
\begin{itemize}
    \item 4 exceptional objects and an unknown category from $\pi_1$.
    \item 4 exceptional objects, the K3 category of $X_{10}$ and $\Db(S_{10})$ from $\pi_2$.
\end{itemize}

\subsubsection*{Explanation} Simple application of  \cref{lem:blowgrass}.
\end{fano}

\section{Some semiorthogonal decompositions}
\label{sec:SOD}

In this appendix we explain how to obtain semiorthogonal decompositions of our examples by describing those of Fano varieties that are involved in the geometric constructions. In particular, we always used maps $\pi:X \to M$ which are either smooth blow ups or Mori fiber spaces. 
In the first case, we appeal to the blow up formula for derived categories \cite{orlovprojbund}. The case of Mori fiber spaces such as projective bundles and quadric, conic, del Pezzo of degree 4 and 5 fibrations is detailed in this appendix, and resumed in \cref{tab:semiorth2}. Other cases are unknown.

We also summarize in \cref{tab:semiorth1} semiorthogonal decompositions of all possible $M$ and centers of blow ups appearing in the examples, giving details in the case of special cubic and Gushel--Mukai fourfolds.

\TOCstop

\subsection{The case of cubic and GM fourfolds}
\label{appB}
If $X$ is either a cubic fourfold or a Gushel--Mukai fourfold, then $\Db(X)$ is generated by a natural exceptional collection of length 3 (cubic) or 4 (GM) whose complementary Kuznetsov component is a K3 category, which is in general not equivalent to the derived category of a (twisted) K3 surface. However, this can happen for special such fourfolds containing specific surfaces, hence lying on some particular Noether--Lefschetz divisor. Here is a list of cases of interest for this paper, but more are known.

\subsubsection{$\mC_6$: Nodal cubics \cite{kuznetcubicfourfold}} If $X \subset \PP^5$ is a cubic fourfold with a single double point, then the (categorical resolution) of its Kuznetsov component is $\Db(S_6)$, where the degree 6 K3 surface $S_6$ is described as follows. Projecting from the node, one gets a birational map $X \dashrightarrow \PP^4$, and blowing up the node this gives a birational morphism $\widetilde{X} \to \PP^4$ which can be  shown to be the blow up of $\PP^4$ along $S_6$ ($\widetilde{X}$ is the Fano fourfold \cref{49-211-2-1-21-19-1-5-1-6}). 

\subsubsection{$\mC_8$: Cubics containing a plane}\label{subsubC8} If $X \subset \PP^5$ is a smooth cubic fourfold containing a plane, then its Kuznetsov component is $\Db(S_2,\alpha)$ \cite{kuznetcubicfourfold}, where the degree 2 K3 surface $S_2$ and the Brauer class $\alpha$ are described as follows. Projecting from the plane, one gets a birational map $X \dashrightarrow \PP^2$, and blowing up the plane gives a quadric surface fibration $\widetilde{X} \to \PP^2$ degenerating along a smooth sextic 
($\widetilde{X}$ is the Fano fourfold \cref{45-192-2-1-22-19-1-3-1-6}). Note that if $X$ is general in this divisor, its Kuznetsov component cannot be equivalent to the derived category of a K3 surface.

We notice that any cubic fourfold containing a plane has equation of the form $\sum l_i q_i$, where $l_i$ are linear and $q_i$ are quadratic polynomials for $i=1,2,3$, the plane being given by the vanishing of the $l_i$'s. In particular such a cubic fourfold also contains quadric surfaces (where one quadratic and two linear forms vanish), 
del Pezzo surfaces $\dP_4$ (where one linear and two quadratic forms vanish), and an octic K3 surface $S_8$ (where the three quadratic forms vanish).

\subsubsection{$\mC_{14}$: Pfaffian cubics}\label{subsubC14} If $X \subset \PP^5$ is Pfaffian, then its Kuznetsov component is $\Db(S_{14})$, where the degree 14 K3 surface $S_{14}$ is the HP dual of $X$ \cite{kuznetcubicfourfold}.
Note that cubic fourfolds in $\mC_{14}$ contain quintic del Pezzo surfaces and rational quartic scrolls.

\subsubsection{GM fourfolds containing a del Pezzo quintic}\label{subsub:GMdp5} If $X \subset \Gr(2,5)$ is a GM fourfold containing a $\dP_5$, then its Kuznetsov component is $\Db(S_{10})$ \cite{Kuz-Perry-GM}, and $X$ is rational. The surface $S_{10}$ is also a GM variety.

Note moreover that any such $X$ contains a $\sigma$-plane (i.e.\ a plane in $\Gr(2,5)$ whose Schubert class is $\sigma_{3,1}$), but the converse is true only in general \cite[\S 7.5]{Deb-Ili-Man}. A general GM fourfold containing such a plane is rational.

\subsubsection{GM fourfolds containing a $\tau$-quadric} If $X \subset \Gr(2,5)$ is a GM fourfold containing a $\tau$-quadric (i.e.\ a quadratic surface which is a  linear section of some $\Gr(2,4)$) then $X$ is rational and has degree 10 discriminant, but the Kuznetsov component of $X$ has not been proven yet to be equivalent to some $\Db(S_{10})$.

\begin{table}[!hbtp]
    \centering
    \begin{tabular}{||c||c|c||}\hline\hline
        $\mathrm{dim}(X)$ & $X$ & Description of $\Db(X)$ \\  
         \hline \hline
        \multirow{2}{*}{any} & $\PP^n$ & $n+1$ exceptional objects \cite{beilinson} \\ \cline{2-3}
       &  $\Gr(k,n)$ & ${\binom{n}{k}}$ exceptional objects \cite{kapranovquadric}\\
        \hline
        odd & \multirow{2}{*}{$\QQ^n$}  & $n+1$ exceptional objects \cite{kapranovquadric} \\
        even & & $n+2$ exceptional objects \cite{kapranovquadric}  \\ \hline
        2 & Rational surface & $\rho(X)+2$ exceptional objects \\
        \hline
        \multirow{4}{*}{3} &
        Cubic in $\PP^4$ & 2 exceptional objects and a $[5/3]$-CY category \cite{kuznetv14} \\ \cline{2-3}
        & Int.\ of 2 quadrics in $\PP^5$ & 2 exceptional objects and a genus 2 curve \cite{bondalorlov} \\  \cline{2-3}
        & $X_5^3$ & 4 exceptional objects \cite{orlov-v5} \\ \cline{2-3}
        & $(1,1)$ divisor in $\PP^2 \times \PP^2$ & 6 exceptional objects (use HPD \cite{berna-bolo-faenzi}) \\ \hline
        
        \multirow{14}{*}{4} &
        Cubic in $\PP^5$ & 3 exceptional objects and a K3 category \cite{kuznetcubicfourfold} \\ \cline{2-3}
        & GM fourfold in $\Gr(2,5)$ & 4 exceptional objects and a K3 category \cite{Kuz-Perry-GM}\\ \cline{2-3}
        & Int .of 2 quadrics in $\PP^6$ & 12 exceptional objects (see, e.g., \cite[\S 6]{marcello-goncalo-chowgroups}) \\ \cline{2-3}
        & del Pezzo 4fold of degree 5 & 6 exceptional objects (use HPD for $\Gr(2,5)$ \cite{Kuz:ICM2014}) \\ \cline{2-3}
        & deg 12 index 2 & 16 exceptional objects \cite[\S 6.2]{KuzHyperplane} \\ \cline{2-3}
        & deg 14 index 2 & 12 exceptional objects \cite[Cor .10.1]{kuznetgrasslines} \\ \cline{2-3}
        & deg 16 index 2 & 8 exceptional objects (use HPD \cite{KuzHyperplane}) \\ \cline{2-3}
        & deg 18 index 2 & 6 exceptional objects (use HPD \cite{KuzHyperplane}) \\ \cline{2-3}
        & Codim 2 linear sect .of $\PP^3 \times \PP^3$ & 12 exceptional objects (use HPD \cite{berna-bolo-faenzi}) \\ \cline{2-3}
        & $(1,1)$ divisor in $\PP^2 \times \PP^3$ & 9 exceptional objects (use HPD \cite{berna-bolo-faenzi}) \\ \cline{2-3}
        & $(1,2)$ divisor in $\PP^2 \times \PP^3$ & 8 exceptional objects and an unknown category (not trivial!) \\ \cline{2-3}
         & $(1,1)$ divisor in $\PP^2 \times \QQ^3$ & 10 exceptional objects, see \cref{27-100-3-1-24-20-1-2-1-3-1-5} \\ \cline{2-3}
         & $\PP^2$-bundle over $\PP^2$ & 9 exceptional objects \cite{orlovprojbund} \\ 
         \cline{2-3}
         & $\PP^1$-bundle over a $\QQ^3$ & 8 exceptional objects \cite{orlovprojbund} \\
         \cline{2-3}
         & $\PP^1$-bundle over a cubic 3fold & 4 exceptional objects and 2 $[5/3]$-CY categories \cite{orlovprojbund} \\
        \hline \hline
    \end{tabular}
    \caption{\centering Semiorthogonal decompositions for (Fano) of dimension at most $4$ that appear in the examples.}
    \label{tab:semiorth1}
\end{table}

{\small
\begin{table}[!hbtp]
\frenchspacing
    \centering
    \begin{tabular}{||c|c||c|c||}
    \hline\hline
        fiber of $X \to M$  & $M$ & exceptional part & Kuznetsov component \\  
         \hline \hline
         
        \multirow{7}{*}{conic} & $\PP^3$ & 4 exceptional objects &  \\
        \cline{2-3}
        & $\PP^1 \times \PP^2$ & 6 exceptional objects& \\
        \cline{2-3}
        & $\PP^1 \times \PP^1 \times \PP^1$ & 8 exceptional objects & $\Db(M,\mC_0)$,\\
        \cline{2-3}
        & $\Bl_p\PP^3$ & 6 exceptional objects & see \cref{subsub-conic} \\
        \cline{2-3}
        & $\QQ^3$ & 4 exceptional objects & \\
        \cline{2-3}
        & $(1,1)$ in $\PP^2 \times \PP^2$ & 6 exceptional objects & \\
        \cline{2-3}
        & 2 quadrics in $\PP^5$ & 2 exc.\ objects, a genus 2 curve$^\star$ & \\
        \hline
        
        \multirow{2}{*}{quadric surface} & $\PP^2$ with sextic discr. & 6 exceptional objects & twisted K3 of degree 2 \\ \cline{2-4}
        & $\PP^1 \times \PP^1$ with $(4,4)$ discr. & 8 exceptional objects
        & twisted K3 of degree 8 \\
        \hline
        
        \multirow{2}{*}{$\dP_5$} & $\PP^2$ & 6 exceptional objects & $\Db(Z)$ for $Z \to M$ \\ \cline{2-3}
        & $\PP^1 \times \PP^1$ & 8 exceptional objects
        & of degree five, see \cref{subsub-dp5} \\
        \hline
        
        \multirow{2}{*}{$\dP_4$} & $\PP^2$ & 3 exceptional objects & $\Db(Z,\mC_0)$ for $Z \to M$ \\ \cline{2-3}
        & $\PP^1 \times \PP^1$ & 4 exceptional objects
        & a $\PP^1$-bundle, see \cref{subsub-intersections} \\
        \hline
        
        \multirow{2}{*}{$\dP_3$} & $\PP^2$ & 3 exceptional objects & \multirow{2}{*}{unknown} \\ \cline{2-3}
        & $\PP^1 \times \PP^1$ & 4 exceptional objects
        & \\
        
        \hline 
        
        \multirow{2}{*}{Int.\ of 2 quadrics} & \multirow{9}{*}{$\PP^1$} & \multirow{8}{*}{4 exceptional objects} & $\Db(S,\mC_0)$ for $S$ a \\
        & & & Hirz.\ surface, see \cref{subsub-intersections} \\ \cline{1-1}\cline{4-4}
        Cubic 3folds &  &  & \multirow{7}{*}{unknown}\\
        \cline{1-1}             GM 3 folds & & &\\
        \cline{1-1}        $X^3_{12}$ & & & \\
       \cline{1-1}         $X^3_{14}$ & & & \\
       \cline{1-1}         $X^3_{18}$ & & & \\
        \cline{1-1}
       $X^3_{16}$ & &  & \\
       \cline{1-1} \cline{3-3}
       $2,3$ c.\ int.\ in $\PP^5$ & & 2 exceptional objects &\\
    
    \hline\hline    
    \end{tabular}
    \caption{\normalsize \centering Semiorthogonal decompositions for fourfolds with a Mori fiber space structure.
        In the case marked with $^\star$ the category of a genus 2 curve comes from the base, so one should not consider it as part of the Kuznetsov component.}
    \label{tab:semiorth2}
\end{table}
}
\nonfrenchspacing

\subsection{Quadric fibrations}
Let $\pi:X \to M$ be a quadric bundle of relative dimension $n$. This means that there exists a rank $n+2$ vector bundle $\mE$ and a line bundle $\mK$ on $M$ such that $X$ is the zero locus of a quadratic form $q: \mK^\vee \to S^2 \mE^\vee$. We assume that $q$ is general so that the generic fiber of $\pi$ is a smooth quadric. The locus $\Delta\subset M$ where the fiber is singular of corank 1 is a divisor, called the \emph{discriminant divisor}, which is in general singular in codimension $2$, and whose singularities correspond to corank 2 fibers.

In this case, one can associate to $q$ the sheaf $\mC_0$ of the even parts of the Clifford algebra, and show that the Lefschetz part of $\Db(X)$ is given by $n$ copies of $\Db(M)$, while the relative Kuznetsov component can be described as $\Db(M,\mC_0)$, the derived category of complexes of $\mC_0$-algebras on $M$ \cite{kuznetconicbundles}.

\subsubsection{The case of conic bundles with a section}\label{subsub-conic}
In the case of conic bundles, there is a structure of $\Z/2\Z$-root stack $\hat{M}$ on $M$, and $\mC_0$ lifts to an algebra $\mB_0$ on $\hat{M}$ such that $\Db(M,\mC_0) \simeq \Db(\hat{M},\mB_0)$. The algebra $\mB_0$ is Azumaya outside the singularities of the discriminant divisor \cite{kuznetconicbundles} and its triviality is equivalent to $\pi: X \to M$ having a section (this is in turn equivalent to $X$ being $k(M)$-rational).
The dg category of perfect complexes on the root stack $\hat{M}$ admits a semiorthogonal decomposition \cite{ishii-ueda,bergh-lunts-schnu}:
$$\mathrm{Perf}(\hat{M})= \sod{\Db(M),\mathrm{Perf}(\Delta)},$$
and the category of perfect complexes $\mathrm{Perf}(\Delta)$ is equivalent to $\Db(\Delta)$ if and only if $\Delta$ is smooth. It follows, that in the case where $\Delta$ is smooth and $\mB_0$ is trivial, we have a semiorthogonal decomposition of $\Db(X)$ by two copies of $\Db(M)$ and one copy of $\Db(\Delta)$. This explains the K3 structure in examples \cref{31-120-3-1-2-22-20-1-6-1-8} and \cref{25-90-3-5-1-22-23-1-5-1-7}.
In the case where $\Delta$ is not smooth one can wonder whether the copy of $\Db(\Delta)$ can be replaced by a categorical resolution of singularities. We can ask the following question.

\begin{question}
Let $X$ be a fourfold and $\pi: X \to M$ be a conic bundle with discriminant divisor $\Delta$ with finitely many singular points. Assume that $\pi:X  \to M$ has a section. Is there a categorical resolution of singularities $\mD$ of $\Delta$ such that
$$\Db(X)=\sod{\mD,\Db(M),\Db(M)}$$
is a semiorthogonal decomposition?
\end{question}
 
Particularly relevant to our study is the case where $\Delta$ is resolved by a K3, for example, when $\Delta$ is a singular quartic in $\PP^3$. 
 
\subsubsection{The case of quadric surface bundles}\label{subsub:quad-surf-fib}
In the case where $\pi: X \to M$ is a quadric surface bundle, the situation is much simpler. Indeed, there is a degree 2 map $T \to M$ ramified along $\Delta$ such that $\mC_0$ lifts to an algebra $\mB_0$ which is Azumaya outside the singularities of $\Delta$. In particular, if $\Delta$ is smooth (and this is in general the case when $X$ is a fourfold), then $\Db(M,\mC_0) \simeq \Db(T,\mB_0)$ is the derived category of twisted sheaves on the variety $T$.

In the examples considered in this paper, there are two families
of fourfolds $X$ with a quadric surface bundle $\pi: X \to M$ and they both give rise to twisted K3 surfaces:

\begin{enumerate}
    \item $M$ is $\PP^2$ and $\Delta$ is a smooth sextic, so that $T$ is a degree 2 K3 surface.
    \item $M$ is $\PP^1  \times \PP^1$ and $\Delta$ is a $(4,4)$ divisor, so that $T$ is a degree 8 K3 surface of Picard rank 2.
\end{enumerate}

\subsubsection{The case of fibrations in intersections of quadrics}\label{subsub-intersections}
Let $\pi:X \to M$ be a Mori fiber space whose general fiber is the intersection of two quadrics. As explained in \cite{auel-berna-bolo}, one can construct a $\PP^1$-bundle $Z \to M$ and a quadric fibration on $Y \to Z$ with associated sheaf $\mC_0$ of even parts of Clifford algebras, such that the Kuznetsov component of the fibration is equivalent to $\Db(Z,\mC_0)$. The considerations on conic and quadric surface bundles may then be carried over these cases as well.

\subsection{\texorpdfstring{$\dP_5$}{dP5} fibrations}\label{subsub-dp5}
If $X \to M$ is a $\dP_5$ fibration, the Kuznetsov component is equivalent to $\Db(Z)$ for $Z \to M$ a flat degree 5 cover \cite{xie-dp5}. In our cases, $M$ is either $\PP^2$ or a quadric, and the question of $Z$ being (related to) a K3 surface remains open. Note that if $X$ is relatively minimal over $M$, then $Z$ cannot be disconnected (this would increase the Picard rank over $k(M)$, see \cite{auel-berna-dp}).

\TOCstart

\section{Tables}
\label{Tabless}
\renewcommand\thefano{F--\arabic{fano}}

In the tables below we follow the notation of the paper, together with some additions. In particular, as usual  $X_3$ denotes a cubic fourfold and $X_{10}$ a Gushel--Mukai fourfold. More generally, $X_d$ denotes a Fano fourfold of index 2 and degree d and $Y_d$ denotes a Fano fourfold of index 3 and degree d. If needed, the superscript $Y^n_d$ will denote Fano manifolds obtained as hyperplane sections of $Y_d$ of dimension $n$. Recall that $W_3$ denotes a cubic threefold. By $D \subset \PP^n \times \PP^m$ (resp.\ $D^c$) we denote a divisor (resp.\ the complete intersection of $c$ divisors), whose bidegree is specified in the detailed description of each Fano fourfold. 
Also $Z^\circ$ denotes a singular fourfold, whose description is expanded in the appropriate section.

In terms of invariants, $h^{3,1}$ is always 1 for fourfolds of K3 type, so it is not listed.  Under the column ``Bl'' are listed the varieties $Y$ from our lists such that $X \cong \Bl_S Y$, where $X$ is the FK3 in question. In the last column, ``+'' means that every variety in the deformation class is rational, ``*'' that some special member are rational, ``?'' means that rationality properties are not known.

\begin{longtable}{ccccccccccc}
\caption{FK3 with $\rho=1$}\label{tab:fk3rho1}\\
\toprule
\multicolumn{1}{c}{ID}&\multicolumn{1}{c}{$\rho$}& \multicolumn{1}{c}{$h^{2,2}$} & \multicolumn{1}{c}{$h^{1,2}$} &\multicolumn{1}{c}{$h^0(-K)$}& \multicolumn{1}{c}{$K^4$}& \multicolumn{1}{c}{$-\chi(T)$}& \multicolumn{1}{c}{$G$}& \multicolumn{1}{c}{$\mF$}& \multicolumn{1}{c}{Bl}& \multicolumn{1}{c}{Rat} \\
\cmidrule(lr){1-1}\cmidrule(lr){2-2}\cmidrule(lr){3-3} \cmidrule(lr){4-4} \cmidrule(lr){5-5} \cmidrule(lr){6-6} \cmidrule(lr){7-7} \cmidrule(lr){8-8} \cmidrule(lr){9-9} \cmidrule(lr){10-10}\cmidrule(lr){11-11}
\endfirsthead
\multicolumn{5}{l}{\vspace{-0.25em}\scriptsize\emph{\tablename\ \thetable{} continued from previous page}}\\
\midrule
\endhead
\multicolumn{5}{r}{\scriptsize\emph{Continued on next page}}\\
\endfoot
\bottomrule
\endlastfoot
\evnrow $X_3$&1&21&0&55&243&20&$\PP^5$& $\of(3)$&-&* \\
$X_{10}$ & $1$& $22$&0& 39&160&24&$\Gr(2,5)$&$\of(1) \oplus \of(2)$& - &*  \\
\evnrow $Z_{66}$&1&24&0&20&66&25&$\Gr(3,7)$& $\W^2\mU^{\vee} \oplus \W^3\mQ \oplus \of(1)$&-&? \\

\end{longtable}

\begin{longtable}{ccccccccccc}
\caption{FK3 from the cubic fourfold}\label{tab:fk3cubic}\\
\toprule
\multicolumn{1}{c}{ID}&\multicolumn{1}{c}{$\rho$}& \multicolumn{1}{c}{$h^{2,2}$} & \multicolumn{1}{c}{$h^{1,2}$} &\multicolumn{1}{c}{$h^0(-K)$}& \multicolumn{1}{c}{$K^4$}& \multicolumn{1}{c}{$-\chi(T)$}& \multicolumn{1}{c}{$G$}& \multicolumn{1}{c}{$\mF$}& \multicolumn{1}{c}{Bl}& \multicolumn{1}{c}{Rat} \\
\cmidrule(lr){1-1}\cmidrule(lr){2-2}\cmidrule(lr){3-3} \cmidrule(lr){4-4} \cmidrule(lr){5-5} \cmidrule(lr){6-6} \cmidrule(lr){7-7} \cmidrule(lr){8-8} \cmidrule(lr){9-9} \cmidrule(lr){10-10}\cmidrule(lr){11-11}
\endfirsthead
\multicolumn{5}{l}{\vspace{-0.25em}\scriptsize\emph{\tablename\ \thetable{} continued from previous page}}\\
\midrule
\endhead
\multicolumn{5}{r}{\scriptsize\emph{Continued on next page}}\\
\endfoot
\bottomrule
\endlastfoot
\evnrow \cref{36-144-2-1-28-28-1-2-1-6}&2&28&0&36&144&28&$\PP^1 \times \PP^5$& $\of(0,3) \oplus \of(1,1)$&$X_3$&* \\ 
\cref{28-99-2-1-1-23-29-1-3-1-6}&2&23&1&28&99&29&$\PP^2 \times \PP^5$&$\mQ_{\PP^2}(0,1) \oplus \of(0,3)$&$X_3$&*\\ 
\evnrow \cref{45-192-2-1-22-19-1-3-1-6} &2&22&0&45&192&19&$\PP^2 \times \PP^5$&$\of(1,2) \oplus \mQ_{\PP^2}(0,1)$&$X_3$&*\\ 
\cref{40-163-2-1-23-24-1-4-1-6}&2&23&0&40&163&24&$ \PP^3 \times \PP^5$&$\of(1,2) \oplus \mQ_{\PP^3}(0,1)$ &$X_3$&* \\ 

\evnrow \cref{39-161-2-1-23-21-1-6-2-8} &2&23&0&39&161&21&$\PP^5 \times \Gr(2,8)$&$\mU_{\Gr(2,8)}^{\vee}(1,0) \oplus {} $&$X_3, Z^{\circ}$&?\\ 
  \evnrow & &&&&&&&$  \of(1,1) \oplus \mQ_{\PP^5} \boxtimes \mU^{\vee}_{\Gr(2,8)} $& & \\ 
\cref{49-211-2-1-21-19-1-5-1-6}&2&21&0&49&211&19&$\PP^4 \times \PP^5$&$\mQ_{\PP^4}(0,1) \oplus \of(2,1)$& $X_3^{\circ}, \PP^4$& +\\ 
\evnrow \cref{24-86-2-1-26-24-1-5-1-6}&2&26&0&24&86&24&$\PP^4 \times \PP^5$& $\of(1,1) \oplus \W^3 \mQ_{\PP^4}(0,1)$&$\PP^4, X_3$&+ \\ 
\cref{30-115-2-1-27-26-1-3-8}&2&27&0&30&115&26&$\Fl(1,3,8)$&$\mQ_2^{\oplus 2} \oplus \of(1,1) \oplus \mR_2^{\vee}(1,1) $& $X_3, Z^{\circ}$& ?\\ 
\evnrow \cref{27-101-2-1-23-21-1-6-2-4} & $2$& $23$&0& 27&101&21&$\PP^5 \times \Gr(2,4)$&$(\mU_{\Gr(2,4)}^{\vee}(1,0))^
{\oplus 2} \oplus \of(1,1)$& $\Gr(2,4), X_3$ &+  \\ 
\cref{20-63-3-1-38-36-1-2-1-2-1-6}&3&38&0&20&63&28&$ \PP^1 \times \PP^1 \times \PP^5$&$\of(0,0,3) \oplus {} $&\cref{36-144-2-1-28-28-1-2-1-6}&*\\
  & &&&&&&&$  \of(0,1,1) \oplus \of(1,0,1) $& & \\ 
\evnrow \cref{24-81-3-1-1-30-31-1-2-1-3-1-6}&3&30&1&24&81&31& $\PP^1 \times \PP^2 \times \PP^5$& $\of(0,0,3) \oplus {} $& \cref{28-99-2-1-1-23-29-1-3-1-6}&*\\ 
  \evnrow & &&&&&&&$ \mQ_{\PP^2}(0,0,1) \oplus \of(1,1,0)  $& & \\ 
\cref{27-99-3-1-30-27-1-2-1-3-1-6}&3&30&0&27&99&27&$\PP^1 \times \PP^2 \times \PP^5$&$\of(0,1,2) \oplus {}$&  \cref{45-192-2-1-22-19-1-3-1-6} &* \\ 
  & &&&&&&&$  \mQ_{\PP^2}(0,0,1) \oplus \of(1,0,1) $& & \\ 
\evnrow \cref{34-134-3-1-28-25-1-2-1-2-6}&3&28&0&34&134&25& $\PP^1 \times \Fl(1,2,6)$ & $\mQ_2 \oplus {}$ & \cref{49-211-2-1-21-19-1-5-1-6}, & +\\
  \evnrow & &&&&&&&$\of(0;1,2) \oplus \of(1;0,1)  $& $\Bl_{\PP^2}\PP^4$ & \\ 
\cref{30-113-3-1-30-28-1-2-1-4-1-6}&3&30&0&30&113&28&$\PP^1 \times \PP^3 \times \PP^5$&$\of(0,1,2) \oplus {}$&\cref{40-163-2-1-23-24-1-4-1-6},& * \\
  & &&&&&&&$ \mQ_{\PP^3}(0,0,1) \oplus \of(1,1,0) $&\cref{36-144-2-1-28-28-1-2-1-6} & \\ 
\evnrow \cref{35-141-3-1-23-18-1-3-1-3-1-6}&3&23&0&35&141&18&$ \PP^2_1 \times \PP^2_2 \times \PP^5$&$\mQ_{\PP^2_1}(0,0,1) \oplus {} $& \cref{45-192-2-1-22-19-1-3-1-6},&+ \\
\evnrow & &&&&&&&$ \mQ_{\PP^2_2}(0,0,1) \oplus \of(1,1,1) $& $\PP^2 \times \PP^2$ & \\ 
\cref{42-176-3-1-23-18-1-3-1-5-1-6}& 3 &23&0&42&176&18&$\PP^2 \times \PP^4 \times \PP^5$&$\mQ_{\PP^4}(0,0,1) \oplus {} $& \cref{49-211-2-1-21-19-1-5-1-6}, &+\\
  & &&&&&&&$\of(1,1,1) \oplus \mQ_{\PP^2}(0,1,0) $&$\Bl_{\PP^1}\PP^4$ & \\ 
\evnrow \cref{37-149-3-1-24-21-1-3-1-4-1-6}&3&24&0&37&149&21& $\PP^2 \times \PP^3 \times \PP^5$& $\mQ_{\PP^3}(0,0,1) \oplus {}$ & \cref{45-192-2-1-22-19-1-3-1-6},&? \\ 
\evnrow & &&&&&&&$ \of(1,0,2) \oplus \mQ_{\PP^2}(0,1,0) $& \cref{40-163-2-1-23-24-1-4-1-6} & \\ 
\end{longtable}

\begin{longtable}{ccccccccccc}
\caption{FK3 from the GM fourfold}\label{tab:fk3gm}\\
\toprule
\multicolumn{1}{c}{ID}&\multicolumn{1}{c}{$\rho$}& \multicolumn{1}{c}{$h^{2,2}$} & \multicolumn{1}{c}{$h^{1,2}$} &\multicolumn{1}{c}{$h^0(-K)$}& \multicolumn{1}{c}{$K^4$}& \multicolumn{1}{c}{$-\chi(T)$}& \multicolumn{1}{c}{$G$}& \multicolumn{1}{c}{$\mF$}& \multicolumn{1}{c}{Bl}& \multicolumn{1}{c}{Rat} \\
\cmidrule(lr){1-1}\cmidrule(lr){2-2}\cmidrule(lr){3-3} \cmidrule(lr){4-4} \cmidrule(lr){5-5} \cmidrule(lr){6-6} \cmidrule(lr){7-7} \cmidrule(lr){8-8} \cmidrule(lr){9-9} \cmidrule(lr){10-10}\cmidrule(lr){11-11}
\endfirsthead
\multicolumn{5}{l}{\vspace{-0.25em}\scriptsize\emph{\tablename\ \thetable{} continued from previous page}}\\
\midrule
\endhead
\multicolumn{5}{r}{\scriptsize\emph{Continued on next page}}\\
\endfoot
\bottomrule
\endlastfoot
\evnrow \cref{22-74-2-1-30-30-1-3-2-5}&2&30&0&22&74&30&$\PP^2 \times \Gr(2,5)$& $\mU^{\vee}_{\Gr(2,5)}(1,0) \oplus {} $&$X_{10}$&* \\
\evnrow  & &&&&&&&$ \of(0,1) \oplus \of(0,2) $& & \\ 
\cref{23-80-2-1-27-26-1-3-2-5}&2&27&0&23&80&26&$\PP^2 \times \Gr(2,5)$&$\of(0,1) \oplus {} $& $X_{10}$& + \\
& &&&&&&&$\of(1,1) \oplus \mQ_{\PP^2}(0,1) $&  & \\ 
\evnrow \cref{26-94-2-1-28-28-1-2-5} &2&28&0&26&94&28 &$\Fl(1,2,5)$&$\of(1,0) \oplus {}$&$X_{10}$&* \\
\evnrow  & &&&&&&&$ \of(0, 1) \oplus \of(0,2) $& & \\ 
\cref{30-114-2-1-24-24-2-4-2-5}&2&24&0&30&114&24&$\Gr(2,4) \times \Gr(2,5)$&$\mQ_{\Gr(2,4)} \boxtimes \mU^\vee_{\Gr(2,5)} \oplus {}$& $X_{10}$& ? \\
 & &&&&&&&$ \of(0,2) \oplus  \of(1,0)$&  & \\ 
 \evnrow \cref{33-130-2-1-23-22-2-4-2-5}&2&23&0&33&130&22&$ \Gr(2,4) \times \Gr(2,5)$&$\mQ_{\Gr(2,4)} \boxtimes \mU^\vee_{\Gr(2,5)} \oplus {} $&$X_{10}, $& + \\
 \evnrow  & &&&&&&&$\of(0,1) \oplus \of(1,1)$&$\Gr(2,4)$  & \\ 
 \cref{30-115-2-1-24-23-1-3-5}&2&24&0&30&115&23&$\Fl(1,3,5)$&$\of(0,1)\oplus {} $&$X_{10},$&+\\
  & &&&&&&&$\of(1,1) \oplus \mR_2(0,1) $&$\PP^4 $  & \\ 
\end{longtable}

{\small
\begin{longtable}{ccccccccccc}
\caption{FK3 from the K3 surfaces}\label{tab:ffromk3}\\
\toprule
\multicolumn{1}{c}{ID}&\multicolumn{1}{c}{$\rho$}& \multicolumn{1}{c}{$h^{2,2}$} & \multicolumn{1}{c}{$h^{1,2}$} &\multicolumn{1}{c}{$h^0(-K)$}& \multicolumn{1}{c}{$K^4$}& \multicolumn{1}{c}{$-\chi(T)$}& \multicolumn{1}{c}{$G$}& \multicolumn{1}{c}{$\mF$}& \multicolumn{1}{c}{Bl}& \multicolumn{1}{c}{Rat} \\
\cmidrule(lr){1-1}\cmidrule(lr){2-2}\cmidrule(lr){3-3} \cmidrule(lr){4-4} \cmidrule(lr){5-5} \cmidrule(lr){6-6} \cmidrule(lr){7-7} \cmidrule(lr){8-8} \cmidrule(lr){9-9} \cmidrule(lr){10-10}\cmidrule(lr){11-11}
\endfirsthead
\multicolumn{5}{l}{\vspace{-0.25em}\scriptsize\emph{\tablename\ \thetable{} continued from previous page}}\\
\midrule
\endhead
\multicolumn{5}{r}{\scriptsize\emph{Continued on next page}}\\
\endfoot
\bottomrule
\endlastfoot
\evnrow \cref{19-60-2-1-32-31-1-2-2-5} &2&32&0&19&60&31& $ \PP^1 \times \Gr(2,5)$& $\of(1,1) \oplus \mU_{\Gr(2,5)}^{\vee}(0,1)$& $X_{12}$& + \\ 
\cref{21-70-2-1-28-27}&2&28&0&21&70&27&$\PP^1 \times \Gr(2,6)$& $\of(1,1) \oplus \of(0,1)^{\oplus 4}$ & $X_{14}$& + \\ 
\evnrow \cref{23-80-2-1-24-23-1-2-3-6} &2&24&0&23&80&23& $\PP^1 \times \Gr(3,6)$& $\of(1,1) \oplus {} $& $X_{16}$& ? \\
\evnrow  & &&&&&&&$ \of(0,1)^{\oplus 2} \oplus \W^2\mU_{\Gr(3,6)}^{\vee}$&  & \\ 
\cref{32-124-2-1-28-28-1-2-8}&2&28&0&32&124&28&$\Fl(1,2,8)$& $\mQ_2 \oplus \of(1,1) \oplus \of(0,2)^{\oplus 2}$& $Y_4, Z^{\circ}$&+\\ 
\evnrow \cref{25-90-2-1-22-21-1-2-2-7}&2&22&0&25&90&21&$\PP^1 \times \Gr(2,7)$&$\of(0,1) \oplus {} $& $X_{18}$&+\\
\evnrow  & &&&&&&&$\of(1,1) \oplus \mQ_{\Gr(2,7)}^{\vee}(0,1) $&  & \\ 
\cref{39-160-2-1-22-21-1-3-1-5}&2&22&0&39&160&21&$\PP^2 \times \PP^4$&$\of(1,1)\oplus \of(1,2) $& $\PP^4$& + \\ 
\evnrow \cref{39-160-2-1-22-21-1-2-2-4} &2&22&0&39&160&21&$\PP^1 \times \Gr(2,4)$&$\of(1,2)$& $\Gr(2,4)$&+ \\ 
\cref{23-80-2-1-28-27-1-4-1-7}&2&28&0&23&80&27& $\PP^3 \times \PP^6$&$\of(0,2) \oplus {} $& $Y_4$& + \\
 & &&&&&&&$\of(1,1) \oplus \W^2 \mQ_{\PP^3}(0,1) $& & \\ 
\evnrow \cref{29-110-2-1-22-21-1-2-5}&2&22&0&29&110&21&$\Fl(1,2,5)$&$\of(1,1) \oplus \of(0,1)^{\oplus 2}$& $Y_5, \PP^4$& +\\ 
\cref{27-100-2-1-24-23-1-6-2-4} &2&24&0&27&100&23&$\PP^5 \times \Gr(2,4)$&$\mU_{\Gr(2,4)}^{\vee}(1,0) \oplus {} $& $\Gr(2,4)$&+ \\
& &&&&&&&$\of(1,1) \oplus \mQ_{\Gr(2,4)}(1,0) $&  & \\ 
\evnrow \cref{26-95-2-1-23-22-1-4-2-5}&2&23&0&26&95&22&$ \PP^3 \times \Gr(2,5)$&$\of(0,1)^{\oplus 2} \oplus {} $& $Y_5$& +\\
\evnrow  & &&&&&&&$ \mU^{\vee}_{\Gr(2,5)}(1,0) \oplus \of(1,1)$&  & \\ 
 \cref{25-90-2-1-24-23-2-4-2-4}&2&24&0&25&90&23&$(\Gr(2,4))^2 $& $\of(0,1) \oplus {}$ & $\Gr(2,4)$& + \\
    & &&&&&&&$ \mU^{\vee}_{\Gr(2,4)_2}(1,0) \oplus \of(1,1)$&  & \\ 
\evnrow \cref{19-60-3-5-1-22-23-1-2-1-5} &3&22&5&19&60&23&$\PP^1 \times \PP^1 \times \PP^4$&$\of(0,0,3) \oplus \of(1,1,1)$& $\PP^1 \times W_3$&? \\ 
 \cref{23-80-3-2-1-33-20-1-2-1-2-1-6}&3&33&2&23&80&20&$\PP^1 \times \PP^1 \times \PP^5$& $\of(0,0,2)^{\oplus 2} \oplus \of(1,1,1)$&$\PP^1 \times Y^3_{4}$&+\\ 
\evnrow \cref{27-100-3-1-22-18-1-2-1-2-2-5}&3&22&0&27&100&18&$(\PP^1)^2 \times  \Gr(2,5)$&$\of(0,0,1)^{\oplus 3} \oplus \of(1,1,1)$&$\PP^1 \times Y^3_5$& + \\ 
 \cref{29-110-3-1-24-20-1-2-1-3-2-4}&3&24&0&29&110&20&$ \PP^1 \times \PP^2 \times \Gr(2,4)$&$\mQ_{\PP^2}(0,0,1) \oplus {} $& \cref{39-160-2-1-22-21-1-2-2-4}, & + \\
  & &&&&&&&$\of(1,1,1) $&$\Bl_C \Gr(2,4) $  & \\ 
\evnrow \cref{22-74-3-1-32-29-1-2-1-2-1-6}&3&32&0&22&74&29&$ \PP^1_1 \times \PP^1_2 \times \PP^5$&$\of(0,0,2) \oplus {} $& \cref{39-160-2-1-22-21-1-2-2-4},&+ \\
\evnrow & &&&&&&&$\of(0,1,1) \oplus \of(1,0,2) $&$  \Bl_{\PP^1 \times \PP^1}\Gr(2,4)$  & \\ 
   \cref{39-160-3-1-22-18-1-2-1-2-1-4} &3&22&0&39&160&18&$(\PP^1)^2 \times  \PP^3$ &$\of(1,1,2)$& $\PP^1 \times \PP^3$& + \\ 
\evnrow \cref{28-104-3-1-27-24-1-2-1-3-2-4}&3&27&0&28&104&24&$\PP^1 \times \PP^2 \times \Gr(2,4)$&$\mU^{\vee}_{\Gr(2,4)}(0,1,0) \oplus {} $& $\Bl_{\PP^2} \Gr(2,4),$& +\\
\evnrow & &&&&&&&$\of(1,0,2) $&$ \PP_{\PP^2}(\mQ^{\vee} \oplus \of(-1)) $  & \\ 
  \cref{32-125-3-1-24-20-1-3-1-3-1-5}&3&24&0&32&125&20&$(\PP^2)^2 \times \PP^4$&$\of(1,0,1) \oplus {} $& \cref{39-160-2-1-22-21-1-3-1-5} &+ \\
  & &&&&&&&$ \of(1,1,1) \oplus \mQ_{\PP^2_2}(0,0,1)$&$\Bl_{\PP^1}\PP^4$  & \\ 
\evnrow \cref{25-89-3-1-30-27-1-2-1-3-1-5}&3&30&0&25&89&27& $\PP^1 \times \PP^2 \times \PP^4$&$\of(0,1,1)\oplus {}$& \cref{39-160-2-1-22-21-1-3-1-5},  & + \\
\evnrow  & &&&&&&&$ \of(0,1,2)\oplus \of(1,0,1)$&$\Bl_{\PP^2}\PP^4 $  & \\ 
   \cref{31-119-3-1-24-21-1-3-1-3-1-5}&3&24&0&31&119&21&$(\PP^2)^2 \times \PP^4$ & $\of(0,1,2) \oplus {}$ & \cref{39-160-2-1-22-21-1-3-1-5}, & + \\
      & &&&&&&&$\of(1,1,0) \oplus \mQ_{\PP^2_1}(0,0,1) $&$\Bl_{\PP^1} \PP^4 $  & \\ 
\evnrow  \cref{33-130-3-1-24-20-1-2-1-4-1-6}&3&24&0&33&130&20& $\PP^1 \times \PP^3 \times \PP^5$& $\mQ_{\PP^3}(0,0,1) \oplus {} $& \cref{39-160-2-1-22-21-1-2-2-4}, $\PP^1 \times \PP^3$ & +\\
\evnrow    & &&&&&&&$ \of(0,1,1) \oplus \of(1,1,1)$&$\Bl_{\PP^1}\Gr(2,4) $  & \\ 
 \cref{27-100-3-1-26-22-1-2-1-4-1-4}&3&26&0&27&100&22& $ \PP^1 \times \PP^3 \times \PP^3$& $\of(0,1,1)^{\oplus 2} \oplus {} $& $\PP^1 \times \PP^3,$& +\\
   & &&&&&&&$ \of(1,1,1)$&$  D^2 \subset \PP^3 \times \PP^3$  & \\ 
\evnrow \cref{31-120-3-1-22-18-1-2-1-2-4}&3&22&0&31&120&18& $\PP^1 \times \Fl(1,2,4)$&$\of(0;0,1) \oplus \of(1;1,1)$&$\PP^1 \times \Q_3, \PP_{\Q_3}(\mU),$& + \\
\evnrow      & &&&&&&&&$ \PP^1 \times \PP^3 $  & \\ 
 \cref{27-100-3-1-24-20-1-2-1-3-1-5} &3&24&0&27&100&20&$\PP^1 \times \PP^2 \times \PP^4$&$\of(0,0,2) \oplus {} $& $\PP^1 \times \Q^3,$&+ \\
 & &&&&&&&$\of(0,1,1) \oplus \of(1,1,1) $&$  D \subset \PP^2 \times \Q^3$  & \\ 
 \evnrow \cref{43-180-3-1-22-18-1-4-1-5-1-5}& 3&22&0&43&180&18&$\PP^3 \times (\PP^4)^2$&$\of(1,1,1) \oplus {} $& $Z^{\circ},$&+ \\
  \evnrow     & &&&&&&&$\mQ_{\PP^3}(0,1,0) \oplus \mQ_{\PP^3}(0,0,1) $&$ \Bl_{p}\PP^4 $  & \\ 
  \cref{23-80-3-1-30-26-1-2-1-3-1-4} &3&30&0&23&80&26& $\PP^1 \times \PP^2 \times \PP^3$&$\of(0,1,2)\oplus \of(1,1,1)$& $D \subset \PP^2 \times \PP^3, $&+ \\
  & &&&&&&&$ $&$\PP^1 \times \PP^3 $  & \\ 
  \evnrow    \cref{29-110-3-1-24-20-1-3-1-3-1-4} &3&24&0&29&110&20& $(\PP^2)^2 \times \PP^3$ & $\of(0,1,1) \oplus {}$ & $\PP^2 \times \PP^2,$ &+ \\
  \evnrow      & &&&&&&&$\of(1,0,1) \oplus \of(1,1,1) $&$  D \subset \PP^2 \times \PP^3$  & \\ 
    \cref{36-145-3-1-23-19-1-3-1-4-1-5} &3&23&0&36&145&19&$ \PP^2 \times \PP^3 \times \PP^4$& $\mQ_{\PP^3}(0,0,1) \oplus {}$ & \cref{39-160-2-1-22-21-1-3-1-5}, $\Bl_{p}\PP^4, $& +\\
     & &&&&&&&$\of(1,1,0) \oplus \of(1,1,1) $&$D \subset \PP^2 \times \PP^3 $  & \\ 
 \evnrow    \cref{34-133-3-1-25-23-1-3-1-4-1-5}&3&25&0&34&133&23& $\PP^2 \times \PP^3 \times \PP^4$& $\of(0,1,2) \oplus {}$ & $\Bl_S\PP^4,  $& + \\
 \evnrow      & &&&&&&&$\mQ_{\PP^2}(0,0,1) \oplus \mQ_{\PP^2}(0,1,0) $&$\Bl_{\PP^1}\PP^4 $  & \\ 
     \cref{37-150-3-1-22-18-1-2-1-5-1-6} &3&22&0&37&150&18&$\PP^1 \times \PP^4 \times \PP^5$&$\mQ_{\PP^4}(0,0,1) \oplus {} $& $Z^{\circ}$&+ \\
       & &&&&&&&$\of(0,2,0) \oplus \of(1,1,1) $&$ \PP^1 \times \Q_3 $  & \\ 
\evnrow \cref{35-140-4-1-24-17-1-2-1-2-1-3-1-4}&4&24&0&35&140&17& $(\PP^1)^2 \times \PP^2 \times \PP^3$ & $\mQ_{\PP^2}(0,0,0,1) \oplus {} $& \cref{39-160-3-1-22-18-1-2-1-2-1-4},  & + \\
 \evnrow  & &&&&&&&$\of(1,1,1,1) $&$\PP^1 \times \Bl_p \PP^3 $  & \\ 
  \cref{34-134-4-1-24-18-1-2-1-3-6}&4&24&0&34&134&18&$\PP^1 \times \Fl(1,3,6)$&$\mQ_2^{\oplus 2} \oplus {} $& $\Bl_{2p}\Gr(2,4)$&+\\
    & &&&&&&&$\of(0;0,2) \oplus \of(1;2,0) $&  & \\ 
  \evnrow       \cref{28-104-4-1-28-22-1-2-1-2-1-2-1-4}&4&28&0&28&104&22&$(\PP^1)^3 \times \PP^3$& $\of(0,0,1,1) \oplus \of(1,1,0,2)$& $ \PP^1 \times \PP^3,$&+ \\
  \evnrow        & &&&&&&&$ $&$  \PP^1 \times \Bl_p \PP^3$  & \\ 
  \cref{31-120-4-1-24-17-1-2-1-2-1-3-1-3}&4&24&0&31&120&17&$(\PP^1)^2 \times (\PP^2)^2$&$\of(0,0,1,1) \oplus \of(1,1,1,1)$& $D \subset \PP^1 \times (\PP^2)^2,$&+ \\
   & &&&&&&&&$  (\PP^1)^2 \times \PP^2$  & \\ 
\evnrow   \cref{31-120-5-1-26-16-1-2-1-2-1-2-1-2-1-2}&5&26&0&31&120&16& $(\PP^1)^5$& $\of(1,1,1,1,1)$& $(\PP^1)^4$&+ \\ 
\end{longtable}
}

{\small
\begin{longtable}{ccccccccccc}
\caption{Rogue FK3}\label{tab:rogue}\\
\toprule
\multicolumn{1}{c}{ID}&\multicolumn{1}{c}{$\rho$}& \multicolumn{1}{c}{$h^{2,2}$} & \multicolumn{1}{c}{$h^{1,2}$} &\multicolumn{1}{c}{$h^0(-K)$}& \multicolumn{1}{c}{$K^4$}& \multicolumn{1}{c}{$-\chi(T)$}& \multicolumn{1}{c}{$G$}& \multicolumn{1}{c}{$\mF$}& \multicolumn{1}{c}{Bl}& \multicolumn{1}{c}{Rat} \\
\cmidrule(lr){1-1}\cmidrule(lr){2-2}\cmidrule(lr){3-3} \cmidrule(lr){4-4} \cmidrule(lr){5-5} \cmidrule(lr){6-6} \cmidrule(lr){7-7} \cmidrule(lr){8-8} \cmidrule(lr){9-9} \cmidrule(lr){10-10}\cmidrule(lr){11-11}
\endfirsthead
\multicolumn{5}{l}{\vspace{-0.25em}\scriptsize\emph{\tablename\ \thetable{} continued from previous page}}\\
\midrule
\endhead
\multicolumn{5}{r}{\scriptsize\emph{Continued on next page}}\\
\endfoot
\bottomrule
\endlastfoot
\evnrow \cref{27-99-2-1-25-25-1-3-2-6}&2&25&0&27&99&25&$\PP^2 \times \Gr(2,6)$&$\mU^{\vee}_{\Gr(2,6)}(0,1) \oplus \mQ_{\PP^2} \boxtimes \mU^{\vee}_{\Gr(2,6)}$& $Z^{\circ}$& +\\
\cref{33-129-2-1-23-23-1-4-2-7}&2&23&0&33&129&23&$\PP^3 \times \Gr(2,7)$&$\of(0,1)^ {\oplus 2} \oplus \of(1,1) \oplus \mQ_{\PP^3}\boxtimes \mU^{\vee}_{\Gr(2,7)}$& $Z^{\circ}$& ?\\
\evnrow \cref{31-120-3-1-2-22-20-1-6-1-8}&3&22&2&31&120&20&$ \PP^5 \times \PP^7$&$\mQ_{\PP^5}(0,1) \oplus \of(0,2) \oplus \of(2,0)^{\oplus 2}$& $Z^{\circ}$&+\\
\cref{25-90-3-5-1-22-23-1-5-1-7}&3&22&5&25&90&23&$\PP^4 \times \PP^6$&$\mQ_{\PP^4}(0,1) \oplus {} $& $Z^{\circ}, $& ? \\
    & &&&&&&&$\of(3,0) \oplus \of(0,2) $&  & \\
\end{longtable}
}


\begin{longtable}{ccccccccccc}
\caption{Blow ups of FK3}\label{tab:appendix}\\
\toprule
\multicolumn{1}{c}{ID}&\multicolumn{1}{c}{$\rho$}& \multicolumn{1}{c}{$h^{2,2}$} & \multicolumn{1}{c}{$h^{1,3}$} &\multicolumn{1}{c}{$h^0(-K)$}& \multicolumn{1}{c}{$K^4$}& \multicolumn{1}{c}{$-\chi(T)$}& \multicolumn{1}{c}{$G$}& \multicolumn{1}{c}{$\mF$}& \multicolumn{1}{c}{Bl}& \multicolumn{1}{c}{Rat} \\
\cmidrule(lr){1-1}\cmidrule(lr){2-2}\cmidrule(lr){3-3} \cmidrule(lr){4-4} \cmidrule(lr){5-5} \cmidrule(lr){6-6} \cmidrule(lr){7-7} \cmidrule(lr){8-8} \cmidrule(lr){9-9} \cmidrule(lr){10-10}\cmidrule(lr){11-11}
\endfirsthead
\multicolumn{5}{l}{\vspace{-0.25em}\scriptsize\emph{\tablename\ \thetable{} continued from previous page}}\\
\midrule
\endhead
\multicolumn{5}{r}{\scriptsize\emph{Continued on next page}}\\
\endfoot
\bottomrule
\endlastfoot
\evnrow \cref{12-27-2-7-79-58-1-2-1-6}& 2&79&7&12&27&58&$ \PP^1 \times \PP^5$&$\of(0,3) \oplus \of(1,2)$& $X_3$& * \\
\cref{16-45-2-3-51-44-1-3-1-6}&2&51&3&16&45&44&$\PP^2 \times \PP^5$&$\of(0,3) \oplus \of(1,1) \oplus \of(1,1)$&$X_3$&*\\
\evnrow \cref{17-51-2-2-41-36-1-3-1-6}&2&41&2&17&51&36&$\PP^2 \times \PP^5$& $\of(1,1) \oplus \mQ_{\PP^2}(0,2)$& $X_3$& ? \\
\cref{26-93-2-2-41-39-1-6-1-7}&2&41&2&26&93&39&$ \PP^5 \times \PP^6$& $\of(1,1) \oplus \mQ_{\PP^5}(0,1) \oplus \of(3,0)$&$X_3, Z^{\circ}$&?\\
\evnrow \cref{17-50-2-2-42-38-1-2-2-5}&2&42&2&17&50&38&$ \PP^1 \times \Gr(2,5)$& $\of(0,1) \oplus \of(0,2) \oplus \of(1,1)$& $X_{10}$& * \\
\end{longtable}

\TOCstop
              
\normalsize

%

\newcommand{\etalchar}[1]{$^{#1}$}

\TOCstart

\end{document}